\documentclass[11pt]{article}
\usepackage{amsmath,amsthm,amssymb}
\usepackage[usenames,dvipsnames]{xcolor}
\usepackage{enumerate}
\usepackage{graphicx}
\usepackage{cite}
\usepackage{comment}
\usepackage{oands}
\usepackage{tikz}
\usepackage{changepage}
\usepackage{bbm}
\usepackage{mathtools}
\usepackage[margin=1in]{geometry}
\usepackage[pagewise,mathlines]{lineno}
\usepackage{appendix}
\usepackage{stmaryrd}
\usepackage{multicol}
\usepackage{microtype}
\usepackage[colorinlistoftodos]{todonotes}

\usepackage[pdftitle={Regularity and confluence of geodesics for the supercritical Liouville quantum gravity metric},
  pdfauthor={Jian Ding and Ewain Gwynne},
colorlinks=true,linkcolor=NavyBlue,urlcolor=RoyalBlue,citecolor=PineGreen,bookmarks=true,bookmarksopen=true,bookmarksopenlevel=2,unicode=true,linktocpage]{hyperref}

\theoremstyle{plain}
\newtheorem{thm}{Theorem}[section]
\newtheorem{cor}[thm]{Corollary}
\newtheorem{lem}[thm]{Lemma}
\newtheorem{prop}[thm]{Proposition}

\newtheorem{notation}[thm]{Notation}

\def\@rst #1 #2other{#1}
\newcommand\MR[1]{\relax\ifhmode\unskip\spacefactor3000 \space\fi
  \MRhref{\expandafter\@rst #1 other}{#1}}
\newcommand{\MRhref}[2]{\href{http://www.ams.org/mathscinet-getitem?mr=#1}{MR#2}}

\theoremstyle{definition}
\newtheorem{defn}[thm]{Definition}
\newtheorem{remark}[thm]{Remark}

\numberwithin{equation}{section}

\newcommand{\dsb}{\begin{adjustwidth}{2.5em}{0pt}
\begin{footnotesize}}
\newcommand{\dse}{\end{footnotesize}
\end{adjustwidth}}

\newcommand{\ssb}{\begin{adjustwidth}{2.5em}{0pt}}
\newcommand{\sse}{\end{adjustwidth}}

\newcommand{\aryb}{\begin{eqnarray*}}
\newcommand{\arye}{\end{eqnarray*}}
\def\alb#1\ale{\begin{align*}#1\end{align*}}
\def\allb#1\alle{\begin{align}#1\end{align}}
\newcommand{\eqb}{\begin{equation}}
\newcommand{\eqe}{\end{equation}}
\newcommand{\eqbn}{\begin{equation*}}
\newcommand{\eqen}{\end{equation*}}

\newcommand{\BB}{\mathbbm}
\newcommand{\ol}{\overline}

\newcommand{\op}{\operatorname}

\newcommand{\frk}{\mathfrak}

\newcommand{\ep}{\varepsilon}
\newcommand{\rta}{\rightarrow}

\newcommand{\wt}{\widetilde}
\newcommand{\wh}{\widehat}
\newcommand{\mcl}{\mathcal}

\newcommand{\bdy}{\partial}
\newcommand{\rng}{\mathring}

\newcommand{\ccM}{{\mathbf{c}_{\mathrm M}}}

\let\originalleft\left
\let\originalright\right
\renewcommand{\left}{\mathopen{}\mathclose\bgroup\originalleft}
\renewcommand{\right}{\aftergroup\egroup\originalright}

\title{Regularity and confluence of geodesics for the supercritical Liouville quantum gravity metric}
 \date{ }
 \author{
\begin{tabular}{c} Jian Ding\footnote{dingjian@math.pku.edu.cn}\\[-5pt]\small Peking University \end{tabular}
\begin{tabular}{c} Ewain Gwynne\footnote{ewain@uchicago.edu}\\[-5pt]\small University of Chicago \end{tabular} 
}

\begin{document}

\maketitle

\begin{abstract} 
Let $h$ be the planar Gaussian free field and let $D_h$ be a supercritical Liouville quantum gravity (LQG) metric associated with $h$. Such metrics arise as subsequential scaling limits of supercritical Liouville first passage percolation (Ding-Gwynne, 2020) and correspond to values of the matter central charge $\mathbf{c}_{\mathrm M} \in (1,25)$. We show that a.s.\ the boundary of each complementary connected component of a $D_h$-metric ball is a Jordan curve and is compact and finite-dimensional with respect to $D_h$. This is in contrast to the \emph{whole} boundary of the $D_h$-metric ball, which is non-compact and infinite-dimensional with respect to $D_h$ (Pfeffer, 2021). Using our regularity results for boundaries of complementary connected components of $D_h$-metric balls, we extend the confluence of geodesics results of Gwynne-Miller (2019) to the case of supercritical Liouville quantum gravity. These results show that two $D_h$-geodesics with the same starting point and different target points coincide for a non-trivial initial time interval. 
\end{abstract}
 

\tableofcontents

\section{Introduction}
\label{sec-intro}

\subsection{Overview}
\label{sec-overview}

\emph{Liouville quantum gravity} (LQG) is a family of models of random surfaces originating in the physics literature in the 1980s~\cite{polyakov-qg1,david-conformal-gauge,dk-qg}. 
One way to define LQG surfaces is in terms of the \emph{matter central charge}, a parameter $\ccM \in (-\infty,25)$. 
Let $U\subset\BB C$ be open. For a Riemannian metric tensor $g$ on $U$, let $\Delta_g$ be the associated Laplace-Beltrami operator and let $\det\Delta_g$ denote its determinant. 
Heuristically speaking, an LQG surface parametrized by $U$ is the random two-dimensional Riemannian manifold $(U,g)$, where $g$ is sampled from the ``uniform measure on Riemannian metric tensors on $U$ weighted by $(\det\Delta_g)^{-\ccM/2}$". We refer to the case when $\ccM  < 1$ as the \emph{subcritical case} and the case when $\ccM \in (1,25)$ as the \emph{supercritical case}. 

The above definition of LQG is very far from rigorous, but it is nevertheless possible to define LQG surfaces rigorously. One way to do this is via the \emph{David-Distler-Kawai (DDK) ansatz}~\cite{david-conformal-gauge,dk-qg}, which says that, at least in the subcritical phase, the Riemannian metric tensor associated with an LQG surface can be expressed in terms of the exponential of a variant of the Gaussian free field (GFF) $h$ on $U$. We refer to~\cite{shef-gff,berestycki-lqg-notes,pw-gff-notes} for background on the GFF. The GFF $h$ is a random distribution, not a function, so its exponential is not well-defined. But, one can construct objects associated with the exponential of $h$ by replacing $h$ by a family of continuous functions $\{h_\ep\}_{\ep >0}$ which approximate $h$, then taking a limit as $\ep\rta 0$. 

In the subcritical and critical cases, i.e., when $\ccM \leq  1$, this approach has been used to construct the \emph{Liouville quantum gravity area measure} (i.e., the volume form) as a limit of regularized versions of $e^{\gamma h}$ integrated against Lebesgue measure~\cite{kahane,shef-kpz,rhodes-vargas-log-kpz,shef-deriv-mart,shef-renormalization}, where $\gamma \in (0,2]$ is related to the central charge by 
\eqb
\ccM = 25-6Q^2 ,\quad Q =\frac{2}{\gamma} + \frac{\gamma}{2} .
\eqe
Most mathematical works on LQG consider only the case when $\ccM \leq 1$ and use $\gamma$, rather than $\ccM$, as the parameter for the model. 

The focus of the present paper is the \emph{metric} (Riemannian distance function) associated with an LQG surface, which we can be defined for all $\ccM \in (-\infty,25)$. Let us explain the construction of this metric for the GFF on the whole plane.  
For $t  > 0$ and $z\in\BB C$, we define the heat kernel $p_t(z) := \frac{1}{2\pi t} e^{-|z|^2/2t}$ and we denote its convolution with the whole-plane GFF\footnote{
The whole-plane GFF $h$ is only defined modulo additive constant. Throughout the paper, we assume that the additive constant is chosen so that the average of $h$ over the unit circle is zero unless otherwise stated.} 
$h$ by
\eqb \label{eqn-gff-convolve}
h_\ep^*(z) := (h*p_{\ep^2/2})(z) = \int_{\BB C} h(w) p_{\ep^2/2} (z  - w) \, dw^2 ,\quad \forall z\in \BB C  
\eqe
where the integral is interpreted in the sense of distributional pairing. 

For a parameter $\xi > 0$, we define the $\ep$-\emph{Liouville first passage percolation} (LFPP) metric associated with $h$, with parameter $\xi$, by
\eqb \label{eqn-gff-lfpp}
D_{h}^\ep(z,w) := \inf_P \int_0^1 e^{\xi h_\ep^*(P(t))} |P'(t)| \,dt ,\quad \forall z,w\in\BB C 
\eqe
where the infimum is over all piecewise continuously differentiable paths $P : [0,1]\rta \BB C$ from $z$ to $w$. 

To extract a non-trivial limit of the metrics $D_h^\ep$, we need to re-normalize. We define our renormalizing factor by
\eqb \label{eqn-gff-constant}
\frk a_\ep := \text{median of} \: \inf\left\{ \int_0^1 e^{\xi h_\ep^*(P(t))} |P'(t)| \,dt  : \text{$P$ is a left-right crossing of $[0,1]^2$} \right\} ,
\eqe  
where a left-right crossing of $[0,1]^2$ is a piecewise continuously differentiable path $P : [0,1]\rta [0,1]^2$ joining the left and right boundaries of $[0,1]^2$. 
We emphasize that $\frk a_\ep$ is the median of a random variable (the inf of the lengths of the left-right crossings) so is deterministic.

It was shown in~\cite[Proposition 1.1]{dg-supercritical-lfpp} that for each $\xi > 0$, there exists $Q = Q(\xi) > 0$ such that 
\eqb  \label{eqn-Q-def}
\frk a_\ep = \ep^{1 - \xi Q  + o_\ep(1)} ,\quad \text{as} \quad \ep \rta 0 . 
\eqe 
Furthermore, $Q$ is a non-increasing function of $\xi$ and satisfies $\lim_{\xi\rta 0} Q(\xi) =\infty$ and $\lim_{\xi\rta\infty} Q(\xi) = 0$.

As explained in~\cite{dg-supercritical-lfpp} (see also~\cite{ghpr-central-charge}), the parameter $\xi$ is (heuristically) related to the matter central charge by
\eqb \label{eqn-Q-c}
\ccM = 25 - 6Q(\xi)^2 .
\eqe
The dependence of $Q$ on $\xi$, or equivalently the dependence of $\xi$ on $\ccM$, is not known explicitly except that $Q(1/\sqrt 6) = 5/\sqrt 6$, which corresponds to $\ccM = 0$~\cite{dg-lqg-dim}. Define
\eqb \label{eqn-xi_crit}
\xi_{\op{crit}} := \inf\{\xi > 0 : Q(\xi) = 2\} .
\eqe
We do not know $\xi_{\op{crit}}$ explicitly, but the bounds from~\cite[Theorem 2.3]{gp-lfpp-bounds} give the reasonably good approximation $\xi_{\op{crit}}  \in [0.4135 , 0.4189]$. 
By~\eqref{eqn-Q-c} and the properties of $Q(\xi)$ from~\cite[Proposition 1.1]{dg-supercritical-lfpp}, we have 
\eqb \label{eqn-xi-phases}
\ccM < 1 \quad \Leftrightarrow \quad \xi < \xi_{\op{crit}} \quad\text{and} \quad \ccM \in (1,25) \quad \Leftrightarrow \quad \xi > \xi_{\op{crit}} .
\eqe

In the subcritical case, it was shown in~\cite{dddf-lfpp} that for $\xi < \xi_{\op{crit}}$, the re-scaled LFPP metrics $\frk a_\ep^{-1} D_h^\ep$ admit non-trivial subsequential scaling limits w.r.t.\ the topology of uniform convergence on compact subsets of $\BB C\times \BB C$. Subsequently, it was shown in~\cite{gm-uniqueness} that the subsequential limit is unique and is characterized by a certain list of natural axioms. 
The limit $D_h$ of $\frk a_\ep^{-1} D_h^\ep$ is called the \emph{LQG metric} with parameter $\xi$. 

The LQG metric in the subcritical case induces the same topology as the Euclidean metric, but its geometric properties are very different. For example, the Hausdorff dimension of the metric space $(\BB C , D_h)$ is $\gamma/ \xi > 2$~\cite{gp-kpz}. Another important property of $D_h$ is \emph{confluence of geodesics}, which states that two $D_h$-geodesics (i.e., paths of minimal $D_h$-length) with the same starting point and different target points typically coincide for a non-trivial initial time interval. Note that this is not true for geodesics for a smooth Riemannian metric. Confluence of geodesics for the subcritical LQG metric was first established in~\cite{gm-confluence} and played a key role in the uniqueness proof in~\cite{gm-uniqueness}. See also~\cite{lqg-zero-one,gwynne-geodesic-network} for extensions of the confluence property for subcritical LQG, \cite{legall-geodesics} for an earlier proof of confluence of geodesics for the Brownian map (which is equivalent to LQG with $\ccM = 0$~\cite{lqg-tbm1,lqg-tbm2}), and~\cite{akm-geodesics,mq-strong-confluence,legall-geodesic-stars} for stronger confluence results in the Brownian map setting. 

In this paper, we will mainly be interested in the supercritical and critical cases, i.e., $\xi\geq \xi_{\op{crit}}$. It was shown in~\cite{dg-supercritical-lfpp} that for this range of parameter values, the re-scaled LFPP metrics $\frk a_\ep^{-1} D_h^\ep$ are tight with respect to the topology on lower semicontinuous functions on $\BB C\times\BB C$ introduced by Beer~\cite{beer-usc} (see Definition~\ref{def-lsc}). 
Later, after this paper appeared on the arXiv, it was shown in~\cite{dg-uniqueness} that the subsequential limit is unique. The proof in~\cite{dg-uniqueness} uses some of the results in this paper (in particular, those in Section~\ref{sec-hit-ball}), so throughout this paper we will work with subsequential limits.

If $D_h$ is a subsequential limit of LFPP for $\xi > \xi_{\op{crit}}$, then $D_h$ is a metric on $\BB C$ which is allowed to take on infinite values. This metric does not induce the Euclidean topology: rather, there is an uncountable, Euclidean-dense set of \emph{singular points} $z\in \BB C$ such that 
\eqb \label{eqn-singular-pt}
D_h(z,w) = \infty,\quad \forall w\in\BB C\setminus\{z\} .
\eqe
On the other hand, for two fixed points $z,w\in \BB C$, a.s.\ $D_h(z,w) < \infty$, and the restriction of $D_h$ to the complement of the set of singular points defines a complete metric~\cite{pfeffer-supercritical-lqg}.
Roughly speaking, singular points for $D_h$ correspond to $\alpha$-thick points of $h$ for $\alpha > Q$, i.e., points $z\in\BB C$ for which $h_\ep^*(z)$ behaves like $\alpha\log\ep^{-1}$ as $\ep\rta 0$~\cite{dg-supercritical-lfpp,pfeffer-supercritical-lqg}. 
It was shown in~\cite{dg-critical-lqg} that the metric $D_h$ induces the Euclidean topology on $\BB C$ for $\xi = \xi_{\op{crit}}$. In particular, there are no singular points in this case. 

Due to the existence of singular points, $D_h$-metric balls in the supercritical case are highly irregular objects. A $D_h$-ball has empty Euclidean interior (since the singular points are Euclidean dense). Moreover, the $D_h$-boundary of a $D_h$-metric ball is not $D_h$-compact and has infinite Hausdorff dimension w.r.t.\ $D_h$~\cite{pfeffer-supercritical-lqg} (see Theorem~\ref{thm-non-compact}). See Figure~\ref{fig-ball-sim} for a simulation of a supercritical LQG metric ball. 

In contrast, we will show that the boundary of a \emph{filled} $D_h$-metric ball (i.e., the union of the ball and the points which it disconnects from some specified target point) is a Jordan curve and is compact and finite-dimensional w.r.t.\ $D_h$ (Theorem~\ref{thm-outer-bdy}). 

Using our regularity results for outer boundaries of $D_h$-metric balls, we will then extend the confluence of geodesic results from~\cite{gm-confluence} to the critical and supercritical cases (Theorems~\ref{thm-clsce} and~\ref{thm-finite-geo0}). 
Unlike in the subcritical case~\cite{gm-uniqueness}, these confluence results are not needed for the proof of the uniqueness of the critical and supercritical LQG metrics in~\cite{dg-uniqueness}. However, they are of independent interest.

An important tool in our work is the paper~\cite{pfeffer-supercritical-lqg}, which shows that subsequential limits of supercritical LFPP satisfy a list of axioms similar to the axioms for a weak LQG metric from~\cite{lqg-metric-estimates} (see Definition~\ref{def-metric}), and establishes a number of estimates for any metric satisfying these axioms. 
All of the results in this paper are valid for any metric satisfying the axioms from~\cite{pfeffer-supercritical-lqg}.
\bigskip

\noindent \textbf{Acknowledgments.} We thank two anonymous referees for helpful comments on an earlier version of this article. We thank Jason Miller and Josh Pfeffer for helpful discussions. J.D.\ was partially supported by NSF grants DMS-1757479 and DMS-1953848. E.G.\ was partially supported by a Clay research fellowship.

\begin{figure}[t!]
 \begin{center}
\includegraphics[width=0.8\textwidth]{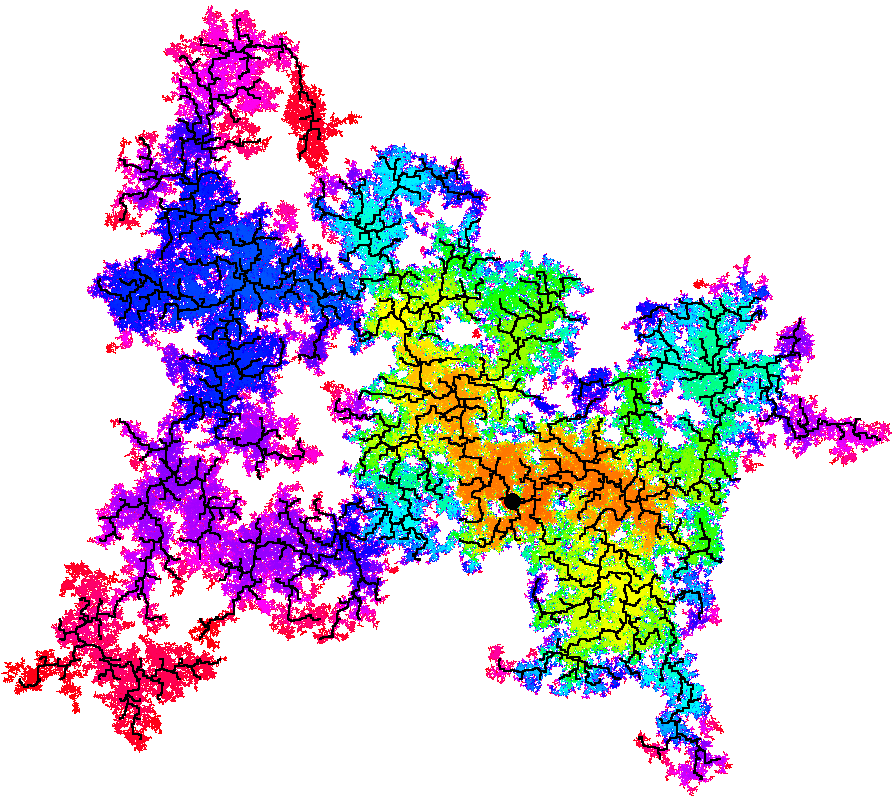}
\vspace{-0.01\textheight}
\caption{Simulation of an LFPP metric ball for $\xi = 1.6 > \xi_{\op{crit}}$. The colors indicate distance to the center point (marked with a black dot) and the black curves are geodesics from the center point to other points in the ball. 
These geodesics have a tree-like structure, which is consistent with our confluence of geodesics results. 
We also note that there are many ``holes" corresponding to complementary connected components of the ball. The boundary of each of these holes is of the form $\bdy\mcl B_s^{y,\bullet}$ for some $y\in\BB C$. 
The simulation was produced using LFPP w.r.t.\ a discrete GFF on a $1024 \times 1024$ subset of $\BB Z^2$.  It is believed that this variant of LFPP falls into the same universality class as the variant in~\eqref{eqn-gff-convolve}. The geodesics go from the center of the metric ball to points in the intersection of the metric ball with the grid $20\BB Z^2$. The code for the simulation was provided by J.\ Miller.
}\label{fig-ball-sim}
\end{center}
\vspace{-1em}
\end{figure}

\subsection{Ordinary and filled LQG metric balls}
\label{sec-results}

Throughout the paper, we let $h$ be a whole-plane GFF, we fix $\xi > 0$, and we let $D_h$ be a weak LQG metric associated with $h$ with parameter $\xi$. For now, the reader can think of $D_h$ as a subsequential limit of the re-scaled LFPP metrics $\frk a_\ep^{-1} D_h^\ep$, but we emphasize that all of our results also hold for any metric satisfying the axioms stated in Definition~\ref{def-metric} below. Also, most of our results are stated for $\xi > 0$ (not just $\xi \geq \xi_{\op{crit}}$), but many of the statements are either obvious or already proven elsewhere when $\xi \in (0,\xi_{\op{crit}})$. For $\xi > \xi_{\op{crit}}$, the metric $D_h$ does not induce the Euclidean topology. We therefore make the following notational convention.

\begin{notation} \label{notation-bdy}
Throughout this paper, topological concepts such as ``open", ``closed", ``boundary", etc., are always defined w.r.t.\ the Euclidean topology unless otherwise stated. Similarly, for a set $A\subset \BB C$, $\bdy A$ denotes its boundary w.r.t.\ the Euclidean topology and $\ol A$ denotes its closure w.r.t.\ the Euclidean topology. Moreover, $z_n \to z$ always refers to convergence with respect to the Euclidean topology, unless otherwise stated.
\end{notation}

For any set $A$, the boundary of $A$ with respect to $D_h$ is contained in the Euclidean boundary $\bdy A$: this is because the Euclidean metric is continuous with respect to $D_h$. 
The reverse inclusion does not necessarily hold. For example, a $D_h$-metric ball is Euclidean-closed (Lemma~\ref{lem-ball-closed}) and has empty Euclidean interior (since the set of singular points is dense), so the Euclidean boundary of such a ball is equal to the whole ball. On the other hand, the $D_h$-distance from each point in the $D_h$-boundary to the center point of the ball is equal to the radius of the ball. Hence, any $D_h$-metric ball with the same center point and a strictly smaller radius is disjoint from the $D_h$-boundary of the ball, so in particular the $D_h$-boundary of the ball is not equal to the whole ball.

We also briefly recall the definition of Hausdorff dimension. For $\Delta >0$, the $\Delta$-Hausdorff content of a metric space $(X,d)$ is
\eqbn
\inf\left\{\sum_{j=1}^\infty r_j^\Delta :\text{there is a covering of $X$ by $d$-metric balls with radii $\{r_j\}_{j\in\BB N}$}\right\}
\eqen
and the \emph{Hausdorff dimension} of $(X,d)$ is the infimum of the values of $\Delta$ for which the $\Delta$-Hausdorff content is zero. 

For $x\in\BB C$ and $s >0$, we write 
\eqb \label{eqn-ball-def}
\mcl B_s(x) := \left\{z\in\BB C : D_h(x,z) \leq s\right\} \quad \text{and} \quad \mcl B_s := \mcl B_s(0)
\eqe
for the closed $D_h$-metric ball of radius $s$. Recall that a \emph{singular point} for $D_h$ is a point which lies at infinite distance from every other point. A \emph{non-singular point} is a point which is not a singular point (i.e., a point which lies at finite distance from some other point). Using the fact that the singular points for $D_h$ are Euclidean-dense, Pfeffer~\cite[Proposition 1.14]{pfeffer-supercritical-lqg} established the following. 

\begin{thm}[\!\!\cite{pfeffer-supercritical-lqg}] \label{thm-non-compact} 
Assume that $\xi > \xi_{\op{crit}}$. Almost surely, for each non-singular point $x\in\BB C$ and each $s > t > 0$, the $D_h$-boundary (hence also the Euclidean boundary) of the $D_h$-metric ball $\mcl B_s(x)$ cannot be covered by finitely many $D_h$-metric balls of radius $t$. Furthermore, $\bdy \mcl B_s(x) = \mcl B_s(x)$ is not $D_h$-compact and has infinite Hausdorff dimension w.r.t.\ $D_h$.
\end{thm}
 
The reason why $\bdy\mcl B_s(x) = \mcl B_s(x)$ is that, as noted above, the fact that the set of singular points for $D_h$ is Euclidean dense implies that $\mcl B_s(x)$ has empty Euclidean interior.
Theorem~\ref{thm-non-compact} tells us that the boundaries of $D_h$-metric balls are in some sense highly irregular. One of the main contributions of this paper is to show that, in contrast, the boundaries of \emph{filled} $D_h$-metric balls are well-behaved.

\begin{defn}  \label{def-filled}
Let $x \in \BB{C}$ and $y \in \mathbb{C} \cup \{\infty\}$.  For $s \geq 0$, we define the \emph{filled $D_h$-metric ball centered at $x$ and targeted at $y$} with radius $s>0$ by
\[
\mcl B^{y,\bullet}_s(x) :=
\begin{cases} 
\text{the union of the closed metric ball $\mcl B_s(x)$} \\
\text{and the set of points which this}  &\qquad\mbox{for } s < D_h(x,y) \\
\text{metric ball disconnects from $y$} \\\\
\BB{C} &\qquad\mbox{for } s  \geq D_h(x,y)\\\\
\end{cases}
\]
We will most often work with filled metric balls centered at zero and filled metric balls targeted at infinity, so to lighten notation, we abbreviate
\eqb
\mcl B^\bullet_s(x) := \mcl B^{\infty,\bullet}_s(x), \qquad
\mcl B_s^{y,\bullet} := \mcl B_s^{y,\bullet}(0 ) \quad \text{and} \quad \mcl B_s^\bullet := \mcl B_s^\bullet(0).
\eqe
\end{defn}

We note that filled $D_h$-metric balls differ from ordinary $D_h$-metric balls since the complement of an ordinary $D_h$-metric ball is typically not connected (see Figure~\ref{fig-ball-sim}). In fact, a.s.\ each such complement has infinitely many connected components, see~\cite[Proposition 1.14]{pfeffer-supercritical-lqg}.
The following theorem summarizes our main results concerning the boundaries of filled $D_h$-metric balls. 

\begin{thm} \label{thm-outer-bdy}
Almost surely, for each non-singular point $x\in\BB C$, each $y\in\BB C \cup \{\infty\}$, and each $s \in (0,D_h(x,y))$, the filled metric ball boundary $\bdy\mcl B_s^{y,\bullet}(x)$ is a Jordan curve. Moreover, this boundary is $D_h$-compact and its Hausdorff dimension is bounded above by a finite constant which depends only on the law of $D_h$. 
\end{thm}

We emphasize that the statement of Theorem~\ref{thm-outer-bdy} holds a.s.\ for all choices of $x,y,s$ simultaneously. We show in Lemma~\ref{lem-outer-bdy-dist} below that the boundaries of $\mcl B_s^{y,\bullet}(x)$ with respect to the Euclidean metric and $D_h$ coincide, so Theorem~\ref{thm-outer-bdy} also applies to the $D_h$-boundary of $\mcl B_s^{y,\bullet}(x)$.

In the subcritical case $\xi < \xi_{\op{crit}}$, Theorem~\ref{thm-outer-bdy} follows from the fact that $D_h$ induces the Euclidean topology and the Hausdorff dimension of $(\BB C , D_h)$ is finite. See~\cite[Proposition 2.1]{tbm-characterization} for a proof that the boundary of a filled metric ball is a Jordan curve for any geodesic metric on $\BB C$ which induces the Euclidean topology. For $\xi \geq \xi_{\op{crit}}$, however, the proof of Theorem~\ref{thm-outer-bdy} requires non-trivial ideas. In particular, we first establish a general criterion for the boundary of an open domain to be a (not necessarily simple) curve (Proposition~\ref{prop-locally-connected}), which is a variant of the well-known fact that if the boundary of a simply connected domain in $\BB C$ is locally connected, then it is a curve (see, e.g.,~\cite[Section 2.2]{pom-book}). We then use some geometric estimates for supercritical LQG to check this criterion for the boundary of a filled supercritical LQG metric ball (see Section~\ref{sec-jordan-proof}), which shows that the filled metric ball boundary is a curve. Finally, we use some fairly straightforward topological considerations to show that the boundary of a filled metric ball does not have cut points, so is in fact a \emph{simple} curve (see Lemma~\ref{lem-metric-multiple}). The basic idea of our proof is similar to the proof of~\cite[Proposition 2.1]{tbm-characterization}, which proceeds by checking that the boundary of a filled metric ball for a geodesic metric which induces the Euclidean topology is locally connected and has no cut points. However, our proof is much more involved since our metric does not induce the Euclidean topology.

Theorem~\ref{thm-outer-bdy} implies that for $x\in\BB C$ and $s > 0$, the boundaries of the connected components of $\BB C\setminus \mcl B_s(x)$ have finite $D_h$-Hausdorff dimension. Since $\bdy\mcl B_s(x)$ itself has infinite $D_h$-Hausdorff dimension (Theorem~\ref{thm-non-compact}), we get that ``most" points of $\bdy\mcl B_s(x)$ do not lie on the boundary of any connected component of $\BB C\setminus \mcl B_s(x)$. Points of this type can arise as accumulation points of arbitrarily small connected components of $\BB C\setminus \mcl B_s(x)$. See~\cite[Theorem 1.14]{lqg-zero-one} for an analogous result in the subcritical case.
 
In fact, we will prove a slightly stronger Hausdorff dimension statement than the one in Theorem~\ref{thm-outer-bdy}. 
For $x \in \BB C$ and $y\in \BB C\cup \{\infty\}$, we define the \emph{metric net} 
\eqb
\mcl N_s^y(x) := \bigcup_{t \in [0,s]} \bdy\mcl B_t^{y,\bullet}(x) .
\eqe 

\begin{thm} \label{thm-net-dim}
There is a deterministic constant $\Delta \in (0,\infty)$ (depending on the law of $D_h$) such that a.s.\ for each non-singular point $x \in \BB C$, each $y\in \BB C \cup \{\infty\}$, and each $s > 0$ the Hausdorff dimension of $\mcl N_s^y(x)$ w.r.t.\ $D_h$ is at most $\Delta$.
\end{thm} 
  
The Hausdorff dimension of the metric net w.r.t.\ $D_h$ or w.r.t.\ the Euclidean metric is not known, even heuristically, for any $\xi > 0$, with one exception: when $\ccM = 0$ ($\xi = 1/\sqrt 6$), we expect that the Hausdorff dimension w.r.t.\ $D_h$ is 3 (this is consistent with scaling relations for quantum Loewner evolution in~\cite{lqg-tbm1,lqg-tbm2,lqg-tbm3}). 
It was shown in~\cite[Theorem 1.11]{gm-confluence} that in the subcritical case, the dimensions of the metric net w.r.t.\ the Euclidean and LQG metrics are each a.s.\ equal to deterministic constants. We expect that the same is true in the supercritical case. 


\subsection{Confluence of geodesics}
\label{sec-confluence-results}

Theorem~\ref{thm-outer-bdy} (and the estimates which go into its proof) can be used to extend the confluence of geodesic results from~\cite{gm-confluence} to the critical and supercritical cases. In particular, we obtain the following theorem for all $\xi > 0$. 

\begin{thm}[Confluence of geodesics at a point] \label{thm-clsce}
 Almost surely, for each radius $s > 0$ there exists a radius $t \in (0,s)$ such that any two $D_h$-geodesics from 0 to points outside of the filled $D_h$-metric ball $\mcl B_s^\bullet =\mcl B_s^\bullet(0)$ coincide on the time interval $[0,t]$.
\end{thm} 
 
Another form of confluence concerns geodesics across an annulus between two filled $D_h$-metric balls (Definition~\ref{def-filled}).  
Let us first note that every $D_h$-geodesic from $0$ to a point $z\in \bdy\mcl B_s^{ \bullet} $ stays in $\mcl B_s^\bullet$. For some points $z$ there might be many such $D_h$-geodesics, but there is always a distinguished $D_h$-geodesic from $0$ to $z$, called the \emph{leftmost geodesic}, which lies (weakly) to the left of every other $D_h$-geodesic from 0 to $z$ if we stand at $z$ and look outward from $ \mcl B_s^\bullet  $ (see Lemma~\ref{lem-leftmost-geodesic}). 

\begin{thm}[Confluence of geodesics across a metric annulus] \label{thm-finite-geo0}
Almost surely, for each $0 < t < s < \infty$ there is a finite set of $D_h$-geodesics from 0 to $\bdy\mcl B_t^\bullet $ such that every leftmost $D_h$-geodesic from 0 to a point of $\bdy\mcl B_s^\bullet $ coincides with one of these $D_h$-geodesics on the time interval $[0,t]$. In particular, there are a.s.\ only finitely many points of $\bdy\mcl B_t^\bullet $ which are hit by leftmost $D_h$-geodesics from 0 to points of $\bdy\mcl B_s^\bullet $. 
\end{thm}
  
Theorems~\ref{thm-clsce} and~\ref{thm-finite-geo0} are identical to~\cite[Theorems 1.3 and 1.4]{gm-confluence}, except that they apply for all $\xi > 0$ rather than just $\xi < \xi_{\op{crit}}$. The proofs of Theorems~\ref{thm-clsce} and~\ref{thm-finite-geo0} are given in Section~\ref{sec-confluence}. Many of the proofs in~\cite{gm-confluence} carry over verbatim to the critical and supercritical cases, but other parts require non-trivial adaptations. To avoid unnecessary repetition, we will only explain the parts of the proofs which are different in the critical and supercritical cases. 
 

\subsection{Outline}
\label{sec-outline}

The rest of this paper is structured as follows. In Section~\ref{sec-prelim} we review the axioms for a weak LQG metric from~\cite{pfeffer-supercritical-lqg}, then re-state some results from the existing literature (mostly from~\cite{pfeffer-supercritical-lqg}) which we will need for our proofs.
In Section~\ref{sec-bdy-estimates}, we prove a number of regularity estimates for the boundaries of filled $D_h$-metric balls, which enable us to prove Theorem~\ref{thm-net-dim} as well as all of Theorem~\ref{thm-outer-bdy} except for the statement that $\bdy\mcl B_s^{y,\bullet}$ is a Jordan curve. 
In Section~\ref{sec-curve}, we prove that $\bdy\mcl B_s^{y,\bullet}$ is a Jordan curve, which completes the proof of Theorem~\ref{thm-outer-bdy}. To do this, we first prove a general criterion for the boundary of a simply connected domain to be a curve, then check this criterion for $\bdy\mcl B_s^{y,\bullet}$ using the estimates from Section~\ref{sec-bdy-estimates}. 
In Section~\ref{sec-confluence} we explain how to prove our confluence of geodesic results, Theorems~\ref{thm-clsce} and~\ref{thm-finite-geo0}, by adapting the arguments of~\cite{gm-confluence} and applying the estimates of Section~\ref{sec-bdy-estimates}.

\section{Preliminaries}
\label{sec-prelim}

\subsection{Notational conventions}
\label{sec-notation}

\noindent
We write $\BB N = \{1,2,3,\dots\}$ and $\BB N_0 = \BB N \cup \{0\}$. 
\medskip

\noindent
For $a < b$, we define the discrete interval $[a,b]_{\BB Z}:= [a,b]\cap\BB Z$. 
\medskip

\noindent
If $f  :(0,\infty) \rta \BB R$ and $g : (0,\infty) \rta (0,\infty)$, we say that $f(\ep) = O_\ep(g(\ep))$ (resp.\ $f(\ep) = o_\ep(g(\ep))$) as $\ep\rta 0$ if $f(\ep)/g(\ep)$ remains bounded (resp.\ tends to zero) as $\ep\rta 0$. We similarly define $O(\cdot)$ and $o(\cdot)$ errors as a parameter goes to infinity. 
\medskip

\noindent
Let $\{E^\ep\}_{\ep>0}$ be a one-parameter family of events. We say that $E^\ep$ occurs with
\begin{itemize}
\item \emph{polynomially high probability} as $\ep\rta 0$ if there is a $p > 0$ (independent from $\ep$ and possibly from other parameters of interest) such that  $\BB P[E^\ep] \geq 1 - O_\ep(\ep^p)$. 
\item \emph{superpolynomially high probability} as $\ep\rta 0$ if $\BB P[E^\ep] \geq 1 - O_\ep(\ep^p)$ for every $p>0$.  
\end{itemize}
We similarly define events which occur with polynomially or superpolynomially high probability as a parameter tends to $\infty$. 
\medskip

\noindent
For $z\in\BB C$ and $r>0$, we write $B_r(z)$ for the open Euclidean ball of radius $r$ centered at $z$. More generally, for $X\subset \BB C$ we write $B_r(X) = \bigcup_{z\in X} B_r(z)$. We also define the open annulus
\eqb \label{eqn-annulus-def}
\BB A_{r_1,r_2}(z) := B_{r_2}(z) \setminus \ol{B_{r_1}(z)} ,\quad\forall 0 < r_r < r_2 < \infty .
\eqe 
\medskip
 
\noindent
For a region $A\subset\BB C$ with the topology of a Euclidean annulus, we write $D_h(\text{across $A$})$ for the $D_h$-distances between the inner and outer boundaries of $A$ and $D_h(\text{around $A$})$ for the infimum of the $D_h$-lengths of paths in $A$ which disconnect the inner and outer boundaries of $A$.  

\subsection{Weak LQG metrics}
\label{sec-weak-lqg-metric}

In this subsection, we will state the axiomatic definition of a weak LQG metric from~\cite{pfeffer-supercritical-lqg}. We first define the topology on the space of metrics that we will work with.

\begin{defn} \label{def-lsc} 
Let $X\subset \BB C$. 
A function $f : X \times X \rta \BB R \cup\{-\infty,+\infty\}$ is \emph{lower semicontinuous} if whenever $(z_n,w_n) \in X\times X$ with $(z_n,w_n) \rta (z,w)$, we have $f(z,w) \leq \liminf_{n\rta\infty} f(z_n,w_n)$. 
The \emph{topology on lower semicontinuous functions} is the topology whereby a sequence of such functions $\{f_n\}_{n\in\BB N}$ converges to another such function $f$ if and only if
\begin{enumerate}[(i)]
\item Whenever $(z_n,w_n) \in X\times X$ with $(z_n,w_n) \rta (z,w)$, we have $f(z,w) \leq \liminf_{n\rta\infty} f_n(z_n,w_n)$.
\item For each $(z,w)\in X\times X$, there exists a sequence $(z_n,w_n) \rta (z,w)$ such that $f_n(z_n,w_n) \rta f(z,w)$. 
\end{enumerate}
\end{defn}

It follows from~\cite[Lemma 1.5]{beer-usc} that the topology of Definition~\ref{def-lsc} is metrizable (see~\cite[Section 1.2]{dg-supercritical-lfpp}). 
Furthermore,~\cite[Theorem 1(a)]{beer-usc} shows that this metric can be taken to be separable. 

\begin{defn} \label{def-metric-properties}
Let $(X,d)$ be a metric space, with $d$ allowed to take on infinite values. 
\begin{itemize}
\item
For a curve $P : [a,b] \rta X$, the \emph{$d$-length} of $P$ is defined by 
\eqbn
\op{len}(P;d) :=  \sup_{T} \sum_{i=1}^{\# T} d(P(t_i) , P(t_{i-1})) 
\eqen
where the supremum is over all partitions $T : a= t_0 < \dots < t_{\# T} = b$ of $[a,b]$. Note that the $d$-length of a curve may be infinite.
\item
We say that $(X,d)$ is a \emph{length space} if for each $x,y\in X$ and each $\ep > 0$, there exists a curve of $d$-length at most $d(x,y) + \ep$ from $x$ to $y$. 
A curve from $x$ to $y$ of $d$-length \emph{exactly} $d(x,y)$ is called a \emph{geodesic}. 
\item
For $Y\subset X$, the \emph{internal metric of $d$ on $Y$} is defined by
\eqb \label{eqn-internal-def}
d(x,y ; Y)  := \inf_{P \subset Y} \op{len}\left(P ; d \right) ,\quad \forall x,y\in Y 
\eqe 
where the infimum is over all paths $P$ in $Y$ from $x$ to $y$. 
Note that $d(\cdot,\cdot ; Y)$ is a metric on $Y$, except that it is allowed to take infinite values.  
\item
If $X \subset \BB C$, we say that $d$ is a \emph{lower semicontinuous metric} if the function $(x,y) \rta d(x,y)$ is lower semicontinuous w.r.t.\ the Euclidean topology.  
We equip the set of lower semicontinuous metrics on $X$ with the topology on lower semicontinuous functions on $X \times X$, as in Definition~\ref{def-lsc}, and the associated Borel $\sigma$-algebra.
\end{itemize}
\end{defn}

The following is a re-statement of~\cite[Definition 1.6]{pfeffer-supercritical-lqg}.

\begin{defn}[Weak LQG metric]
\label{def-metric}
Let $\mcl D'$ be the space of distributions (generalized functions) on $\BB C$, equipped with the usual weak topology.   
For $\xi > 0$, a \emph{weak LQG metric with parameter $\xi$} is a measurable functions $h\mapsto D_h$ from $\mcl D'$ to the space of lower semicontinuous metrics on $\BB C$ with the following properties. Let $h$ be a \emph{GFF plus a continuous function} on $\BB C$: i.e., $h$ is a random distribution on $\BB C$ which can be coupled with a random continuous function $f$ in such a way that $h-f$ has the law of the whole-plane GFF. Then the associated metric $D_h$ satisfies the following axioms. 
\begin{enumerate}[I.]
\item \textbf{Length space.} Almost surely, $(\BB C,D_h)$ is a length space. \label{item-metric-length} 
\item \textbf{Locality.} Let $U\subset\BB C$ be a deterministic open set. 
The $D_h$-internal metric $D_h(\cdot,\cdot ; U)$ is a.s.\ given by a measurable function of $h|_U$.  \label{item-metric-local}
\item \textbf{Weyl scaling.} For a continuous function $f : \BB C \rta \BB R$, define
\eqb \label{eqn-metric-f}
(e^{\xi f} \cdot D_h) (z,w) := \inf_{P : z\rta w} \int_0^{\op{len}(P ; D_h)} e^{\xi f(P(t))} \,dt , \quad \forall z,w\in \BB C ,
\eqe 
where the infimum is over all $D_h$-continuous paths from $z$ to $w$ in $\BB C$ parametrized by $D_h$-length.
Then a.s.\ $ e^{\xi f} \cdot D_h = D_{h+f}$ for every continuous function $f: \BB C \rta \BB R$. \label{item-metric-f}
\item \textbf{Translation invariance.} For each deterministic point $z \in \BB C$, a.s.\ $D_{h(\cdot + z)} = D_h(\cdot+ z , \cdot+z)$.  \label{item-metric-translate}
\item \textbf{Tightness across scales.} Suppose that $h$ is a whole-plane GFF and let $\{h_r(z)\}_{r > 0, z\in\BB C}$ be its circle average process. 
There are constants $\{\frk c_r\}_{r>0}$ such that the following is true.
Let $A\subset \BB C$ be a deterministic Euclidean annulus.
In the notation defined at the end of Section~\ref{sec-notation}, the random variables
\eqbn
\frk c_r^{-1} e^{-\xi h_r(0)} D_h\left( \text{across $r A$} \right) \quad \text{and} \quad
\frk c_r^{-1} e^{-\xi h_r(0)} D_h\left( \text{around $r A$} \right)
\eqen
and the reciprocals of these random variables for $r>0$ are tight. Finally, there exists  
$\Lambda > 1$ such that for each $\delta \in (0,1)$,  \label{item-metric-coord} 
\eqb \label{eqn-scaling-constant}
\Lambda^{-1} \delta^\Lambda \leq \frac{\frk c_{\delta r}}{\frk c_r} \leq \Lambda \delta^{-\Lambda} ,\quad\forall r  > 0.
\eqe
\end{enumerate}
\end{defn}

The axioms of Definition~\ref{def-metric} are the same as the axioms which define a weak LQG metric in~\cite[Section 1.2]{lqg-metric-estimates}, with two exceptions: one works with lower semicontinuous metrics instead of continuous metrics, and the tightness across scales axiom (Axiom~\ref{item-metric-coord}) is formulated differently: we require tightness for re-scaled distances around and across Euclidean annuli, rather than requiring tightness of the re-scaled metrics themselves.

It was shown in~\cite{pfeffer-supercritical-lqg} that if $h$ is a GFF plus a continuous function and $D$ is a weak LQG metric, then a.s.\ the Euclidean metric is $D_h$-continuous (see Proposition~\ref{prop-holder} below for a quantitative version of this). In particular, a.s.\ every $D_h$-continuous path (e.g., a $D_h$-geodesic) is also Euclidean continuous. 

Axiom~\ref{item-metric-coord} allows us to get bounds for $D_h$-distances which are uniform across different Euclidean scales.
This axiom serves as a substitute for exact scale invariance (i.e., the LQG coordinate change formula), which is difficult to prove for subsequential limits of LFPP before we know that the subsequential limit is unique. See~\cite{lqg-metric-estimates,gm-uniqueness,pfeffer-supercritical-lqg} for further discussion of this point. 

The following theorem is proven as~\cite[Theorem 1.7]{pfeffer-supercritical-lqg}, building on the tightness result from~\cite{dg-supercritical-lfpp}.

\begin{thm}[\!\!\cite{pfeffer-supercritical-lqg}] \label{thm-lfpp-axioms}
Let $\xi > 0$. For every sequence of $\ep$'s tending to zero, there is a weak LQG metric $D$ with parameter $\xi$ and a subsequence $\{\ep_n\}_{n\in\BB N}$ for which the following is true. Let $h$ be a whole-plane GFF, or more generally a whole-plane GFF plus a bounded continuous function. Then the re-scaled LFPP metrics $\frk a_{\ep_n}^{-1} D_h^{\ep_n}$, as defined in~\eqref{eqn-gff-lfpp} and~\eqref{eqn-gff-constant}, converge in probability to $D_h$ w.r.t.\ the metric on lower semicontinuous functions on $\BB C\times \BB C$. 
\end{thm}

Theorem~\ref{thm-lfpp-axioms} implies in particular that for each $\xi > 0$, there exists a weak LQG metric with parameter $\xi$.

\begin{remark}
It was shown in~\cite{dg-uniqueness}, subsequently to this paper, that the axioms in Definition~\ref{def-metric} uniquely characterize $D_h$, up to multiplication by a deterministic positive constant. This implies that one has actual convergence (not just subsequential convergence) in Theorem~\ref{thm-lfpp-axioms} and that Axiom~\ref{item-metric-coord} can be improved to the LQG coordinate change formula for spatial scaling. Some of the results of this paper (in particular, those in Section~\ref{sec-hit-ball}) are used in~\cite{dg-uniqueness}. 
\end{remark}

\subsection{Results from prior work}
\label{sec-prior}

Throughout the rest of the paper, we fix $\xi >0$ and a weak LQG metric $D : h\mapsto D_h$ with parameter $\xi$. We will not make the dependence on the parameter $\xi$ or the particular choice of metric $D$ explicit in our estimates. 
We also let $h$ be a whole-plane GFF and we let $\{h_r(z) : r > 0, z\in\BB C\}$ be its circle average process (as in Axiom~\ref{item-metric-coord}). 
 
Many of the quantitative estimates in this paper involve a parameter $\BB r > 0$, which represents the ``Euclidean scale". 
The estimates are required to be uniform in the choice of $\BB r$. 
The reason for including $\BB r$ is the same as in other papers concerning weak LQG metrics, such as~\cite{lqg-metric-estimates,gm-confluence,gm-uniqueness,pfeffer-supercritical-lqg}: we only have tightness across scales (Axiom~\ref{item-metric-coord}), rather than exact scale invariance, so it is not possible to directly transfer estimates from one Euclidean scale to another. 

In this subsection, we state some previously known results for the GFF and the LQG metric (mostly from~\cite{pfeffer-supercritical-lqg}) which we will cite regularly. 
We start with the fact that $D_h$-geodesics exist~\cite[Proposition 1.12]{pfeffer-supercritical-lqg}, which is not immediate from the axioms since Axiom~\ref{item-metric-length} only shows that $D_h(z,w)$ is the infimum of the $D_h$-lengths of paths joining $z$ and $w$, not that a length-minimizing path exists.
 
\begin{lem}[\!\!\cite{pfeffer-supercritical-lqg}] \label{lem-geodesic-cont}
Almost surely, for any two non-singular points $z,w\in\BB C$, there exists an LQG geodesic $P$ joining $z$ and $w$.  
\end{lem}

We will frequently use without comment the following fact, which implies in particular that every $D_h$-bounded set is Euclidean bounded. See~\cite[Lemma 3.12]{pfeffer-supercritical-lqg} for a proof. 

\begin{lem}[\!\!\cite{pfeffer-supercritical-lqg}] \label{lem-bounded}
Almost surely, for every Euclidean-compact set $K\subset \BB C$, 
\eqbn
\lim_{R\rta\infty} D_h(K,\bdy B_R(0)) = \infty . 
\eqen
\end{lem}

It was shown in~\cite[Lemma 3.1]{pfeffer-supercritical-lqg} that one has the following stronger version of Axiom~\ref{item-metric-coord}. 

\begin{lem}[\!\!\cite{pfeffer-supercritical-lqg}] \label{lem-set-tightness}
Let $U\subset \BB C$ be open and let $K_1,K_2\subset U$ be two disjoint, deterministic compact sets (allowed to be singletons). The re-scaled internal distances $\frk c_r^{-1} e^{-\xi h_r(0)} D_h(r K_1,r K_2; r U)$ and their reciporicals are tight. 
\end{lem}

The following proposition, which is~\cite[Proposition 1.8]{pfeffer-supercritical-lqg}, is a more quantitative version of Lemma~\ref{lem-set-tightness} in the case when $K_1,K_2$ are connected and are not singletons. It will be our most important estimate for $D_h$-distances. 

\begin{prop}[\!\!\cite{pfeffer-supercritical-lqg}] \label{prop-two-set-dist}
Let $U \subset \BB C$ be an open set (possibly all of $\BB C$) and let $K_1,K_2\subset U$ be connected, disjoint compact sets which are not singletons. 
Also let $\{\frk c_r\}_{r >0}$ be the scaling constants from Axiom~\ref{item-metric-coord}. 
For each $\BB r  >0$, it holds with superpolynomially high probability as $A\rta \infty$, at a rate which is uniform in the choice of $\BB r$, that 
\eqb \label{eqn-two-set-dist}
 A^{-1}\frk c_{\BB r} e^{\xi h_{\BB r}(0)} \leq D_h(\BB r K_1,\BB r K_2 ; \BB r U) \leq A \frk c_{\BB r} e^{\xi h_{\BB r}(0)} .  
\eqe
\end{prop}

Recall that notation for $D_h$-distance across and around Euclidean annuli from Section~\ref{sec-notation}. 
We will most frequently use Proposition~\ref{prop-two-set-dist} to lower-bound $D_h(\text{across $\BB A_{a\BB r , b\BB r}(z)$})$ and upper-bound $D_h(\text{around $\BB A_{a\BB r , b\BB r}(z)$})$ where $b >a > 0$ are fixed. To do this, we first note that due to Axiom~\ref{item-metric-translate} we can assume without loss of generality that $z =0$. To lower-bound $D_h(\text{across $\BB A_{a\BB r , b\BB r}(z)$})$ we apply Proposition~\ref{prop-two-set-dist} with $K_1 = \bdy B_a(0)$, $K_2 = \bdy B_b(0)$, and $U=\BB C$. 
To upper-bound $D_h(\text{around $\BB A_{a\BB r , b\BB r}(z)$})$, we apply Proposition~\ref{prop-two-set-dist} twice, with the sets $K_1,K_2,U$ and $K_1',K_2',U'$ chosen so that the union of any path from $K_1$ to $K_2$ in $U$ and any path from $K_1'$ to $K_2'$ in $U'$ is contained in $\BB A_{a,b}(0)$ and disconnects the inner and outer boundaries of $\BB A_{a,b}(0)$. 

Axiom~\ref{item-metric-coord} only gives polynomial upper and lower bounds for the ratios of the scaling constants $\frk c_r$. The following proposition, which is~\cite[Proposition 1.9]{pfeffer-supercritical-lqg}, gives much more precise bounds for these scaling constants and relates them to LFPP. 

\begin{prop}[\!\!\cite{pfeffer-supercritical-lqg}] \label{prop-scaling-constants}
With $Q$ as in~\eqref{eqn-Q-def}, the scaling constants from Axiom~\ref{item-metric-coord} satisfy $\frk c_r = r^{\xi Q + o_r(1)}$ as $r\rta 0$ or $r\rta\infty$.
\end{prop}

We also have a H\"older continuity condition for the Euclidean metric w.r.t.\ $D_h$. See~\cite[Proposition 3.8]{pfeffer-supercritical-lqg}. 

\begin{prop}[\!\!\cite{pfeffer-supercritical-lqg}]  \label{prop-holder}
Let $\chi  \in (0, ( \xi(Q+2))^{-1}) $ and let $U\subset\BB C$ be a Euclidean-bounded open set. For each $\BB r  >0$, it holds with polynomially high probability as $\ep\rta 0$, at a rate which is uniform in $\BB r$, that
\eqb
|z-w| \leq D_h(z,w)^\chi,\quad  \forall z,w\in \BB r U   \quad \text{with} \quad |z-w| \leq \ep \BB r. 
\eqe 
In particular, the identity mapping from $(\BB C , D_h)$ to $\BB C$, equipped with the Euclidean metric, is $\chi$-H\"older continuous when restricted to any Euclidean-compact set. 
\end{prop}
  
We note that in contrast to the subcritical case (see~\cite[Theorem 1.7]{gm-confluence}), the H\"older continuity in Proposition~\ref{prop-holder} only goes in one direction. 

Finally, we state an estimate which is a consequence of the fact that the restrictions of the GFF $h$ to disjoint concentric annuli are nearly independent. See~\cite[Lemma 3.1]{local-metrics} for a proof of a slightly more general result.

\begin{lem}[\!\!\cite{local-metrics}] \label{lem-annulus-iterate}
Fix $0 < \mu_1<\mu_2 < 1$. Let $\{r_k\}_{k\in\BB N}$ be a decreasing sequence of positive real numbers such that $r_{k+1} / r_k \leq \mu_1$ for each $k\in\BB N$ and let $\{E_{r_k} \}_{k\in\BB N}$ be events such that $E_{r_k} \in \sigma\left( (h-h_{r_k}(0)) |_{\BB A_{\mu_1 r_k , \mu_2 r_k}(0)  } \right)$ for each $k\in\BB N$ (here we use the notation for Euclidean annuli from Section~\ref{sec-notation}). 
For $K\in\BB N$, let $N(K)$ be the number of $k\in [1,K]_{\BB Z}$ for which $E_{r_k}$ occurs.  
For each $a > 0$ and each $b\in (0,1)$, there exists $p = p(a,b,\mu_1,\mu_2) \in (0,1)$ and $c = c(a,b,\mu_1,\mu_2) > 0$ such that if  
\eqb \label{eqn-annulus-iterate-prob}
\BB P\left[ E_{r_k}  \right] \geq p , \quad \forall k\in\BB N  ,
\eqe 
then 
\eqb \label{eqn-annulus-iterate}
\BB P\left[ N(K)  < b K\right] \leq c e^{-a K} ,\quad\forall K \in \BB N. 
\eqe  
\end{lem}

\section{Estimates for the outer boundary of an LQG metric ball}
\label{sec-bdy-estimates}

We continue to assume that $\xi > 0$, $h$ is a whole-plane GFF, and $D_h$ is a weak LQG metric with parameter $\xi$.
In this section, we will prove a variety of estimates for $D_h$-distance which will eventually lead to proofs of Theorem~\ref{thm-net-dim} and the compactness and finite-dimensionality parts of Theorem~\ref{thm-outer-bdy}. We start out in Section~\ref{sec-bdy-prelim} by proving some basic facts about $D_h$ which are relatively straightforward consequences of existing results, e.g., the fact that $D_h$-metric balls are Euclidean closed and every filled $D_h$-metric ball contains a Euclidean ball with the same center point. 
In Section~\ref{sec-hit-ball}, we will prove a technical lemma which will be a key tool in our proofs: basically, it says that points on the boundary of a filled $D_h$-metric ball can be surrounded by paths with small $D_h$-lengths (Lemma~\ref{lem-hit-ball}). Using this lemma, in Section~\ref{sec-bdy-dist} we will prove a lower bound for the Euclidean distance between the boundaries of two filled metric balls with the same center point. Finally, in Section~\ref{sec-net-dim} we will prove Theorem~\ref{thm-net-dim} and part of Theorem~\ref{thm-outer-bdy}.

\subsection{Basic facts about the LQG metric}
\label{sec-bdy-prelim}

Before proving our main results for LQG metric ball boundaries, we will record some facts about $D_h$ which are easy consequences of the axioms from Definition~\ref{def-metric} and the estimates from Section~\ref{sec-prior}. For our first statement, we recall that $\bdy$ always denotes the boundary w.r.t.\ the \emph{Euclidean} topology.

\begin{lem} \label{lem-ball-closed}
Almost surely, for each $x \in \BB C$, each $y\in \BB C\cup \{\infty\}$, and each $s \in (0,D_h(x,y))$, the ordinary metric ball $\mcl B_s(x)$ and the filled metric ball $\mcl B_s^{y,\bullet}(x)$ are both Euclidean-closed and $\bdy\mcl B_s^{y,\bullet}(x) \subset  \mcl B_s(x)$.  
\end{lem}
\begin{proof}
The function $z\mapsto D_h(x,z)$ is lower semicontinuous, so if $z_n$ is a sequence of points in $\mcl B_s(x)$ with $|z_n-z| \rta 0$, then $D_h(x,z) \leq \liminf_{n\rta\infty} D_h(x,z_n) \leq s $, so $z\in \mcl B_s(x)$. 
Hence $\mcl B_s(x)$ is Euclidean-closed. 
Consequently, each connected component of $\BB C\setminus \mcl B_s (x)$ is Euclidean-open.  In particular, the connected component of $\BB C\setminus \mcl B_s(x)$ containing $y$, namely $\BB C\setminus  \mcl B_s^{y,\bullet}(x)$, is Euclidean-open, so $\mcl B_s^{y,\bullet}(x)$ is Euclidean-closed. 
Since $\mcl B_s(x)$ is Euclidean-closed, it contains the boundary of each of its complementary connected components. In particular,  $\bdy\mcl B_s^{y,\bullet}(x) \subset   \mcl B_s(x)$.  
\end{proof}

Our next several lemmas are based on the following straightforward consequence of Lemma~\ref{lem-annulus-iterate}, see~\cite[Proposition 1.13]{pfeffer-supercritical-lqg} for a proof.

\begin{lem}[\!\!\cite{pfeffer-supercritical-lqg}] \label{lem-separating-annuli}
Almost surely, for each non-singular point $z\in\BB C$ there is a sequence of disjoint $D_h$-continuous loops $\{\pi_n\}_{n\in\BB N}$, each of which separates a neighborhood of $z$ from $\infty$, 
such that the Euclidean radius of $\pi_n$, the $D_h$-length of $\pi_n$, and the $D_h$-distance from $z$ to $\pi_n$ each tend to zero as $n\rta\infty$.
\end{lem}

Since the set of singular points is a.s.\ Euclidean-dense, a.s.\ every $D_h$-metric ball has empty Euclidean interior. In contrast, the following lemma tells us that a filled $D_h$-metric ball a.s.\ contains a Euclidean ball with the same center. 

\begin{lem} \label{lem-ball-contain-all}
Almost surely, for each non-singular point $x\in\BB C$, each $y\in \BB C\cup\{\infty\}$, and each $s\in (0,D_h(x,y))$, the filled $D_h$-metric ball $\mcl B_s^{y,\bullet}(x)$ contains a Euclidean ball centered at $x$ with positive radius. 
\end{lem}
\begin{proof}
Let $\{\pi_n\}_{n\in\BB N}$ be a sequence of loops surrounding $x$ as in Lemma~\ref{lem-separating-annuli}.  
Let $P$ be a $D_h$-geodesic from $x$ to $y$. 
The Euclidean radii and the $D_h$-lengths of the $\pi_n$'s shrink to zero as $n\rta\infty$ and $P$ is Euclidean continuous.
Hence a.s.\ for each sufficiently large $n\in\BB N$, the loop $\pi_n$ disconnects $x$ from $y$, the $D_h$-length of $\pi_n$ is less than $s / 2$, and $P$ hits $\pi_n$ before time $s/2$. This shows that $\pi_n$ is contained in $\mcl B_s(x)$, so $\pi_n \subset \mcl B_s^{y,\bullet}(x)$. Since $\pi_n$ disconnects a Euclidean ball of positive radius centered at $x$ from $y$, this gives the lemma statement. 
\end{proof} 

For our next lemma, we recall that $\bdy$ always denotes the boundary w.r.t.\ the \emph{Euclidean} topology. 

\begin{lem} \label{lem-outer-bdy-dist}
Almost surely, for each non-singular point $x \in \BB C$, each $y\in \BB C\cup \{\infty\}$, and each $s \in (0,D_h(x,y))$,  
\eqb \label{eqn-outer-bdy-dist}
D_h(x,z) = s ,\quad \forall z \in \bdy\mcl B_s^{y,\bullet}(x) .
\eqe 
Furthermore, the Euclidean boundary $\bdy\mcl B_s^{y,\bullet}(x)$ is equal to the $D_h$-boundary of $\mcl B_s^{y,\bullet}(x)$.
\end{lem}
\begin{proof}
By Lemma~\ref{lem-ball-closed}, a.s.\ for each $x,y,s$ as in the lemma statement we have $\bdy\mcl B_s^{y,\bullet}(x) \subset \mcl B_s(x)$, so $D_h(x,z) \leq s$ for each $z\in  \bdy\mcl B_s^{y,\bullet}(x)$. We need to prove the reverse inequality. 
To this end, we fix $x,y,s$ as in the lemma statement. All statements are required to hold for all choices of $x,y,s$ simultaneously.

Let $z\in\bdy\mcl B_s(x)$. Then $D_h(x,z) \leq s < \infty$ so $z$ is not a singular point. Let $\{\pi_n\}_{n\in\BB N}$ be a sequence of disjoint $D_h$-continuous loops surrounding $z$ as in Lemma~\ref{lem-separating-annuli}. 
Since $D_h(x,w) \leq s < D_h(x,y)$ for each $w\in\bdy\mcl B_s^{y,\bullet}(x)$ and $\bdy\mcl B_s^{y,\bullet}(x)$ is Euclidean-closed, $\bdy\mcl B_s^{y,\bullet}(x)$ lies at positive Euclidean distance from $y$. 
The Euclidean radius of $\pi_n$ tends to zero as $n\rta\infty$ and each $\pi_n$ disconnects a neighborhood of $z$ from $\infty$. Hence for each large enough $n \in \BB N$, $y$ lies in the unbounded complementary connected component of $\pi_n$, and hence $\pi_n$ disconnects a neighborhood of $z$ from $y$.

If $D_h(x,z) < s$, then since the $D_h$-length of $\pi_n$ and the $D_h$-distance from $z$ to $\pi_n$ both tend to zero as $n\rta\infty$, the triangle inequality shows that $\pi_n\subset\mcl B_s(x)$ for each large enough $n\in\BB N$. But, $\pi_n$ disconnects a neighborhood of $z$ from $y$ for each large enough $n$, so if $D_h(x,z) < s$ then $z$ must be in the interior of $\mcl B_s^{y,\bullet}(x)$, not in $\bdy\mcl B_s^{y,\bullet}(x)$. 
We thus obtain~\eqref{eqn-outer-bdy-dist}. 

Since $\bdy\mcl B_s^{y,\bullet}(x) \subset \mcl B_s(x) \subset \mcl B_s^{y,\bullet}(x)$ and the $D_h$-boundary of any set is contained in its Euclidean boundary, to get the last statement of the lemma, we need to show that each point $z\in \bdy \mcl B_s^{y,\bullet}(x)$ is a $D_h$-accumulation point of $\BB C\setminus \mcl B_s^{y,\bullet}(x)$. 
Since the loop $\pi_n$ disconnects $z$ from $y$ for each large enough $n$, it follows that $\pi_n$ disconnects $\BB C\setminus \mcl B_s^{y,\bullet}(x)$ into at least two connected components for each large enough $n$. Since $\BB C\setminus \mcl B_s^{y,\bullet}(x)$ is connected, it follows that $\pi_n$ contains a point $z_n \in \BB C\setminus \mcl B_s^{y,\bullet}(x)$ for each large enough $n$. Since the $D_h$-distance from $z$ to $\pi_n$ and the $D_h$-length of $\pi_n$ each tend to zero as $n\rta\infty$, we infer that $z$ is a $D_h$-accumulation point of $\BB C\setminus \mcl B_s^{y,\bullet}(x)$, as required.
\end{proof}

Finally, we record a more quantitative version of Lemma~\ref{lem-ball-contain-all} which applies when the center point of the filled metric ball is fixed. 
In the lemma statement and in several places later in the paper, we will use the notation
\eqb \label{eqn-tau_r-def}
\tau_r = D_h(0,\bdy B_r(0)) = \inf\{t > 0 : \mcl B_t \not\subset B_r(0)\} , \quad\forall r > 0 .
\eqe

\begin{lem} \label{lem-ball-contain}
Let $\BB r > 0$ and let $\tau_{\BB r} $ be as in~\eqref{eqn-tau_r-def}.
It holds with polynomially high probability as $\ep\rta 0$, uniformly over the choice of $\BB r$, that $B_{\ep \BB r}(0) \subset \mcl B_{\tau_{\BB r}}^\bullet$.
\end{lem}
\begin{proof}
Let $\zeta  \in (0,\xi Q / 100)$ be a small exponent. By Proposition~\ref{prop-two-set-dist}, it holds with superpolynomially high probability as $\ep\rta 0$, uniformly in $\BB r$, that 
\eqb
D_h\left( \text{across $\BB A_{\BB r/2,\BB r}(0)$} \right) \geq \ep^\zeta \frk c_{\BB r} e^{\xi h_{\BB r}(0)} \quad \text{and} \quad
D_h\left( \text{around $\BB A_{\ep \BB r, 2\ep \BB r}(0)$} \right) \leq \ep^{-\zeta} \frk c_{\ep \BB r} e^{\xi h_{\ep \BB r}(0)} .
\eqe
Here we use the notation for $D_h$-distances across and around Euclidean annuli as explained in Section~\ref{sec-notation}. 

By Proposition~\ref{prop-scaling-constants}, we have $\frk c_{\ep \BB r} / \frk c_{\BB r} = \ep^{\xi Q + o_\ep(1)}$, with the rate of convergence of the $o_\ep(1)$ uniform in $\BB r$, so with superpolynomially high probability as $\ep \rta 0$, 
\eqb \label{eqn-ball-contain-ratio}
\frac{ D_h\left( \text{around $\BB A_{\ep \BB r, 2\ep \BB r}(0)$} \right) }{D_h\left( \text{across $\BB A_{\BB r/2,\BB r}(0)$} \right)}
\leq \ep^{\xi Q - 2\zeta  +  o_\ep(1)} e^{\xi (h_{\ep \BB r}(0) - h_{\BB r}(0)) } .
\eqe
The random variable $h_{\ep \BB r}(0) - h_{\BB r}(0)$ is centered Gaussian with variance $\log \ep^{-1}$, so by the Gaussian tail bound it holds with polynomially high probability as $\ep\rta 0$ that $e^{\xi (h_{\ep \BB r}(0) - h_{\BB r}(0)) } \leq \ep^{-(\xi Q - 3\zeta)}$. By~\eqref{eqn-ball-contain-ratio}, it therefore holds with polynomially high probability as $\ep\rta 0$ that 
\eqb  \label{eqn-ball-contain-smaller}
D_h\left( \text{around $\BB A_{\ep \BB r, 2\ep \BB r}(0)$} \right)  <  D_h\left( \text{across $\BB A_{\BB r/2,\BB r}(0)$} \right) .
\eqe

Suppose that~\eqref{eqn-ball-contain-smaller} holds. We claim that $B_{\ep \BB r}(0) \subset \mcl B_{\tau_\BB r}^\bullet$. Let $\pi$ be a path in $\BB A_{\ep \BB r, 2\ep \BB r}(0)$ which disconnects the inner and outer boundaries of $\BB A_{\ep \BB r, 2\ep \BB r}(0)$ and has $D_h$-length less than $D_h\left( \text{across $\BB A_{\BB r/2,\BB r}(0)$} \right)$. Also let $P$ be a $D_h$-geodesic from 0 to a point of $\bdy\mcl B_{\tau_{\BB r}}^\bullet \cap \bdy B_{\BB r}(0)$. Then $P$ hits $\pi$ before leaving $B_{\BB r/2}(0)$ and the segment of $P$ after it leaves $B_{\BB r/2}(0)$ has $D_h$-length at least $ D_h\left( \text{across $\BB A_{\BB r/2,\BB r}(0)$} \right)$. Since the $D_h$-length of $\pi$ is smaller than $ D_h\left( \text{across $\BB A_{\BB r/2,\BB r}(0)$} \right)$, we get that $\pi \subset \mcl B_{\tau_{\BB r}}$. Since $\pi$ disconnects $B_{\ep \BB r}(0)$ from $\infty$, it follows that $B_{\ep \BB r}(0) \subset \mcl B_{\tau_{\BB r}}^\bullet$.
\end{proof}

\subsection{Regularity of distances on outer boundaries of metric balls}
\label{sec-hit-ball}

A key ingredient for many of the proofs in this paper is the following lemma, which implies every point on the boundary of a filled $D_h$-metric ball can be surrounded by a path of small $D_h$-length, in a sense which is uniform over all points in any Euclidean-bounded open set (this is in contrast to Lemma~\ref{lem-separating-annuli}, which does not give any uniform control on the rate of convergence). A closely related lemma for LQG geodesics is proven in~\cite[Section 2.4]{pfeffer-supercritical-lqg}. Note that we include a Euclidean scale parameter $\BB r$ in the estimates of this subsection since we will need them to be uniform across Euclidean scales.

\begin{figure}[ht!]
\begin{center}
\includegraphics[scale=1]{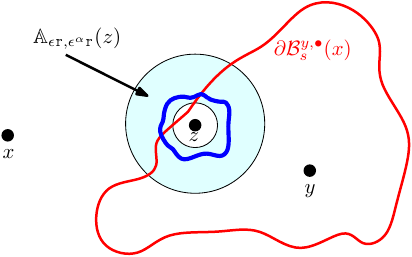} 
\caption{\label{fig-hit-ball} 
Illustration of the statement of Lemma~\ref{lem-hit-ball}. 
The lemma asserts that if $z\in \BB r U$, $x,y\in\BB C \setminus B_{ \ep^\alpha \BB r}(z)$, and $s > 0$ such that the filled metric ball boundary $\bdy\mcl B_s^{y,\bullet}(x)$ (red) intersects $B_{\ep \BB r}(z)$, then we can find a path (blue) which disconnects the inner and outer boundaries of the annulus $\BB A_{\ep \BB r, \ep^\alpha \BB r}(z)$ which has short $D_h$-length, in the sense of~\eqref{eqn-hit-ball}.  
}
\end{center}
\end{figure}

\begin{lem} \label{lem-hit-ball}
For each $\alpha \in (0,1)$, there exists $\beta = \beta(\alpha  ) > 0$ such that for each Euclidean-bounded open set $U\subset \BB C$ and each $\BB r > 0$,  
it holds with polynomially high probability as $\ep \rta 0$, uniformly over the choice of $\BB r$, that following is true. 
Suppose $z\in \BB r U$, $x,y\in\BB C \setminus B_{ \ep^\alpha \BB r}(z)$, and $s > 0$ such that the boundary of the filled metric ball $\bdy\mcl B_s^{y,\bullet}(x)$ intersects $B_{\ep \BB r}(z)$. 
Then 
\eqb \label{eqn-hit-ball}
D_h\left( \text{around $\BB A_{ \ep \BB r , \ep^\alpha \BB r}(z)$} \right) 
\leq \ep^\beta D_h\left( \text{across $\BB A_{   \ep \BB r , \ep^\alpha \BB r}(z)$} \right) .
\eqe
\end{lem}

See Figure~\ref{fig-hit-ball} for an illustration of the statement of Lemma~\ref{lem-hit-ball}. 
We will most often use the following slightly weaker estimate, which is an immediate consequence of Lemma~\ref{lem-hit-ball}. 

\begin{cor} \label{cor-hit-ball'}
Suppose we are in the setting of Lemma~\ref{lem-hit-ball}. On the polynomially high probability event of that lemma, the following is true. Suppose $z\in \BB r U$, $x,y\in\BB C \setminus B_{\ep^\alpha \BB r}(z)$, and $s>0$ such that either $\bdy\mcl B_s^{y,\bullet}(x) \cap B_{\ep \BB r}(z) \not=\emptyset$ or there is a $D_h$-geodesic $P$ from $x$ to $y$ with $P(s) \in B_{\ep \BB r}(z)$. Then  
\eqb \label{eqn-hit-ball'}
D_h\left( \text{around $\BB A_{ \ep \BB r , \ep^\alpha \BB r}(z)$} \right) 
\leq \ep^\beta s.
\eqe
\end{cor}
\begin{proof}
If $P$ is a $D_h$-geodesic from $x$ to $y$, then necessarily $P(s) \in \bdy\mcl B_s^{y,\bullet}(x)$, so if $P(s) \in B_{\ep \BB r}(z)$ then $\bdy\mcl B_s^{y,\bullet}(x) \cap B_{\ep \BB r}(z) \not=\emptyset$. Hence, Lemma~\ref{lem-hit-ball} implies that for $z,x,y,s$ as in the lemma statement the bound~\eqref{eqn-hit-ball} is satisfied. Since any $D_h$-geodesic from $x$ to a point of $\bdy\mcl B_s^{\bullet,y} \cap B_{\ep \BB r}(z)$ has length $s$ and must cross between the inner and outer boundaries of $B_{\ep^\alpha \BB r}(z) \setminus B_{\ep \BB r}(z)$, we see that~\eqref{eqn-hit-ball} implies~\eqref{eqn-hit-ball'}. 
\end{proof}

Intuitively, the reason why Lemma~\ref{lem-hit-ball} and Corollary~\ref{cor-hit-ball'} are true is that points on the boundary of a filled $D_h$-metric ball or on a $D_h$-geodesic should be in some sense far from being singular points (since they are at finite distance from at least one point). Hence it should be possible to find short paths which disconnect small Euclidean neighborhoods of such points from $\infty$ (roughly speaking, this is a quantitative version of Lemma~\ref{lem-separating-annuli}).

Corollary~\ref{cor-hit-ball'} can be thought of as a substitute for the fact that for supercritical LQG (unlike in the subcritical case) we do not know that the identity mapping $(\BB C , |\cdot|) \rta (\BB C , D_h)$ is H\"older continuous. To be more precise, the corollary tells us that points on the outer boundary of a filled $D_h$-metric ball or on a $D_h$-geodesic can be surrounded by paths of small Euclidean size whose $D_h$-length is small. By forcing these paths to cross other paths, we will be able to establish upper bounds for the $D_h$-distance between points near filled metric ball boundaries or geodesics in terms of their Euclidean distance. Since many estimates for the LQG metric only require us to work with points near filled metric ball boundaries or geodesics, this will be a suitable substitute for H\"older continuity. 

The idea of the proof of Lemma~\ref{lem-hit-ball} is to surround the Euclidean ball $B_{\ep \BB r}(z)$ by logarithmically many disjoint concentric Euclidean annuli contained in $\BB A_{\ep \BB r , \ep^\alpha \BB r}(z)$ with the property that the $D_h$-distances around and across each of the annuli are comparable. This will be done using Lemma~\ref{lem-annulus-iterate}. We will order these annuli from outside to inside. Using a deterministic lemma (see Lemma~\ref{lem-seq-sum-count}), we will argue that in order for a filled metric ball boundary to intersect $B_\ep(z)$, there must be at least one annulus such that the $D_h$-distance around this annulus is smaller than a positive power of $\ep$ times the sum of the $D_h$-distances across the subsequent annuli. This latter sum provides a lower bound for $ D_h\left( \text{across $\BB A_{   \ep \BB r , \ep^\alpha \BB r}(z)$} \right)$. 

Let us now construct the concentric annuli that we will work with. 
An \emph{annulus with aspect ratio 2} is an open annulus of the form $A = \BB A_{r,2r}(z)$ for some $z\in\BB C$ and $r>0$. For an annulus $A$ with aspect ratio 2 and a number $c >0$, we define
\eqb \label{eqn-annulus-event}
E_c(A) := \left\{ D_h\left(\text{around $A $} \right) \leq  (1/c) D_h\left(\text{across $A$} \right) \right\} .
\eqe 

\begin{lem} \label{lem-annulus-cover}
For each $\alpha \in (0,1)$, there exists $\eta  \in (0,1-\alpha)$ and $c \in (0,1)$ such that for each Euclidean-bounded open set $U\subset\BB C$ and each $\BB r > 0$, it holds with polynomially high probability as $\ep\rta 0$, uniformly over the choice of $\BB r$, that the following is true. 
For each $z\in \BB r U$, there exist $N := \lceil \eta \log\ep^{-1} \rceil$ disjoint concentric annuli $A_1,\dots,A_N  \subset \BB A_{\ep\BB r , \ep^\alpha\BB r}(z)$ which each disconnects $\bdy B_{\ep\BB r}(z)$ from $\bdy B_{\ep^\alpha \BB r}(z)$ such that $E_c(A_n)$ occurs for each $n=1,\dots,N$. 
\end{lem}
\begin{proof}
This is a straightforward consequence of the near-independence of the restriction of the GFF to disjoint concentric annuli (Lemma~\ref{lem-annulus-iterate}) together with a union bound over points in an fine mesh of $\BB r U$. Let us now give the details. 

For $z \in \BB C$ and $k \in \BB N$, let $A_{k,\ep}(z) := \BB A_{ 2^{2k} \ep \BB r  ,  2^{2k+1}\ep \BB r}(z) $. Note that the annuli $A_{k,\ep}(z)$ for different values of $k$ are disjoint and for each $k$, the region between the annuli $  A_{k,\ep}(z) $ and $ A_{k+1,\ep}(z) $ is the annulus $\BB A_{2^{2k+1}\ep \BB r , 2^{2 k+ 2}\ep \BB r }(z)$. 
Furthermore, if we set $K_\ep := \lfloor \frac13 \log_2 \ep^{-(1-\alpha)} \rfloor -1$, then 
\eqb \label{eqn-annulus-nest}
  A_{k,\ep}(z)  \subset \BB A_{2\ep \BB r , \ep^\alpha \BB r /2}(z)  ,\quad \forall k\in [1,K_\ep]_{\BB Z} .
\eqe 
The reason why we want $2\ep\BB r$ and $\ep^\alpha\BB r/2$ instead of just $\ep\BB r$ and $\ep^\alpha\BB r$ in~\eqref{eqn-annulus-nest} is that we will need to slightly adjust the radii of our annuli when we pass from a statement for points in a fine mesh to a statement for all points simultaneously. 

By the definition~\eqref{eqn-annulus-event} of $E_c(A_{k,\ep}(z) )$, this event is a.s.\ determined by the internal metric of $D_h$ on $A_{k,\ep}(z) $.
By the locality and Weyl scaling properties of $D_h$ (Axioms~\ref{item-metric-local} and~\ref{item-metric-f}), each of the events $E_c(A_{k,\ep}(z))$ is a.s.\ determined by the restriction of $h$ to $A_{k,\ep}(z) $, viewed modulo additive constant. 
By the translation invariance and tightness across scales properties of $D_h$ (Axioms~\ref{item-metric-translate} and~\ref{item-metric-coord}), for any $p\in (0,1)$ we can find $c  = c(p) \in (0,1)$ such that $\BB P[E_c(A_{k,\ep}(z))] \geq p$ for each $z\in\BB C$, $k\in\BB N$, and $\ep > 0$. 

We may therefore apply Lemma~\ref{lem-annulus-iterate} to find $c \in (0,1)$ and $\eta \in (0,1-\alpha)$ such that for each $z\in \BB C$, it holds with probability at least $1-O_\ep(\ep^3)$ that there are at least $\eta \log \ep^{-1}$ values of $k \in [1,K_\ep]_{\BB Z}$ for which $E_c(A_{\ep,k}(z))$ occurs. By a union bound, it holds with polynomially high probability as $\ep\rta 0$ that for each $z \in \left(\frac{\ep \BB r}{4} \BB Z^2 \right) \cap B_{\ep \BB r}(\BB r U)$, there are at least $\eta \log \ep^{-1}$ values of $k \in [1,K_\ep]_{\BB Z}$ for which $E_c(A_{\ep,k}(z))$ occurs. Henceforth assume that this is the case. 

Let $z\in U$. We can find $z'\in \left(\frac{\ep \BB r}{4} \BB Z^2 \right) \cap B_{\ep\BB r}(U)$ such that $z\in B_{\ep \BB r/2}(z')$. Then $B_{\ep\BB r}(z) \subset B_{2\ep \BB r}(z')$ and $B_{\ep^\alpha \BB r/2}(z') \subset B_{\ep^\alpha \BB r}(z )$. By~\eqref{eqn-annulus-nest}, the conditions in the lemma statement hold with $A_1,\dots,A_N$ chosen to be $N = \lceil \eta \log\ep^{-1} \rceil$ of the annuli $A_{k,\ep}(z')$ for $k\in [1,K_\ep]_{\BB Z}$ for which $E_c(A_{k,\ep}(z'))$ occurs. 
\end{proof}

The following deterministic lemma will allow us to choose one of the annuli $A_n$ from Lemma~\ref{lem-annulus-cover} in such a way that $D_h(\text{around $A_n$})$ is much smaller than $\sum_{j=n+1}^N D_h(\text{across $A_j$})$. See~\cite[Lemma 2.20]{pfeffer-supercritical-lqg} for a proof. 

\begin{lem} \label{lem-seq-sum-count}
Let $X_1 , ..., X_N$ be non-negative real numbers. For each $c > 0$, 
\eqb \label{eqn-seq-sum-count}
\#\left\{ n  \in [1,N]_{\BB Z} : X_n \geq c \sum_{j=n+1}^N X_j \right\} \leq  \max\left\{1 ,  \frac{ \log\left( \frac{1}{X_N} \max_{n \in [1,N]_{\BB Z}} X_n \right) }{\log(c+1)} - \frac{\log c}{\log(c+1)}  + 2 \right\} . 
\eqe
\end{lem}

\begin{proof}[Proof of Lemma~\ref{lem-hit-ball}]
Let $\alpha > 0$  and let $\eta $ and $c$ be chosen as in Lemma~\ref{lem-annulus-cover}. 
Also fix a Euclidean-bounded open set $U\subset\BB C$ and a number $\BB r > 0$. 
Throughout the proof, we work on the polynomially high probability event of Lemma~\ref{lem-annulus-cover}. 

Let $z\in \BB r U$ and let $A_1,\dots,A_N \subset \BB A_{\ep \BB r , \ep^\alpha \BB r}(z)$ be the disjoint concentric annuli from Lemma~\ref{lem-annulus-cover}, numbered from outside in. 
For $n\in [1,N]_{\BB Z}$, define
\eqb \label{eqn-X_n-def} 
X_n := D_h\left(\text{around $A_n $} \right) ,
\eqe
so that by the definition of $E_c(A_n)$,  
\eqb \label{eqn-X_n-relation}
X_n \leq (1/c) D_h\left(\text{across $A_n$}\right) .
\eqe

Suppose that there exists $x,y\in\BB C \setminus B_{\ep^\alpha \BB r}(z)$ and $s > 0$ such that $\bdy\mcl B_s^{y,\bullet}(x) \cap B_{\ep\BB r}(z) \not=\emptyset$. We need to show that~\eqref{eqn-hit-ball} holds for an appropriate choice of $\beta$. With a view toward applying Lemma~\ref{lem-seq-sum-count}, we claim that
\eqb \label{eqn-X_n-sum}
X_n \geq   c \sum_{j=n+1}^N X_j , \quad\forall n \in [1,N-1]_{\BB Z} .
\eqe

Indeed, suppose by way of contradiction that~\eqref{eqn-X_n-sum} does not hold for some $n \in [1,N-1]_{\BB Z}$, i.e., $X_n < c\sum_{j=n+1}^N X_j$. By~\eqref{eqn-X_n-relation}, for this choice of $n$,
\eqb \label{eqn-dist-sum}
D_h\left(\text{around $A_n $} \right) <  \sum_{j=n+1}^N D_h\left(\text{across $A_j$}\right) \leq D_h\left(A_n  , B_{\ep\BB r}(z) \right) ,
\eqe
where the last inequality follows since the $A_j$'s are disjoint, numbered from outside in, and surround $B_{\ep\BB r}(z)$. 
Let $\pi$ be a path in $A_n $ which disconnects the inner and outer boundaries of $A_n $ and has $D_h$-length strictly less than $D_h\left(A_n  , B_{\ep\BB r}(z) \right)$. 

Let $w\in \bdy\mcl B_s^{y,\bullet}(x) \cap B_{\ep\BB r}(z)$. By Lemma~\ref{lem-outer-bdy-dist}, $D_h(x,w) = s$. 
There is a $D_h$-geodesic $P  : [0,s] \rta \BB C$ from $x$ to $w$. 
Let $\tau$ be the first time that $P$ hits $\pi$. Since $P$ is a geodesic, the $D_h$-distance from $x$ to each point of $\pi$ is at most $\tau + (\text{$D_h$-length of $\pi$})$, which by the preceding paragraph is less than $\tau + D_h(A_n , B_{\ep\BB r}(z))$. 

On the other hand, $P$ must travel from $A_n $ to $B_{\ep\BB r}(z)$ after time $\tau$, so $s  \geq \tau +  D_h(A_n , B_{\ep\BB r}(z))$. Therefore, each point of $\pi$ lies at $D_h$-distance less than $  s$ from $x$, so $\pi \subset \mcl B_s(x)$. Since $x, y\notin B_{\ep^\alpha \BB r}(z)$ and $\pi$ is contained in $B_{\ep^\alpha \BB r}(z)$ and disconnects $B_{\ep\BB r}(z)$ from $\bdy B_{\ep^\alpha \BB r}(z)$, we get that $\mcl B_s(x)$ disconnects $B_{\ep \BB r}(z)$ from $x$ and $y$. Therefore, $B_{\ep \BB r}(z) \cap \mcl B_s^{y,\bullet}(x) =\emptyset$, which is our desired contradiction.
Hence~\eqref{eqn-X_n-sum} holds.

By~\eqref{eqn-X_n-sum}, there are $N$ values of $n\in[1,N]_{\BB Z}$ for which $X_n \geq   c \sum_{j=n+1}^N X_j$. 
Therefore, Lemma~\ref{lem-seq-sum-count} gives
\eqb \label{eqn-use-seq-sum-count} 
\frac{\log\left(\frac{1}{X_N} \max_{n\in [1,N]_{\BB Z}} X_n \right) }{\log(c+1)} \geq N - O(1) ,
\eqe
where the $O(1)$ denotes a constant depending only on $c$ (not on $\ep$). 
Therefore, 
\allb
D_h\left( \text{around $B_{\ep^\alpha \BB r}(z) \setminus B_{\ep \BB r}(z)$} \right)
&\leq X_N \quad \text{(by the definition~\eqref{eqn-X_n-def} of $X_n$)} \notag\\
&\leq (c+1)^{-N + O(1)} \max_{n\in [1,N]_{\BB Z}} X_n \quad \text{(by re-arranging~\eqref{eqn-use-seq-sum-count})} \notag\\ 
&\leq O(1) (c+1)^{-N}   D_h\left(\text{across $\BB A_{\ep\BB r , \ep^\alpha \BB r}(z)$} \right)\quad \text{(by~\eqref{eqn-X_n-relation})} . 
\alle
Since $N = \lceil \eta\log\ep^{-1} \rceil$, for small enough $\ep$ the quantity $O(1) (c+1)^{-N} $ is bounded above by $\ep^\beta$ for an appropriate choice of $\beta > 0$.
This gives~\eqref{eqn-hit-ball}. 
\end{proof}

\subsection{Lower bound for $D_h$-distances across LQG annuli}
\label{sec-bdy-dist}

An easy consequence of Lemma~\ref{lem-hit-ball} is the following lemma, which gives a polynomial lower bound for the Euclidean distance between the outer boundaries of concentric filled $D_h$-metric balls. This lemma will play an important role in the proof of Theorem~\ref{thm-net-dim} and in the proof of confluence of geodesics.

\begin{lem} \label{lem-bdy-dist0}
There exists $\beta  > 0$ such that the following is true. 
Fix $b >1$ and for $\BB r >0$ let $\tau_{\BB r} = D_h(0,\bdy B_{\BB r}(0))$ be as in~\eqref{eqn-tau_r-def}. It holds with probability tending to 1 as $\delta \rta 0$, uniformly in the choice of $\BB r$, that for each $ s, t \in [\tau_{\BB r} , \tau_{b \BB r}]$ with $|s-t| \leq \delta   \frk c_{\BB r} e^{\xi h_{\BB r}(0)} $, 
\eqb \label{eqn-bdy-dist0}
\op{dist}\left( \bdy \mcl B_s^{\bullet} , \bdy \mcl B_t^\bullet \right) \geq \left(\frac{|s-t|}{\frk c_{\BB r} e^{\xi h_{\BB r}(0)}} \right)^{1/\beta} \BB r ,
\eqe 
where $\op{dist}$ denotes Euclidean distance. 
\end{lem}

Note that in the subcritical case,~\eqref{eqn-bdy-dist0} is immediate from the local H\"older continuity of $D_h$ w.r.t.\ the Euclidean metric~\cite[Theorem 1.7]{lqg-metric-estimates}, so the lemma has non-trivial content only in the case when $\xi \geq \xi_{\op{crit}}$. 
We will deduce Lemma~\ref{lem-bdy-dist0} from the following more quantitative statement, which allows for a general choice of starting points and target points for the filled metric balls. For the statement, we recall the notation $B_r(X) = \bigcup_{z\in X} B_r(z)$ for the Euclidean $r$-neighborhood of a set $X \subset\BB C$. 

\begin{lem} \label{lem-bdy-dist}
For each $\alpha  \in (0,1)$, there exists $\beta  = \beta(\alpha) > 0  $ such that the following is true. 
Let $U\subset \BB C$ be a Euclidean-bounded open set and let $\BB r > 0$. 
With polynomially high probability as $\ep\rta 0$, uniformly over the choice of $\BB r$, it holds for each non-singular point $x \in \BB C$, each $y\in\BB C\cup \{\infty\}$, and each $s \in [\ep \frk c_{\BB r} e^{\xi h_{\BB r}(0)} ,D_h(x,y)  ]$ that
\eqb \label{eqn-bdy-dist} 
\op{dist}\left( \bdy \mcl B_s^{y,\bullet}(x) \cap (\BB r U\setminus B_{\ep^{\alpha  /\beta} \BB r}(\{x,y\}) ) , \bdy \mcl B_{(1-\ep) s}^{y,\bullet}(x)  \right) \geq \ep^{1/\beta} \BB r ,
\eqe
where $\op{dist}$ denotes Euclidean distance, we define $B_{\ep^{\alpha/\beta}}(\{x, \infty\} ) = B_{\ep^{\alpha/\beta}}(x)$, and we make the convention that the distance from any set to the empty set is $\infty$ (which is consistent with the convention that the infimum of the empty set is $\infty$).
\end{lem}
\begin{proof}
Let $\wt\beta = \wt\beta(\alpha)  > 0$ be the parameter $\beta$ from Corollary~\ref{cor-hit-ball'} and let $\beta = \wt\beta/2$. By Corollary~\ref{cor-hit-ball'} (applied with $\ep^{1/\beta}$ instead of $\ep$), it holds with polynomially high probability as $\ep\rta 0$ that for each $x \in \BB C$, each $y\in \BB C \cup \{\infty\}$, each $s\in [0,D_h(x,y)]$, and each $z\in  \bdy \mcl B_s^{y,\bullet}(x)  \cap (\BB r U\setminus B_{\ep^{\alpha/\beta} \BB r}(\{x,y\}) )$,  
\eqb
D_h\left( \text{around $\BB A_{\ep^{1/\beta} \BB r  , \ep^{\alpha/\beta} \BB r }(z) $} \right) 
\leq   \ep^2 s .
\eqe
We henceforth work on the polynomially high probability event that this is the case. 

Let $x,y,s,z$ be as above with $x$ non-singular and let $\pi$ be a path in $\BB A_{\ep^{1/\beta} \BB r  , \ep^{\alpha/\beta} \BB r }(z) $ which disconnects the inner and outer boundaries of this annulus and has $D_h$-length at most $(\ep/2) s$. 
If $\bdy \mcl B_{(1-\ep)s}^{y,\bullet}(x) \cap B_{\ep^{1/\beta} \BB r}(z) \not=\emptyset$, then since $x,y\notin B_{\ep^{\alpha/\beta} \BB r}(z)$ it must be the case that each of $\bdy\mcl B_{(1-\ep)s}^{y,\bullet}(x)$ and $\bdy\mcl B_s^{y,\bullet}(x)$ intersects $\pi$. This implies that the $D_h$-distance between $\bdy\mcl B_{(1-\ep)s}^{y,\bullet}(x)$ and $\bdy\mcl B_s^{y,\bullet}(x)$ is at most $(\ep/2) s$. This cannot be the case since Lemma~\ref{lem-outer-bdy-dist} implies the $D_h$-distance between  $\bdy\mcl B_{(1-\ep)s}^{y,\bullet}(x)$ and $\bdy\mcl B_s^{y,\bullet}(x)$ is $\ep s$. Therefore, $\bdy \mcl B_{(1-\ep)s}^{y,\bullet}(x) \cap B_{\ep^{1/\beta} \BB r}(z)  =\emptyset$, so~\eqref{eqn-bdy-dist} holds. 
\end{proof}

\begin{proof}[Proof of Lemma~\ref{lem-bdy-dist0}]
Let $\beta$ be the parameter from Lemma~\ref{lem-bdy-dist} with $\alpha = 1/2$. By Lemma~\ref{lem-ball-contain}, it holds with probability tending to 1 as $\delta \rta 0$ that $B_{\delta^{1/(2\beta) }\BB r}(0) \subset \mcl B_{\tau_{\BB r}}^\bullet$, which means that also $B_{\delta^{1/(2\beta)} \BB r}(0) \subset \mcl B_s^\bullet$ for each $s\geq \tau_{\BB r}$.  
Furthermore, by tightness across scales (Axiom~\ref{item-metric-coord}) it holds with probability tending to 1 as $\delta \rta 0$ that $\tau_{\BB r} \geq \delta \frk c_{\BB r} e^{\xi h_{\BB r}(0)}$. 
Hence with probability tending to 1 as $\delta \rta 0$, we have $\tau_{\BB r} \geq \delta \frk c_{\BB r} e^{\xi h_{\BB r}(0)}$ and $\bdy\mcl B_s^\bullet \cap (B_{b\BB r}(0) \setminus B_{\delta^{1/(2\beta)}\BB r}(0))= \bdy\mcl B_s^\bullet $ for each $s\in [\tau_{\BB r} , \tau_{b\BB r}]$. 

We now apply Lemma~\ref{lem-bdy-dist} (with $U = B_b(0)$) and a union bound over dyadic values of $\ep$, followed by the estimate of the preceding paragraph, to get that with probability tending to 1 as $\delta \rta 0$, the following is true. For $\ep \in (0,\delta) \cap \{2^{-k}\}_{k\in\BB N}$ and each $s\in [\tau_{\BB r} , \tau_{b\BB r}]$,
\eqb \label{eqn-bdy-dist0-eps}
\op{dist}\left( \bdy \mcl B_s^\bullet, \bdy\mcl B_{(1-\ep) s}^\bullet \right) \geq \ep^{1/ \beta} \BB r . 
\eqe
By Lemma~\ref{lem-set-tightness}, for any $p\in (0,1)$ we can find $C = C(p , b) > 1$ such that for each $\BB r > 0$,  
\eqb \label{eqn-bdy-dist0-reg}
\BB P\left[ C^{-1}  \frk c_{\BB r}  e^{ \xi h_{\BB r}(0)} \leq \tau_{\BB r} \leq \tau_{b\BB r} \leq C  \frk c_{\BB r}  e^{ \xi h_{\BB r}(0)}  \right] \geq p    .
\eqe 

Now suppose that the event in~\eqref{eqn-bdy-dist0-reg} holds and the event in~\eqref{eqn-bdy-dist0-eps} holds with $C\delta$ in place of $\delta$, which happens with probability at least $p - o_\delta(1)$. By~\eqref{eqn-bdy-dist0-reg}, for any $s,t\in [\tau_{\BB r} , \tau_{b\BB r}]$ with $s- \delta  \frk c_{\BB r} e^{\xi h_{\BB r}(0)}   \leq t \leq s$, we have $t \leq (1-\ep) s$ for some dyadic $\ep \in (0, C\delta)$ which satisfies
\eqb \label{eqn-bdy-dist0-choice}
\ep \geq \frac{s-t}{2s} \geq \frac{1}{2C} \frac{s-t}{\frk c_{\BB r} e^{\xi h_{\BB r}(0)}  } .
\eqe
We conclude by combining~\eqref{eqn-bdy-dist0-choice} with~\eqref{eqn-bdy-dist0-eps}, replacing $\beta$ by a slightly smaller number to absorb the factor of $1/(2C)$ into a small power of $\ep$, and noting that the parameter $p$ from~\eqref{eqn-bdy-dist0-reg} can be made arbitrarily close to 1. 
\end{proof}

\subsection{The metric net is finite-dimensional}
\label{sec-net-dim}

We will now use Lemma~\ref{lem-bdy-dist} to prove Theorem~\ref{thm-net-dim}. 
Since we are proving an a.s.\ statement, we no longer need to include the Euclidean scale parameter $\BB r$.

\begin{proof}[Proof of Theorem~\ref{thm-net-dim}] 
We write $\dim_h$ for Hausdorff dimension w.r.t.\ $D_h$. 
Fix a Euclidean-bounded open set $U\subset \BB C$, a number $r > 0$, and numbers $s_2 > s_1 > 0$. By the countable stability of Hausdorff dimension, it suffices to show that there exists $\Delta \in (0,\infty)$ (not depending on $U , r , s_1,s_2$) such that a.s.\ for each non-singular point $x \in U$ and each $y\in\BB C \cup \{\infty\}$, 
\eqb \label{eqn-net-dim-show}
\dim_h\left( \left( \mcl N_{s_2}^y(x) \setminus \mcl N_{s_1}^y(x) \right) \cap \left( U \setminus B_r(\{x,y\}) \right) \right) \leq \Delta .
\eqe

See Figure~\ref{fig-net-dim} for an illustration of the proof. The idea is as follows. We consider the set of $\ep^{1/\beta} \times \ep^{1/\beta}$ squares with corners in $\ep^{1/\beta}\BB Z^2$ which intersect a neighborhood of $U$. By Proposition~\ref{prop-two-set-dist} and an estimate for the maximum of the circle average process $h_{\ep^{1/\beta}}$, each of these squares can be surrounded by a path $\pi_S$ of Euclidean diameter comparable to $\ep^{1/\beta}$ whose $D_h$-length is at most a negative power of $\ep$. The number of $D_h$-balls of radius $\ep$ needed to cover each of these paths is at most a negative power of $\ep$. Using Lemma~\ref{lem-bdy-dist}, we show that for each $s\in [s_1,s_2]$, each $D_h$-geodesic from a point of $\bdy\mcl B_s^{y,\bullet}(x)$ to $x$ must hit $\pi_S$ for one of the $\ep^{1/\beta} \times \ep^{1/\beta}$ squares $S$ which intersects $\bdy\mcl B_{s-\ep}^{y,\bullet}(x)$, and it must do so before time $\ep$. This shows that the set in~\eqref{eqn-net-dim-show} is contained in the union of a polynomial (in $\ep$) number of $D_h$-balls of radius $2\ep$. 
\medskip

\begin{figure}[ht!]
\begin{center}
\includegraphics[scale=1]{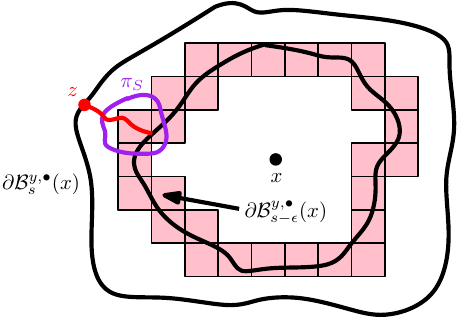} 
\caption{\label{fig-net-dim} Illustration of the proof of Theorem~\ref{thm-net-dim}. The figure shows a point $z\in \bdy\mcl B_s^{y,\bullet}(x)$ for some $s\in [s_1,s_2]$ (red) and the subset of $\mcl S_\ep$ consisting of squares $S$ which intersect $\bdy\mcl B_{s-\ep}^{y,\bullet}(x)$ (pink). Thanks to Lemma~\ref{lem-bdy-dist}, we can arrange that each of these squares lies at Euclidean distance at least $\ep^{1/\beta}$ from $\bdy\mcl B_s^{y,\bullet}$. Moreover, using basic estimates for $D_h$ we can arrange that each square $S$ is surrounded by a path $\pi_S \subset B_{\ep^{1/\beta}}(S) \setminus S$ whose $D_h$-length is bounded above by a negative power of $\ep$ (one such path is shown in purple). Hence the number of $D_h$-balls needed to cover $\pi_S$ is at most a negative power of $\ep$. If $P$ is a $D_h$-geodesic from 0 to $z$, then $P([s-\ep,s])$ (red) must intersect $\pi_S$ for some $S\in\mcl S_\ep$, so $z$ is contained in the $\ep$-neighborhood of one of the $D_h$-metric balls in our covering of $\pi_S$.  This leads to an upper bound for the number of $D_h$-balls of radius $2\ep$ needed to cover the metric net. 
}
\end{center}
\end{figure}

\noindent\textit{Step 1: regularity events.}
Let $\wt\beta > 0$ be the parameter $\beta$ from Lemma~\ref{lem-bdy-dist} with $\alpha =1/2$, say, and let $\beta = \wt\beta/2$. 
By Lemma~\ref{lem-bdy-dist}, it holds with probability tending to 1 as $\ep\rta 0$ that for each non-singular point $x\in \BB C$, each $y\in\BB C\cup\{\infty\}$, and each $s \in [s_1 , s_2 \wedge D_h(x,y) ]$, 
\eqb \label{eqn-use-bdy-dist} 
\op{dist}\left( \bdy \mcl B_s^{y,\bullet}(x) \cap (U\setminus B_r(\{x,y\}) ) , \bdy \mcl B_{s - \ep}^{y,\bullet}(x)  \right) \geq 4 \ep^{1/\beta} . 
\eqe
Note that here we have used that $s \in [s_1,s_2]$ to absorb an $s$-dependent constant factor into a power of $\ep$. 

Let $\mcl S_\ep$ be the set of $\ep^{1/\beta} \times \ep^{1/\beta}$ squares with corners in $\ep^{1/\beta} \BB Z^2$ which intersect the LQG $s_2$-neighborhood $\mcl B_{s_2}(U)$.
For each $S\in\mcl S_\ep$, we define the annular region
\eqb
A_S := B_{\ep^{1/\beta}}(S) \setminus S .
\eqe 
Since $\mcl B_{s_2}(U)$ is a.s.\ contained in some Euclidean-bounded open set, we can apply Proposition~\ref{prop-two-set-dist} and a union bound over $S\in\mcl S_\ep$ to get that with probability tending to 1 as $\ep\rta 0$, 
\eqb \label{eqn-net-dim-around0}
D_h\left(\text{around $A_S$}\right) \leq \ep^{o_\ep(1)} \frk c_{\ep^{1/\beta}} e^{\xi h_{\ep^{1/\beta}}(v_S)} , \quad\forall S \in \mcl S_\ep , 
\eqe
where $v_S$ is the center of $S$ and the rate of convergence of the $o_\ep(1)$ is deterministic and uniform over all $S\in\mcl S_\ep$. 
 
The random variables $e^{\xi h_{\ep^{1/\beta}}(v_S)}$ for $S\in\mcl S_\ep$ are centered Gaussian with variances $\log\ep^{-1/\beta}  + O_\ep(1)$. By the Gaussian tail bound and a union bound over $O_\ep(\ep^{-2/\beta})$ squares, we get that with probability tending to 1 as $\ep\rta 0$, we have $h_{\ep^{1/\beta}}(v_S) \leq (2/\beta + o_\ep(1) ) \log\ep^{-1}$ for each $S\in\mcl S_\ep$. By Proposition~\ref{prop-scaling-constants}, we also have $\frk c_{\ep^{1/\beta}} = \ep^{(1/\beta) \xi Q + o_\ep(1)}$.  
By plugging these estimates into~\eqref{eqn-net-dim-around0}, we get that with probability tending to 1 as $\ep\rta 0$, 
\eqb \label{eqn-net-dim-around}
D_h\left(\text{around $A_S$}\right) \leq   \ep^{-(1/\beta) (2-Q) + o_\ep(1)}       , \quad\forall S \in \mcl S_\ep , 
\eqe
where the rate of convergence of the $o_\ep(1)$ is deterministic and uniform over all $S\in\mcl S_\ep$. 

Henceforth assume that~\eqref{eqn-use-bdy-dist} and~\eqref{eqn-net-dim-around} both hold, which happens with probability tending to 1 as $\ep\rta 0$. We will prove an upper bound for the number of $D_h$-balls of radius $\ep$ needed to cover the set on the left side of~\eqref{eqn-net-dim-show}.
\medskip

\noindent\textit{Step 2: defining a collection of $D_h$-metric balls.}
For $S\in\mcl S_\ep$, let $\pi_S$ be a path in $A_S$ which separates the inner and outer boundaries of $A_S$ and which has $D_h$-length at most $\ep^{-(1/\beta) (2-Q) + o_\ep(1)} $ (such a path exists by~\eqref{eqn-net-dim-around}). There is a set $\mcl M_S$ of $\#\mcl M_S \leq \ep^{-(1/\beta)(2-Q) - 1 + o_\ep(1)}$ $D_h$-metric balls of radius $\ep$ whose union contains $\pi_S$. Since $\#\mcl S_\ep = O_\ep(\ep^{-2/\beta})$, we have 
\eqb \label{eqn-net-ball-count} 
\#\left( \bigcup_{S\in\mcl S_\ep} \mcl M_S \right) \leq \ep^{-\Delta + o_\ep(1) } \quad \text{for $\Delta = \frac{1}{\beta}(4-Q) + 1$} .
\eqe

By the definition of Hausdorff dimension, to prove~\eqref{eqn-net-dim-show} with $\Delta$ as in~\eqref{eqn-net-ball-count}, it suffices to show (continuing to assume~\eqref{eqn-use-bdy-dist} and~\eqref{eqn-net-dim-around}), that for each non-singular point $x\in\BB C$ and $y\in\BB C\cup \{\infty\}$,
\eqb \label{eqn-net-dim-show'} 
\left( \mcl N_{s_2}^y(x) \setminus \mcl N_{s_1}^y(x) \right) \cap \left( U \setminus B_r(\{x,y\}) \right) \subset \bigcup_{S\in\mcl S_\ep} \bigcup_{\mcl B \in \mcl M_S}   \mcl B' ,
\eqe
where $  \mcl B'$ denotes the $D_h$-ball with the same center as $\mcl B$ and twice the radius.
\medskip

\noindent\textit{Step 3: covering the metric net.}
To prove~\eqref{eqn-net-dim-show'}, let $x\in\BB C, y\in \BB C \cup \{\infty\}  , s \in [s_1, D_h(x,y) \wedge s_2]$, and $z\in \bdy\mcl B_s^{y,\bullet}(x) \cap (U\setminus B_r(\{x,y\}) )$. We need to show that $z\in \mcl B'$ for some $\mcl B \in \bigcup_{S\in\mcl S_\ep} \mcl M_S$. Let $P$ be a $D_h$-geodesic from $x$ to $z$. By Lemma~\ref{lem-outer-bdy-dist}, $D_h(x,z) = s$ so $P : [0,s] \rta \BB C$. We have $P(s-\ep) \in \bdy\mcl B_{s-\ep}^{y,\bullet}(x)$. Furthermore, since $x \in U$ and $s\leq s_2$, we have $P(s-\ep) \in \mcl B_{s_2}(U)$, so there exists $S\in \mcl S_\ep$ such that $P(s-\ep) \in S$. 
By~\eqref{eqn-use-bdy-dist}, $P(s-\ep)$ lies at Euclidean distance at least $4\ep^{1/\beta}$ from $z = P(s)$, so $z \notin B_{\ep^{1/\beta}}(S)$. Therefore, the path $\pi_S$ disconnects $P(s-\ep)$ from $z$, so $P$ must cross $\pi_S$ between time $s-\ep$ and time $s$. This implies that there is a point of $\pi_S$ which lies at $D_h$-distance at most $\ep$ from $z$. This point is contained in $\mcl B$ for one of the $\ep$-balls $\mcl B \in \mcl M_S$. Therefore, $z\in\mcl B'$, as required.
\end{proof}

Our proof of Theorem~\ref{thm-net-dim} also yields the following proposition, which is a slightly stronger version of the compactness statement from Theorem~\ref{thm-outer-bdy}. 
 
\begin{prop} \label{prop-ball-bdy-compact}
Almost surely, for each $x\in\BB C$, each $y\in\BB C \cup \{\infty\}$, and each $0 < s_1 < s_2 < D_h(x,y)$, the set $\mcl N_{s_2}^y(x) \setminus \mcl N_{s_1}^y(x)$ is precompact w.r.t.\ $D_h$ (i.e., its $D_h$-closure is compact). 
\end{prop}
\begin{proof}
Let $s_2 > s_1 > 0$, let $U\subset\BB C$ be a Euclidean-bounded open set, and let $r > 0$. 
The proof of Theorem~\ref{thm-net-dim} shows there exists $\Delta > 0$ such that with probability tending to 1 as $\ep\rta 0$, it holds for each $x\in U$ and each $y\in \BB C\cup \{\infty\}$ that the set $\left( \mcl N_{s_2}^y(x) \setminus \mcl N_{s_1}^y(x) \right) \cap \left( U \setminus B_r(\{x,y\}) \right)$ can be covered by $\ep^{-\Delta + o_\ep(1)}$ $D_h$-balls of radius $\ep$.  
Hence, a.s.\ there is a sequence $\ep_k \rta 0$ (depending on $U,r$) such that for each $x,y$ as above and each $k\in\BB N$, this set can be covered by $\ep_k^{-2\Delta  }$ $D_h$-balls of radius $\ep_k$. 

Let $\{U_n\}_{n\in\BB N}$ be an increasing sequence of Euclidean-bounded open sets whose union is all of $\BB C$ and let $\{r_n\}_{n\in\BB N}$ be a sequence of positive radii tending to zero. By the conclusion of the preceding paragraph, a.s.\ for each $n\in\BB N$ there exists a sequence $\ep_{n,k}\rta 0$ such that for each $k\in\BB N$, each $x\in U_n$, and each $y\in \BB C\cup \{\infty\}$, the set $\left( \mcl N_{s_2}^y(x) \setminus \mcl N_{s_1}^y(x) \right) \cap \left( U_n \setminus B_{r_n}(\{x,y\}) \right)$ can be covered by $\ep_{n,k}^{-2\Delta  }$ $D_h$-balls of radius $\ep_{n,k}$.  

If $x$ is a singular point or $y = x$, then $\mcl N_{s_2}^y(x) \setminus \mcl N_{s_1}^y(x) = \emptyset$, so we can assume without loss of generality that $x$ is non-singular and $y\not=x$. 
For each non-singular $x\in\BB C$ and each $y\in \BB C\cup \{\infty\}$ with $y\not=x$, the set $\mcl N_{s_2}^y(x) \setminus \mcl N_{s_1}^y(x)$ is Euclidean-bounded, so both $x$ and this set are contained in $U_n$ for each large enough $n\in\BB N$. By Lemma~\ref{lem-ball-contain-all}, $\mcl B_{s_1}^{y,\bullet}(x)$ contains $B_{r_n}(x)$ for each large enough $n\in\BB N$, which implies that $\mcl N_{s_2}^y(x) \setminus \mcl N_{s_1}^y(x)$ is disjoint from $B_{r_n}(x)$ for each large enough $n\in\BB N$. Furthermore, if $s_2 < D_h(x,y)$ then since $\mcl B_{s_2}^{y,\bullet}(x) \supset \mcl N_{s_2}^y(x) \setminus \mcl N_{s_1}^y(x)$ is Euclidean-closed (Lemma~\ref{lem-ball-closed}) and does not contain $y$, this set lies at positive Euclidean distance from $y$. 

Hence, a.s.\ for each non-singular $x\in\BB C$ and each $y\in \BB C\cup \{\infty\}$ with $0 < s_2 < D_h(x,y)$, the set $\mcl N_{s_2}^y(x) \setminus \mcl N_{s_1}^y(x)$ is contained in $U_n \setminus B_{r_n}(\{x,y\})$ for each large enough $n\in\BB N$. Therefore, as shown above, $\mcl N_{s_2}^y(x) \setminus \mcl N_{s_1}^y(x)$ can be covered by finitely many $D_h$-balls of radius $\ep_{n,k}$ for each large enough $n \in\BB N$ and each large enough $k\in\BB N$. Hence $\mcl N_{s_2}^y(x) \setminus \mcl N_{s_1}^y(x)$ is totally bounded w.r.t.\ $D_h$, hence precompact w.r.t.\ $D_h$. 

This proves the proposition for a deterministic choice of $s_1$ and $s_2$. Every interval $[s_1,s_2] \subset (0,\infty)$ is contained in $[s_1', s_2']$ for some $s_1',s_2' \in \BB Q \cap (0,\infty)$ with $s_1',s_2'$ arbitrarily close to $s_1,s_2$. This gives the proposition statement in general. 
\end{proof}

\section{Outer boundaries of LQG metric balls are Jordan curves}
\label{sec-curve}

The goal of this section is to prove the following proposition, which is the missing ingredient needed to prove Theorem~\ref{thm-outer-bdy}. 

\begin{prop} \label{prop-jordan}
Almost surely, for each non-singular point $x\in\BB C$, each $y\in \BB C\cup \{\infty\}$, and each $s\in (0,D_h(x,y))$, the set $\bdy\mcl B_s^{y,\bullet}(x)$ is a Jordan curve.
\end{prop}

\subsection{A criterion for a domain boundary to be a curve}
\label{sec-jordan-criterion}

In this subsection we will prove a general criterion for the boundary of a planar domain to be a curve. Our criterion will be stated in terms of disconnecting sets, defined as follows. 

\begin{defn} \label{def-disconnect}
Let $X , Y\subset \BB C$ and $A_1,A_2\subset X$. We say that $Y$ \emph{disconnects} $A_2$ from $A_1$ in $X$ if the following is true: $A_1 $ is disjoint from $Y$; and any two points $x \in A_1$ and $y\in A_2 \setminus Y$ lie in different connected components of $X\setminus Y$. 
\end{defn}

We note that by definition $Y$ disconnects any subset of $Y\cap X$ from any subset of $X\setminus Y$.

\begin{prop} \label{prop-locally-connected}
Let $U \subset \BB C$ be a domain containing 0, not all of $\BB C$, such that $\bdy U$ is compact. 
We assume that either $U$ is bounded and simply connected; or  $U$ is unbounded and $U\cup \{\infty\}$ is a simply connected subset of the Riemann sphere.  
Suppose that for each $\ep > 0$, there exists $\delta > 0$ such that each subset of $ U$ which can be disconnected from 0 in $  U$ by a set $Y$ of Euclidean diameter at most $\delta$ with $Y\cap \bdy U\not=\emptyset$ has Euclidean diameter at most $\ep$. 
Then $\bdy U$ is the image of a (not necessarily simple) curve.
\end{prop}

The criterion of Proposition~\ref{prop-locally-connected} is similar in spirit to the concept of $\bdy U$ being locally connected (see, e.g.,~\cite[Section 2.2]{pom-book}), which is a different condition that implies that $\bdy U$ is a curve. The reason why we require that $Y\cap \bdy U \not=\emptyset$ is to rule out the possibility that $Y$ is a small loop surrounding 0, in which case $Y$ would disconnect most of $U$ from 0. 
 
For the proof of Proposition~\ref{prop-locally-connected}, we first need to recall some standard definitions from complex analysis. See, e.g.,~\cite{pom-book} for more detail on these concepts. 
A \emph{crosscut} of a domain $U\subset\BB C$ is a simple curve $C : [0,1] \rta \ol U$ such that $C(0) , C(1) \in \bdy U$ and $C((0,1)) \subset U$. 
If $\bdy U$ is bounded, we define a \emph{null chain} in $U$ to be a sequence of crosscuts $\{C_n\}_{n\in\BB N}$ with the following properties.
\begin{enumerate}[(i)]
\item $ C_n \cap   C_{n+1} =\emptyset$ for each $n\in\BB N$. 
\item $C_n $ disconnects $C_{n+1}$ from $C_1$ in $U$ for each $n\in\BB N$. 
\item As $n\rta\infty$, the Euclidean diameter of $C_n$ converges to zero. 
\end{enumerate}

If $\{C_n\}$ and $\{C_n'\}$ are two null chains, we say that $\{C_n\}$ and $\{C_n'\}$ are \emph{equivalent} if for each large enough $m\in\BB N$, there exists $n\in\BB N$ for which $C_m'$ disconnects $C_n$ from $C_1$ in $U$ and $C_m$ disconnects $C_n'$ from $C_1'$ in $U$. 
A \emph{prime end} of $\bdy U$ is an equivalence class of null chains. 

For a prime end $p$ represented by a null chain $\{C_n\}$, we define $A_p$ to be the intersection over all $n\in\BB N$ of the closure of the set of points in $  U$ which are disconnected from $C_1$ by $C_n$ in $ U$. Then $A_p \subset \bdy U$. We call $A_p$ the \emph{set of points corresponding to $p$}. 

In the next two lemmas, we assume that $U$ is a domain satisfying the hypotheses of Proposition~\ref{prop-locally-connected}. 

\begin{lem} \label{lem-metric-prime-end}
Each prime end of $U$ corresponds to a single point of $\bdy U$.
\end{lem}
\begin{proof}
Let $p $ be a prime end for $U$ and let $\{C_n\}_{n\in\BB N}$ be a null chain corresponding to $p$. By possibly removing finitely many of the $C_n$'s, we can assume without loss of generality that $C_1$ disconnects 0 from $C_n$ for each $n \geq 2$. 
Since the diameter of $C_n$ tends to zero as $n\rta\infty$, our assumption on $U$ implies that the diameter of the set of points in $U$ which are disconnected from 0 in $U$ by $C_n$, hence also its closure, tends to zero as $n\rta\infty$. Hence the (decreasing) intersection of the closures of these sets has diameter zero, so is a single point. 
\end{proof}

In what follows, if $U$ is unbounded we view $\infty$ as a point of $U$, so that by the Riemann mapping theorem there exists a conformal map from the open unit disk $\BB D$ to $U$. 

\begin{lem} \label{lem-metric-cont}
Every conformal map $\phi : \BB D \rta U$ extends to a continuous map $\ol{\BB D} \rta \ol U$. 
\end{lem}
\begin{proof}
By~\cite[Theorem 2.15]{pom-book} there is a bijection $\wh \phi$ from $\bdy \BB D$ to prime ends of $U$ such that for each $u\in\bdy\BB D$ and each null chain $\{C_n\}_{n\in\BB N}$ for the prime end $\wh \phi(u)$, $\{f^{-1}(C_n)\}_{n\in\BB N}$ is a null chain for $u$.
By Lemma~\ref{lem-metric-prime-end}, for each $u\in \bdy\BB D$ the prime end $\wh \phi(u)$ corresponds to a single point of $\bdy U$.
Let $\phi(u)$ be this point. We need to show that $\phi$, thus extended, is continuous.

Obviously, $\phi$ is continuous at each point of $\BB D$, so consider a point $u\in \bdy\BB D$ and a sequence $\{z_k\}_{k\in\BB N}$ in $\ol{\BB D}$ which converges to $u$. 
We will show that $\phi(z_k) \rta \phi(u)$. 

For this purpose let $\ep > 0$ and let $\{C_n\}_{n\in\BB N}$ be a null chain for the prime end $\wh\phi(u)$, as above. 
By possibly removing one of the $C_n$'s, we can assume without loss of generality that $0\notin C_n$ for each $n$. 
By~\cite[Proposition 2.12]{pom-book}, each of the cross cuts $C_n$ separates $U$ into exactly two connected components. 
Let $G_n$ be the one of these connected components which does not contain 0. 
Then $\phi(u) \in \bdy G_n$. 
Since the Euclidean diameter of $C_n$ tends to 0 as $n\rta\infty$, our hypothesis on $U$ implies that the Euclidean diameter of $G_n$, and hence also the Euclidean diameter of $\ol G_n$, tends to 0 as $n\rta\infty$. 
Hence there is some $n_* \in \BB N$ such that for $n\geq n_*$, each point of $\ol G_n$ lies at Euclidean distance at most $\ep$ from $\phi(u)$. 

By the defining property of $\wh\phi$, the sets $\phi^{-1}(C_n)$ are a null chain for the prime end $u\in\bdy\BB D$. 
In particular, each $\phi^{-1}(C_n)$ separates $ \BB D$ into two connected components, namely $\phi^{-1}(G_n)$ and $\phi^{-1}(U \setminus \ol G_n)$. 
Each prime end of $U$ which does not correspond to a point of $\bdy G_n$ corresponds to a point which lies at positive distance from $G_n$, so can be represented by a null chain whose cross cuts (except for their endpoints) are contained in $U \setminus \ol G_n$. 

We claim that $\phi(\ol{\phi^{-1}(G_n)}) \subset \ol G_n$. 
Since $\phi|_{\BB D}$ is a homeomorphism from $\BB D$ to $U$, we have $\phi(\ol{\phi^{-1}(G_n)} \cap \BB D) \subset \ol G_n$.
Now let $w \in \ol{\phi^{-1}(G_n)} \cap \bdy\BB D$ and suppose by way of contradiction that $\phi(w) \notin \ol G_n$. 
By the preceding paragraph there is a null chain $\{\wt C_n\}_{n\in\BB N}$ for $\phi(w)$ whose cross cuts (except for their endpoints) are contained in $U\setminus \ol G_n$.
But, then $\{\phi^{-1}(\wt C_n)\}_{n\in\BB N}$ is a null chain for $w$ whose cross cuts (except for their endpoints) are contained in $\phi^{-1}(U\setminus \ol G_n)$, hence lie at positive distance from $\phi^{-1}(G_n)$. This contradicts the fact that $w\in \ol{\phi^{-1}(G_n)}$, as desired. 
Therefore, $\phi(\ol{\phi^{-1}(G_n)}) \subset \ol G_n$. 

Recall the sequence $z_k\rta z$ from above. For each large enough $k$, $z_k$ is disconnected from $\phi^{-1}(0)$ in $\BB D$ by $\phi^{-1}(C_n)$, so $z_k \in \ol{\phi^{-1}(G_n)}$. 
It therefore follows from the conclusion of the preceding paragraph that for each such $k$, we have $\phi(z_k) \in \ol G_n$ and hence $|\phi(z_k) - \phi(u)| < \ep$. 
Since $\ep$ is arbitrary, this gives the continuity of $\phi$. 
\end{proof}

\begin{proof}[Proof of Proposition~\ref{prop-locally-connected}] 
Lemma~\ref{lem-metric-cont} implies that $\bdy U$ is a curve, since it is the continuous image of $\bdy\BB D$ under $\phi$ (in fact, the statement of Lemma~\ref{lem-metric-cont} is equivalent to the statement that $\bdy U$ is a curve, see~\cite[Theorem 2.1]{pom-book}). 
\end{proof}

\subsection{Proof of Proposition~\ref{prop-jordan}}
\label{sec-jordan-proof}

In this subsection we will use Proposition~\ref{prop-locally-connected} to prove Proposition~\ref{prop-jordan}. Let us first introduce the domain $U$ that we will work with. 
For each non-singular point $x\in\BB C$, each $y\in\BB C \cup \{\infty\}$, and each $s \in (0,D_h(x,y))$, we let 
\eqb \label{eqn-origin-component-def}
U_s^y(x) := \left(\text{connected component of the interior of $\mcl B_s^{y,\bullet}(x)$ which contains $x$} \right) .
\eqe 
By Lemma~\ref{lem-ball-contain-all}, a.s.\ $x$ lies in the interior of $\mcl B_s^{y,\bullet}(x)$ for every $x,y,s$ as above. Hence, almost surely $U_s^y(x)$ is well-defined for every such $x,y,s$. 

Once we show that $\bdy\mcl B_s^{y,\bullet}(x)$ is a Jordan curve, we will get that $U_s^y(x)$ is in fact the only connected component of the interior of $\mcl B_s^{y,\bullet}(x)$. 
However, we do not rule out the possibility that the interior of $\mcl B_s^{y,\bullet}(x)$ is not connected \emph{a priori}. 
The following lemma will allow us to work with $U_s^y(x)$ instead of $\mcl B_s^{y,\bullet}(x)$ throughout the proof of Proposition~\ref{prop-jordan}.

\begin{lem} \label{lem-metric-bdy}
Almost surely, for each non-singular point $x\in\BB C$, each $y\in \BB C\cup \{\infty\}$, and each $s\in (0,D_h(x,y))$, the following is true, with $U_s^y(x)$ as in~\eqref{eqn-origin-component-def}.   
We have $\bdy U_s^y(x) = \bdy \ol U_s^y(x) = \bdy\mcl B_s^{y,\bullet}(x)$ and $U_s^y(x)$ is simply connected. Furthermore, each $D_h$-geodesic from $x$ to a point of $\bdy\mcl B_s^{y,\bullet}(x)$ is contained in $U_s^y(x)$ except for its terminal endpoint.
\end{lem}
\begin{proof}
All of the statements in the proof are required to hold a.s.\ for each $x,y,s$ as in the lemma statement. 
To prove that $\bdy U_s^y(x) = \bdy\mcl B_s^{y,\bullet}(x)$, we first argue that $\bdy U_s^y(x) \subset \bdy\mcl B_s^{y,\bullet}(x)$. Indeed, each $z\in \bdy U_s^y(x)$ is an accumulation point of $U_s^y(x)  \subset \mcl B_s^{y,\bullet}(x)$, so in particular $\bdy U_s^y(x)  \subset \mcl B_s^{y,\bullet}(x)$. Hence it suffices to show that if $z$ is in the interior of $ \mcl B_s^{y,\bullet}(x)$, then $z\notin \bdy U_s^y(x)$. Indeed, for such a $z$ either $z\in U_s^y(x)$ or $z$ belongs to a connected component of the interior of $\mcl B_s^{y,\bullet}(x)$ other than $U_s^y(x)$. In the former case, $z\notin \bdy U_s^y(x)$ since $U_s^y(x)$ is open. In the latter case, $z\notin U_s^y(x)$ since the other connected components of the interior of $\mcl B_s^{y,\bullet}(x)$ are open sets disjoint from $U_s^y(x)$, so they are also disjoint from $\bdy U_s^y(x)$. 

To prove that $\bdy\mcl B_s^{y,\bullet}(x) \subset \bdy U_s^y(x)$, let $z\in \bdy\mcl B_s^{y,\bullet}(x)$. By Lemma~\ref{lem-outer-bdy-dist}, $D_h(x,z) =s$. Let $P : [0,s] \rta \BB C$ be a $D_h$-geodesic from $x$ to $z$. Then $P\subset \mcl B_s^{y,\bullet}(x)$. Furthermore, for $t <s$ we have $D_h(x,P(t)) = t$, so Lemma~\ref{lem-outer-bdy-dist} implies that $P(t) \notin \bdy\mcl B_s^{y,\bullet}(x)$. Therefore, $P([0,s)) $ is contained in the interior of $\mcl B_s^{y,\bullet}(x)$ and hence $P([0,s))\subset U_s^y(x)$. This shows that $z$ is an accumulation point of $U_s^y(x)$, so $z\in \bdy U_s^y(x)$. 

We have shown that $\bdy\mcl B_s^{y,\bullet}(x) = \bdy U_s^y(x)$. 
Since $U_s^y(x) \subset \ol U_s^y(x) \subset \mcl B_s^{y,\bullet}(x)$, we also have $\bdy\ol U_s^y(x) = \bdy\mcl B_s^{y,\bullet}(x)$. 
 
The argument in the second paragraph of the proof shows that each $D_h$-geodesic from $x$ to a point of $\bdy\mcl B_s^{y,\bullet}(x) = \bdy U_s^y(x)$ is contained in $  U_s^y(x)$ except for its terminal endpoint.  
 
Since $U_s^y(x)$ is connected, to show that $U_s^y(x)$ is simply connected, it suffices to show that $\BB C\setminus U_s^y(x)$ is connected. 
Let $\mcl V$ be the set of connected components of the interior of $\mcl B_s^{y,\bullet}(x)$ other than $U_s^y(x)$ (we will eventually show that $\mcl V =\emptyset$, but we do not know this yet). 
We can write $\BB C\setminus U_s^y(x)$ as the union of $\ol{\BB C \setminus \mcl B_s^{y,\bullet}(x)}$ and the sets $\ol V$ for $V\in\mcl V$. Each of the sets $\ol{\BB C \setminus \mcl B_s^{y,\bullet}(x)}$ and $\ol V$ is the closure of a connected set, so is connected. Furthermore, each $\bdy V$ for $V\subset\mcl V$ is contained in $\bdy\mcl B_s^{y,\bullet}(x)$ (by the same argument that we used for $U_s^y(x)$ above), which in turn is contained in $\ol{\BB C \setminus \mcl B_s^{y,\bullet}(x)}$. Hence $\BB C\setminus U_s^y(x)$ is the union of connected sets which all intersect a common connected set, so is connected.
\end{proof}

The set $U_s^y(x)$ contains points $z$ with $D_h(x,z)  > s$. For such points $z$, it is possible that a $D_h$-geodesic from $x$ to $z$ intersects $\bdy U_s^y(x)$. Since we will be interested in sets which are disconnected from $x$ in $U_s^y(x)$ (c.f.\ Proposition~\ref{prop-locally-connected}), it is important for us to work with paths which are contained in $U_s^y(x)$. The following lemma will allow us to do so.

\begin{figure}[ht!]
\begin{center}
\includegraphics[scale=1]{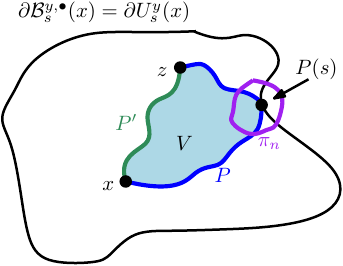} 
\caption{\label{fig-approx-path} Illustration of the proof of Lemma~\ref{lem-approx-path}. The path $P$ is a $D_h$-geodesic from $x$ to $z$. If $P(s) \in \bdy U_s^y(x)$, we replace a segment of $P$ by a segment of the small loop $\pi$ to get a path from $x$ to $z$ which is contained in $U_s^y(x)$ and which is not much longer than $P$. 
}
\end{center}
\end{figure}

\begin{lem} \label{lem-approx-path}
Almost surely, for each non-singular point $x\in\BB C$, each $y\in \BB C\cup \{\infty\}$, and each $s\in (0,D_h(x,y))$, the following is true.
For each $z \in U_s^y(x)$ and each $\delta > 0$, there is a path in $U_s^y(x)$ from $x$ to $z$ with $D_h$-length at most $D_h(x,z) + \delta$.
\end{lem}
\begin{proof}
See Figure~\ref{fig-approx-path} for an illustration of the proof.
The statement is vacuous if $z$ is a singular point (i.e., $D_h(x,z) = \infty$), so assume that $z$ is not a singular point. 

Let $P$ be a $D_h$-geodesic from $x$ to $z$. If $ t:= D_h(x,z)  \leq s$, then $P \subset \mcl B_s(x) \subset \mcl B_s^{y,\bullet}(x)$. Furthermore, since $D_h(x,w) = s \geq t$ for each $w\in \bdy\mcl B_s^{y,\bullet}(x)$ (Lemma~\ref{lem-outer-bdy-dist}), $P([0,t))$ is contained in the interior of $\mcl B_s^{y,\bullet}(x)$. Since $P(t) = z \in U_s^y(x)$, we get that $P$ is contained in the interior of $\mcl B_s^{y,\bullet}(x)$. By the definition~\eqref{eqn-origin-component-def} of $U_s^y(x)$, it follows that $P\subset U_s^y(x)$. 

Hence we only need to treat the case when $D_h(x,z)  > s$. By Lemma~\ref{lem-outer-bdy-dist}, the path $P$ can hit $\bdy\mcl B_s^{y,\bullet}(x)$ at most once, namely at time $s$. Consequently, $P$ cannot exit and subsequently re-enter $\mcl B_s^{y,\bullet}(x)$, so $P\subset \mcl B_s^{y,\bullet}(x)$. Furthermore, $P(t)$ is contained in the interior of $\mcl B_s^{y,\bullet}(x)$ for each $t\not=s$.

If $P(s) \notin \bdy\mcl B_s^{y,\bullet}(x)$, then we are done so we can assume without loss of generality that $P(s) \in \bdy\mcl B_s^{y,\bullet}(x)$. 
Since $U_s^y(x)$ is open and connected, hence path connected, we can find a simple path $P'$ in $U_s^y(x)$ from $x$ to $z$ (we make no assumption on the $D_h$-length of $P'$).
In fact, since $U_s^y(x)$ is homeomorphic to the disk and $P$ is a simple path, we can arrange that $P'$ does not intersect $P$ except at $x$ and $z$. 
Let $V$ be the unique bounded complementary connected component of $\BB C\setminus (P\cup P')$. Then $V$ is a Jordan domain and $P(s) \in \bdy V$. Furthermore, $\bdy V \subset \mcl B_s^{y,\bullet}(x)$, so each point of $V$ is disconnected from $y$ by $\mcl B_s^{y,\bullet}(x)$. Hence $\ol V\subset \mcl B_s^{y,\bullet}(x)$. In fact, $\bdy V\setminus \{P(s)\}$ is contained in the interior of $\mcl B_s^{y,\bullet}(x)$, so it follows that $\ol V\setminus \{P(s)\}$ is contained in the interior of $\mcl B_s^{y,\bullet}(x)$. Since $\ol V \setminus \{P(s)\}$ is connected and contains $x$ it follows that $\ol V \setminus \{P(s)\} \subset U_s^y(x)$. 

By Lemma~\ref{lem-separating-annuli}, there is a sequence of disjoint $D_h$-continuous loops $\{\pi_n\}_{n\in\BB N}$, each of which separates a neighborhood of $P(s)$ from $\infty$, such that the Euclidean radius of $\pi_n$, the $D_h$-length of $\pi_n$, and the $D_h$-distance from $z$ to $\pi_n$ each tend to zero as $n\rta\infty$.  
If $n\in\BB N$ is chosen to be sufficiently large, then $\pi_n$ is disjoint from $P'$, the $D_h$-length of $\pi_n$ is at most $\delta$, and there is a segment $\eta_n$ of $\pi_n$ which is a crosscut of $V$ (i.e., it is contained in $V$ except for its endpoints). The segment $\eta_n$ joins $P(t_1)$ to $P(t_2)$ for some $t_1 < s < t_2$. Let $\wt P$ be the concatenation of $P|_{[0,t_1]}$, $\eta_n$, and $P|_{[t_2,D_h(x,z)]}$. Then $\wt P$ is a path in $\ol V\setminus \{P(s)\}$ from $x$ to $z$ with $D_h$-length at most $D_h(x,z) + \delta$. By the preceding paragraph, $\wt P$ is contained in $U_s^y(x)$.
\end{proof}

We will now check the criterion of Proposition~\ref{prop-locally-connected} for the domain $U_s^y(x)$.

\begin{figure}[ht!]
\begin{center}
\includegraphics[scale=1]{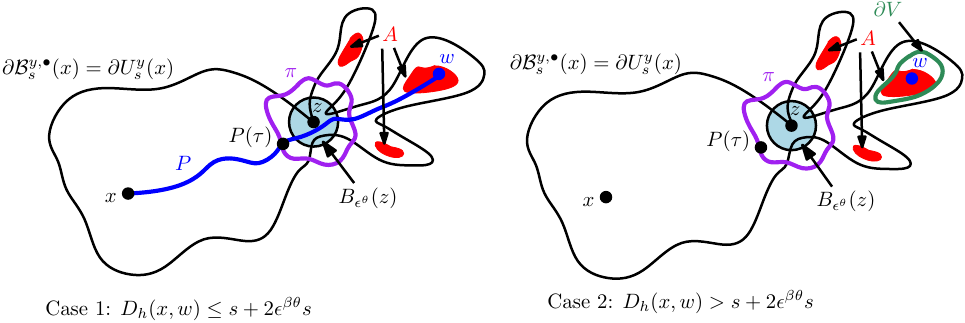} 
\caption{\label{fig-metric-disc} Illustration of the proof of Lemma~\ref{lem-metric-disc}. The red set $A\subset U_s^y(x)$ is disconnected from $x$ in $ U_s^y(x)$ by the Euclidean ball $\ol{B_{\ep^\theta}(z)}$. Using Corollary~\ref{cor-hit-ball'}, we produce a path $\pi$ (purple) disconnecting the inner and outer boundaries of the annulus $\BB A_{\ep^\theta , \ep^{\theta/2}}(z)$ with $D_h$-length at most $\ep^{\beta\theta} s $. We seek to bound the Euclidean distance from a point $w \in A$ to $\ol{B_{\ep^\theta}(z)}$.  The left panel shows the case when $D_h(x,w) \leq s + 2\ep^{\beta\theta}s$. In this case, Lemma~\ref{lem-approx-path} gives a path $P$ from $x$ to $w$ whose $D_h$-length is at most $s+3\ep^{\beta\theta}s$. The path $P$ must hit $\pi$, say at a time $\tau $. Our upper bound for the $D_h$-length of $\pi$ shows that $D_h(x,w) -\tau \leq 4\ep^{\beta\theta} s$. Using a H\"older continuity bound for the Euclidean metric w.r.t.\ $D_h$ (Proposition~\ref{prop-holder}), we then obtain an upper bound for the Euclidean diameter of $P([\tau,D_h(x,w)])$, which then gives an upper bound for the Euclidean distance from $w$ to $B_{\ep^\theta}(z)$.  
The right panel shows the case when $D_h(x,w) > s + 2\ep^{\beta\theta} s$. In this case, we consider the complementary connected component $V$ of the ball $\mcl B_{s+2\ep^{\beta\theta} s}(x)$ with $w\in V$. We bound the $D_h$-distance from $x$ to $\ol{B_{\ep^\theta}(z)}$ in terms of $\sup_{u \in \bdy V} \op{dist}(u , \ol{B_{\ep^\theta}(z)} )$, then bound this last quantity using the previous case. 
}
\end{center}
\end{figure}

\begin{lem} \label{lem-metric-disc}
There exists $\theta > 1$ such that a.s.\ for each non-singular point $x\in \BB C$, each $y\in \BB C\cup \{\infty\}$, and each $s \in (0,D_h(x,y))$, there exists a random $\ol\ep = \ol\ep(x,y,s)>0$ with the following property. For each $\ep \in (0,\ol\ep]$, each set $A\subset  U_s^y(x)$ which can be disconnected from $x$ in $ U_s^y(x)$ by a set $Y$ of Euclidean diameter at most $\ep^\theta$ which intersects $\bdy U_s^y(x)$ has Euclidean diameter at most $\ep$. 
\end{lem} 
\begin{proof} 
See Figure~\ref{fig-metric-disc} for an illustration of the proof.  

Let $\theta > 1$ to be chosen later. 
We will first state some estimates which hold a.s.\ for each $x,y,s$ as in the lemma statement and each small enough $\ep > 0$ (depending on $x,y,s$). We will then truncate on the event that these estimates are satisfied and show that the conclusion of the lemma statement is satisfied.

Almost surely, for each $s > 0$ and each $x\in\BB C$, the $D_h$-ball $\mcl B_s(x)$ is  Euclidean-bounded, so its Euclidean 1-neighborhood $B_1(\mcl B_s(x))$ is also Euclidean-bounded. 
We note that this latter set contains every point which lies at Euclidean distance less than 1 from $\bdy\mcl B_s^{y,\bullet}(x) = \bdy U_s^y(x)$ (see Lemma~\ref{lem-metric-bdy}) for each $y\in \BB C\cup \{\infty\}$. 

If we let $\{V_n\}_{n\in\BB N}$ be an increasing sequence of Euclidean-bounded open sets whose union is all of $\BB C$, then a.s.\ for each $x\in\BB C$ and each $s > 0$ we have $B_1(\mcl B_s(x)) \subset V_n$ for large enough $n\in\BB N$. 
Furthermore, if $y\in \BB C\cup \{\infty\}$ and $s \in (0,D_h(x,y))$, then $x$ and $y$ each lie at positive Euclidean distance from $\bdy\mcl B_s^{y,\bullet}(x)$ (see Lemma~\ref{lem-ball-contain-all}). 
We may therefore apply Corollary~\ref{cor-hit-ball'} (with $\alpha=1/2$, $\ep^\theta$ instead of $\ep$, and $U=V_n$), then send $n\rta \infty$, to get that there exists $\beta > 0$ such that a.s.\ for each $x,y,s$ as in the lemma statement, it holds for small enough $\ep  > 0$ that 
\eqb \label{eqn-metric-disc-cont}
D_h\left( \text{around $\BB A_{ \ep^\theta , \ep^{\theta/2}}(z) $} \right) 
\leq \ep^{\beta\theta} s ,\quad \forall z \in \bdy\mcl B_s^{y,\bullet}(x).
\eqe 

By Proposition~\ref{prop-holder} (again applied to each of the sets $V_n$ above), if $\chi \in (0, (\xi(Q+2))^{-1}) $, then for each $x \in \BB C$ and each $s > 0$, it is a.s.\ the case that for each small enough $\ep > 0$, 
\eqb \label{eqn-metric-disc-coer}
\text{$\forall z,w\in B_1(\mcl B_s(x))$ with $D_h(z,w) \leq 4\ep^{\beta\theta} s $, we have $|z-w| \leq  \ep^{\chi \beta \theta} s^\chi$}. 
\eqe  

By Lemma~\ref{lem-ball-contain-all}, it is a.s.\ the case that for each $x,y,s$ as in the lemma statement and each small enough $\ep  > 0$,  
\eqb \label{eqn-metric-disc-origin}
B_{4\ep^{\theta/2}}(x) \subset \mcl B_s^{y,\bullet}(x) .
\eqe
By the definition~\eqref{eqn-origin-component-def} of $U_s^y(x)$, we see that~\eqref{eqn-metric-disc-origin} implies that also $B_{4\ep^{\theta/2}}(x) \subset U_s^y(x)$. 
 
We henceforth work on the full-probability event that for each $x,y,s$ as in the lemma statement,~\eqref{eqn-metric-disc-cont}, \eqref{eqn-metric-disc-coer}, and~\eqref{eqn-metric-disc-origin} all hold each small enough $\ep > 0$. We will show that the lemma statement holds provided $\theta > \max\{2 ,  1/(\beta \chi) \}$. 
To see this, let $x,y,s$ be as in the lemma statement, assume that $\ep > 0$ is small enough enough that the above three estimates hold. 
Let $A\subset  U_s^y(x)$ be a set which can be disconnected from $x$ in $  U_s^y(x)$ by a set $Y$ of Euclidean diameter at most $\ep^\theta $ which intersects $\bdy U_s^y(x)$.
We claim that the Euclidean diameter of $A$ is at most $\ep$.  

Choose $z \in Y\cap \bdy U_s^y(x)$. Then $Y\subset \ol{B_{\ep^\theta}(z)}$ so $A$ is disconnected from $x$ in $  U_s^y(x)$ by $\ol{B_{\ep^\theta}(z)}$. 
We can assume without loss of generality that $A \not\subset \ol{B_{\ep^{\theta/2}}(z)}$ (otherwise, the Euclidean diameter of $A$ is at most $\ep^{\theta/2} < \ep$). Furthermore, we have $z\in \bdy U_s^y(x)  \subset B_1(\mcl B_s(x))$ and by~\eqref{eqn-metric-disc-origin}, the Euclidean distance from $x$ to $\bdy U_s^y(x) = \bdy\mcl B_s^{y,\bullet}(x)$ is at least $4\ep^{\theta/2}$, so $x\notin B_{\ep^{\theta/2}}(z)$. 

The estimate~\eqref{eqn-metric-disc-cont} implies that there is a path $\pi$ in $\BB A_{\ep^\theta ,  \ep^{\theta/2}}(z)$ which disconnects the inner and outer boundaries of $\BB A_{\ep^\theta ,  \ep^{\theta/2}}(z)  $ and satisfies
\eqb \label{eqn-metric-disc-path}
\left(\text{$D_h$-length of $\pi$} \right)   \leq \ep^{\beta \theta} s .
\eqe
Since $z\in \bdy\mcl B_s^{y,\bullet}(x)$, the path $\pi$ intersects $\bdy U_s^y(x) = \bdy\mcl B_s^{y,\bullet}(x)$.
So, 
\eqb \label{eqn-metric-disc-path-dist}
D_h(u,\bdy\mcl B_s^{y,\bullet}(x)) \leq \ep^{\beta\theta} s ,\quad\forall u \in \pi .
\eqe
 
We claim that for each $w\in A \setminus \ol{B_{\ep^{\theta/2 }}(z)}$, 
\eqb \label{eqn-metric-disc-claim}
\op{dist}\left( w ,   \ol{B_{\ep^{\theta/2}}(z)} \right) \leq \ep/4 ,
\eqe
where $\op{dist}$ denotes Euclidean distance.
Once~\eqref{eqn-metric-disc-claim} is established, we will obtain that the Euclidean diameter of $A $ is at most $\ep/2 + 2\ep^{\theta/2} \leq \ep $, as desired.
To prove~\eqref{eqn-metric-disc-claim}, we treat two cases depending on the value of $D_h(x,w)$. 
\medskip

\noindent\textit{Case 1: $D_h(x,w) \leq s + 2\ep^{\beta\theta} s$.} 
By Lemma~\ref{lem-approx-path}, there is a path $P  $ from $x$ to $w$ in $U_s^y(x)$ with $D_h$-length $T \leq D_h(x,w) + \ep^{\beta\theta} s$. We take $P$ to be parametrized by its $D_h$-length.
By our choice of $z$, $P$ passes through $\ol{B_{\ep^\theta}(z)}$. 
Since $x,w  \notin B_{\ep^{\theta/2}}(z)$, $P$ must hit the path $\pi$. Let $\tau$ be the first time that $P$ hits $\pi$.  
By~\eqref{eqn-metric-disc-path-dist}, $D_h( P(\tau) , \bdy\mcl B_s^{y,\bullet}(x) ) \leq  \ep^{\beta\theta} s$.

Since each point of $\bdy\mcl B_s^{y,\bullet}(x)$ lies at $D_h$-distance $s$ from $x$ (Lemma~\ref{lem-outer-bdy-dist}), this implies that $\tau \geq s - \ep^{\beta\theta} s$ and hence that 
\eqbn
T - \tau \leq T - s + \ep^{\beta\theta} s \leq D_h(x,w) - s + 2\ep^{\beta\theta} s \leq 4\ep^{\beta\theta} s .
\eqen 
By~\eqref{eqn-metric-disc-coer}, if we let 
\eqbn
\sigma := s \wedge \inf\left\{ t > \tau : P(t) \notin B_1(\bdy\mcl B_s^{y,\bullet}(x))    \right\}
\eqen
then the Euclidean diameter of $P([\tau,\sigma])$ is at most $\ep^{\chi\beta\theta} s^\chi $, which by our choice of $\theta$ is at most $\ep/4$ (provided $\ep$ is small enough). 

Since $P(\tau) \in \pi$ and $\pi \subset B_{\ep^{\theta/2}}(\bdy\mcl B_s^{y,\bullet}(x))$, each point of $P([\tau,\sigma])$ lies at Euclidean distance at most $\ep^{\theta/2} + \ep/4 < 1$ from $\bdy\mcl B_s^{y,\bullet}(x)$. 
Therefore, $\sigma = D_h(x,w)$ and $P(\sigma) = w$. 
Hence $w$ lies at Euclidean distance at most $\ep/4$ from $B_{\ep^{\theta/2}}(z)$, as required.
\medskip

\noindent\textit{Case 2: $D_h(x,w) > s + 2\ep^{\beta\theta} s$.} Let $V$ be the connected component of $\BB C\setminus \mcl B_{s +2 \ep^{\beta\theta} s} (x)$ which contains $w$. 
Then $V$ is contained in $\BB C\setminus \mcl B_s^{w,\bullet}(x)$, which is the connected component of $\BB C \setminus \mcl B_s(x)$ which contains $w$. 
By Lemmas~\ref{lem-ball-closed} and~\ref{lem-outer-bdy-dist}, $\bdy V = \bdy\mcl B_{s + 2\ep^{\beta\theta} s}^{w,\bullet}(x)$ lies at positive Euclidean distance from $\bdy\mcl B_{s  }^{w,\bullet}(x)$ and hence also from $\mcl B_s(x)$. It follows that $V$ lies at positive Euclidean distance from $\bdy \mcl B_s^{y,\bullet}(x)$, so $\ol V $ is contained in the interior of $\mcl B_s^{y,\bullet}(x)$. Since $\ol V$ is connected and $w \in U_s^y(x)$, we have $\ol V\subset U_s^y(x)$. 

We claim that $\ol V$ is disjoint from $\ol{B_{\ep^\theta}(z)}$. Indeed, if $\ol V$ intersects $\ol{B_{\ep^\theta}(z)}$, then since $w\in \ol V \setminus \ol{B_{\ep^{\theta/2}}(z)}$ and $\ol V$ is connected, it must be the case that $\ol V$ intersects the inner and outer boundaries of the annulus $\BB A_{\ep^\theta ,  \ep^{\theta/2}}(z)  $. Hence $\ol V$ intersects $\pi$. By~\eqref{eqn-metric-disc-path-dist}, the $D_h$-distance from $\ol V$ to $\bdy\mcl B_s^{y,\bullet}(x)$ is at most $\ep^{\beta\theta}s$, so the $D_h$-distance from $\ol V$ to $x$ is at most $s + \ep^{\beta\theta} s$. But, by Lemma~\ref{lem-outer-bdy-dist} (applied to $\bdy\mcl B_{s+2\ep^{\beta\theta} s}^{w,\bullet}(x)$), the $D_h$-distance from $\ol V$ to $x$ is  equal to $s+2\ep^{\beta\theta} s$, which is a contradiction. 

Since $\ol V$ is connected, $w\in \ol V$, and $w$ is disconnected from $x$ by $\ol{B_{\ep^\theta}(z)}$ in $U_s^y(x)$, we get that $\ol V$ is disconnected from $x$ by $\ol{B_{\ep^{\theta}}(z)}$ in $U_s^y(x)$. 

Let $V_\infty = \BB C\setminus  \mcl B_{s+2\ep^{\beta\theta}s}^\bullet(x)$ be the unbounded connected component of $\BB C\setminus \mcl B_{s+2\ep^{\beta\theta}s}(x)$.
We will now reduce to the case when $V\not=V_\infty$. See Figure~\ref{fig-Vinfty} for an illustration of this part of the argument. Obviously, if $V_\infty\not\subset U_s^y(x)$, then $V\not= V_\infty$, so we can assume that $V_\infty \subset U_s^y(x)$ (which implies that $U_s^y(x)$ is unbounded). 
Then $\BB C\setminus \mcl B_{2s}^\bullet(x) \subset V_\infty  \subset U_s^y(x)$. We can choose a path $\Pi$ in $U_s^y(x)$ from $x$ to a point of $\BB C\setminus  \mcl B_{2s}^\bullet(x)$ in a manner which depends only on $U_s^y(x)$ and $\BB C\setminus \mcl B_{2s}^\bullet(x)$ (not on $\ep$). Let $\ep_0$ be the Euclidean distance from $\Pi$ to $\bdy U_s^y(x)$. Then $\ep_0$ is a random number depending on $x,y,s$ (not on $\ep$). If $\ep^{\theta } < \ep_0$ then the path $\Pi$ cannot intersect $\ol{B_{\ep^\theta}(z)}$. Hence if $\ep^\theta < \ep_0$, then $\BB C\setminus  \mcl B_{2s}^\bullet(x)$ is not disconnected from $x$ in $U_s^y(x)$ by $\ol{B_{\ep^\theta}(z)}$. Since $\BB C\setminus  \mcl B_{2s}^\bullet(x) \subset V_\infty$ and $V$ is disconnected from $x$ in $U_s^y(x)$ by $\ol{B_{\ep^\theta}(z)}$ (as explained above), this implies that so long as $\ep < \ep_0$, we have $V\not=V_\infty$. We henceforth assume that $\ep < \ep_0$, so that $\ol V$ is compact. 

The Euclidean-furthest point of $\ol V$ from $\ol{B_{\ep^\theta}(z)}$ must lie on $\bdy V$, so since $w \in \ol V$ we have
\eqbn
\op{dist}\left(w , \ol{B_{\ep^\theta}(z)} \right) \leq \sup_{u \in \bdy V} \op{dist}(u , \ol{B_{\ep^\theta}(z)} ) .
\eqen
Each point of $\bdy V$ lies at $D_h$-distance $s + 2\ep^{\beta\theta} s$ from $x$, so we can apply Case 1 with $\bdy V$ in place of $A$ to get that $\sup_{u \in \bdy V} \op{dist}(u , \ol{B_{\ep^\theta}(z)} ) \leq \ep/4$. This yields~\eqref{eqn-metric-disc-claim}. 
\end{proof}

\begin{figure}[ht!]
\begin{center}
\includegraphics[scale=1]{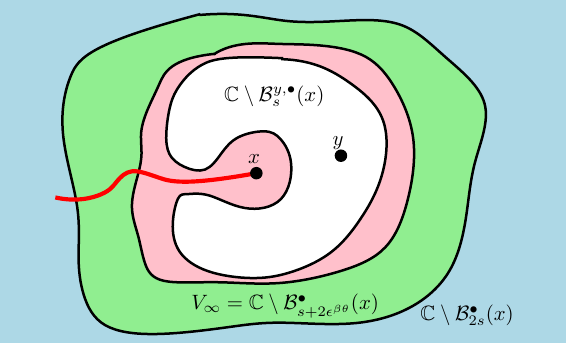} 
\caption{\label{fig-Vinfty} Illustration of how we reduce to the case when $V\not=V_\infty$ in the proof of Lemma~\ref{lem-metric-disc}. Here we have shown the case when $U_s^y(x)$ (the union of the pink, green, and blue regions) is unbounded. The set $V_\infty$ is the union of the green and blue regions. We can choose a path $\Pi$ (red) in $U_s^y(x)$ from $x$ to the blue region $\BB C\setminus \mcl B_{2s}^\bullet(x)$ in a manner which does not depend on $\ep$. The Euclidean distance from $\Pi$ to $\bdy U_s^y(x)$ is a positive constant $\ep_0 > 0$ which does not depend on $\ep$. Hence, if $\ep^\theta < \ep_0$ then $\BB C\setminus \mcl B_{2s}^\bullet(x)$ is not disconnected from $x$ in $U_s^y(x)$ by $\ol{B_{\ep^\theta}(z)}$. Since $\BB C\setminus \mcl B_{2s}^\bullet(x) \subset V_\infty$, the same is true for $V_\infty$. Since $V$ is disconnected from $x$ in $U_s^y(x)$ by $\ol{B_{\ep^\theta}(z)}$, we infer that if $\ep^\theta < \ep_0$, then $V \not= V_\infty$.  
}
\end{center}
\end{figure}

We can now apply Proposition~\ref{prop-locally-connected} to get the following.

\begin{lem} \label{lem-curve}
Almost surely, for each non-singular point $x\in\BB C$, each $y\in\BB C \cup \{\infty\}$, and each $s \in (0,D_h(x,y))$ the set $\bdy\mcl B_s^{y,\bullet}(x)$ is the image of a (not necessarily simple) curve. 
\end{lem}
\begin{proof} 
By Lemma~\ref{lem-metric-bdy}, $\bdy \mcl B_s^{y,\bullet}(x) = \bdy U_s^y(x)$, so it suffices to show that $\bdy U_s^y(x)$ is a curve. 
By Lemma~\ref{lem-metric-bdy}, it is a.s.\ the case that for each $x,y,s$ as in the lemma statement, the set $U_s^y(x)$ is simply connected. Furthermore, $\bdy U_s^y(x) \subset \bdy\mcl B_s(x)$ is Euclidean-compact. 
By Proposition~\ref{prop-locally-connected}, to show that $\bdy U_s^y(x)$ is a curve it therefore suffices to show that for each $\ep > 0$, there exists $\delta > 0$ such that each set $A\subset   U_s^y(x)$ which can be disconnected from $x$ in $  U_s^y(x)$ by a set of Euclidean diameter at most $\delta$ which intersects $\bdy U_s^y(x)$ has Euclidean diameter at most $\ep$. 
This follows from Lemma~\ref{lem-metric-disc}. 
\end{proof}

To prove that $\bdy\mcl B_s^{y,\bullet}(x)$ is a \emph{Jordan} curve, we need to prove that it can be represented by a curve with no double points. 
The following lemma will help us to do that.

\begin{lem} \label{lem-metric-multiple}
Almost surely, for each non-singular point $x\in\BB C$, each $y\in\BB C \cup \{\infty\}$, and each $s \in (0,D_h(x,y))$, the following is true. 
Let $\psi : \bdy\BB D \rta \bdy\mcl B_s^{y,\bullet}(x)$ be a continuous map (which exists by Lemma~\ref{lem-curve}).
Let $u,v \in \bdy\BB D$ be distinct points such that $\psi(u) = \psi(v)$. 
Let $I$ and $J$ be the two closed arcs of $\bdy\BB D$ between $u$ and $v$. 
Then either $\psi(I)\subset \psi(J)$ or $\psi(J)\subset \psi(I)$. 
\end{lem}
\begin{proof} 
Since $\bdy\mcl B_s^{y,\bullet}(x)$ disconnects $x$ from $y$, the homotopy class of the loop $\psi$ in $(\BB C\cup \{\infty\}) \setminus \{x,y\}$ is non-trivial.  
Since $\psi(u) = \psi(v)$, each of $\psi|_I$ and $\psi|_J$ is a loop in $\BB C$, and $\psi|_{\bdy\BB D}$ is the concatenation of these two loops. 
The concatenation of two homotopically trivial loops is also homotopically trivial.
Therefore, one of $\psi|_I$ or $\psi|_J$ is not homotopic to a point in $(\BB C\cup \{\infty\}) \setminus \{x,y\}$. 
This implies that one of $\psi(I)$ or $\psi(J)$ disconnects $x$ from $y$.   
 
Assume without loss of generality that $\psi(I)$ disconnects $x$ from $y$.  
Since $\psi(I) , \psi(J) \subset \bdy \mcl B_s^{y,\bullet}(x)$, no point of $\psi(J) \setminus \psi(I)$ can be disconnected from $y$ by $\psi(I)$. 
Hence, $\psi(J) \subset \ol O$, where $O$ is the connected component of $\BB C\setminus \psi(I)$ which contains $y$. 
By assumption, $x \notin O$. 

If $z \in \psi(J) \setminus \psi(I)$ then $z\in O  \cap \bdy\mcl B_s^{y,\bullet}(x)$. By Lemma~\ref{lem-outer-bdy-dist}, we have $D_h(x,z) = s$. If $P : [0,s] \rta \BB C$ is a $D_h$-geodesic from $x$ to $z$, then since $z\in O$ there is a time $\tau < s$ such that $P(\tau) \in \psi(I)$. But, $\psi(I)\subset\bdy\mcl B_s^{y,\bullet}(x)$ so by Lemma~\ref{lem-outer-bdy-dist}, $D_h(x,P(\tau)) = s$. This contradicts the fact that $P$ is a $D_h$-geodesic. We conclude that $\psi(J) \subset \psi(I)$. 
\end{proof}

\begin{proof}[Proof of Proposition~\ref{prop-jordan}]
Throughout the proof, we fix a non-singular point $x\in\BB C$, a point $y \in \BB C\cup\{\infty\}$, and $s\in (0,D_h(x,y))$. All statements are required to hold a.s.\ for all choices of $x,y,s$ simultaneously.

Let $U_s^y(x)$ be as in~\eqref{eqn-origin-component-def} and let $\phi = \phi_s^y : \BB D\rta U_s^y(x)$ be a conformal map (such a map exists since $U_s^y(x)$ is simply connected, see Lemma~\ref{lem-metric-bdy}). Since $\bdy U_s^y(x) = \bdy\mcl B_s^{y,\bullet}(x)$ (Lemma~\ref{lem-metric-bdy}), it follows from Lemma~\ref{lem-curve} and~\cite[Theorem 2.1]{pom-book} (or just Lemma~\ref{lem-metric-cont}) that the map $\phi$ extends to a continuous map $\ol{\BB D} \rta \ol U_s^y(x)$. We henceforth assume that $\phi$ has been replaced by such a continuous extension. We will show that $\phi$, thus extended, is a homeomorphism.

We say that $z\in\bdy\mcl B_s^{y,\bullet}(x) =\bdy U_s^y(x)$ is a \emph{cut point} if $\bdy U_s^y(x) \setminus \{z\}$ is not connected. By~\cite[Theorem 2.6]{pom-book}, it suffices to show that $\bdy U_s^y(x)$ has no cut points. 
 
Assume by way of contradiction that $z \in \bdy U_s^y(x)$ is a cut point. By~\cite[Proposition 2.5]{pom-book}, $\# \phi^{-1}(z) \geq 2$ (in principle $\#\phi^{-1}(z) $ could be infinite, even uncountable).
Furthermore, if $\mcl I$ is the set of connected components of $\bdy\BB D\setminus \phi^{-1}(z)$, then the set of connected components of $\bdy U_s^y(x) \setminus \{z\}$ is $\{\phi(I) : I\in \mcl I\}$. 
 
Each $I\in\mcl I$ is an open arc of $\bdy\BB D$ whose endpoints are distinct points of $\phi^{-1}(z)$.
Let $J  = J(I) := \bdy\BB D\setminus I$. 
By Lemma~\ref{lem-metric-multiple}, either $\phi(\ol I)\subset \phi(J)$ or $\phi(J)\subset\phi(\ol I)$. 
By the preceding paragraph, $\phi(\ol I)$ is the union of $\{z\}$ and a connected component of $\bdy U_s^y(x) \setminus \{z\}$; and $\phi(J)$ is the union of $\{z\}$ and the other connected components of $\bdy U_s^y(x) \setminus\{z\}$. 
Therefore, $\phi(\ol I)\cap \phi(J) = \{z\}$. 
Hence one of $\phi(\ol I)$ or $\phi(J)$ is equal to $\{z\}$. 
This means that $\bdy U_s^y(x) \setminus \{z\}$ has only one connected component, so $z$ was not a cut point after all.
\end{proof}

\begin{proof}[Proof of Theorem~\ref{thm-outer-bdy}]
 Proposition~\ref{prop-jordan} implies that a.s.\ $\bdy\mcl B_s^{y,\bullet}(x)$ is a Jordan curve for each non-singular $x\in\BB C$, each $y \in \BB C\cup\{\infty\}$, and each $s \in (0,D_h(x,y))$. 
Theorem~\ref{thm-net-dim} implies that a.s.\ each of these filled metric ball boundaries has finite Hausdorff dimension w.r.t.\ $D_h$. 
Proposition~\ref{prop-ball-bdy-compact} implies that a.s.\ each of these filled metric ball boundaries is pre-compact w.r.t.\ $D_h$. Since each such boundary is Euclidean-closed, it is also $D_h$-closed, and hence $D_h$-compact. 
\end{proof}

\section{Confluence of geodesics}
\label{sec-confluence}

In this section we will explain how to adapt the proof of confluence of geodesics for subcritical LQG from~\cite{gm-confluence} to the supercritical case. This will lead to proofs of Theorems~\ref{thm-clsce} and~\ref{thm-finite-geo0}. Many of the arguments of~\cite{gm-confluence} carry over verbatim to the supercritical case, but in some places non-trivial modifications to the arguments, using results from Sections~\ref{sec-bdy-estimates} and~\ref{sec-curve} of the present paper, are needed. As such, we will not repeat the full argument from~\cite{gm-confluence}. Instead, we will only explain the parts of the argument which require modification. We aim to strike a balance between minimizing repetition of arguments from the subcritical case and making the paper readable without the reader having to frequently refer to~\cite{gm-confluence}.   

The proof of confluence of geodesics for subcritical LQG has four steps.
\begin{enumerate}
\item Establish some preliminary facts about geodesics, such as uniqueness of geodesics between typical points and certain monotonicity properties for the cyclic ordering of geodesics from 0 to points of the boundary of the filled metric ball $\mcl B_s^\bullet$~\cite[Section 2.1]{gm-confluence}. \label{item-step-unique}
\item Suppose we condition on $(\mcl B_s^\bullet , h|_{\mcl B_s^\bullet})$ and $I\subset\bdy\mcl B_s^\bullet$ is an arc chosen in a way which depends only on  $(\mcl B_s^\bullet , h|_{\mcl B_s^\bullet})$. Show that if $I$ can be disconnected from $\infty$ in $\BB C\setminus \mcl B_s^\bullet$ by a set of small Euclidean diameter, then it holds with high conditional probability that there is a ``shield" in $\BB C\setminus \mcl B_s^\bullet$ which disconnects $I$ from $\infty$ with the property that no $D_h$-geodesic started from 0 can cross this shield~\cite[Sections 3.2 and 3.3]{gm-confluence}.  \label{item-step-shield}
\item Start with a positive radius $t$ and a collection of boundary arcs $\mcl I_0$ of $\bdy\mcl B_t^\bullet$. Iteratively apply Step~\ref{item-step-shield} for several successive radii $s_k > t$ to iteratively ``kill off" all of the geodesics started from 0 which pass through $I \in \mcl I_0$. Repeat until the number of remaining arcs in $\mcl I_0$ which have not yet been killed off is at most a large deterministic constant (independent of the initial choice of $\mcl I_0$). By sending the size of the arcs in $\mcl I_0$ to zero (and the number of such arcs to $\infty$), conclude that for each fixed $s > t$, there are a.s.\ only finitely many points on $\bdy\mcl B_t^\bullet$ which are hit by $D_h$-geodesics from 0 to $\bdy\mcl B_s^\bullet$~\cite[Section 3.4]{gm-confluence}. This yields Theorem~\ref{thm-finite-geo0}.  \label{item-step-iterate}
\item Reduce from finitely many points on $\bdy\mcl B_t^\bullet$ to a single point by ``killing off" the points one at a time~\cite[Section 4]{gm-confluence}. This yields Theorem~\ref{thm-clsce}. \label{item-step-single}
\end{enumerate}
See~\cite[Section 3.1]{gm-confluence} for a detailed overview of Steps~\ref{item-step-shield} and~\ref{item-step-iterate}. 

Most of the arguments involved in step~\ref{item-step-unique} carry over verbatim to the supercritical case once we know that the boundary of a filled metric ball is a Jordan curve (Proposition~\ref{prop-jordan}). So, we will not repeat many of these arguments here. Rather, we will just state a few of the most important results; see Section~\ref{sec-unique}.

Step~\ref{item-step-shield} requires non-trivial  modifications in the supercritical case. This is because the definition of the event used to build the ``shield" in the subcritical case involves a bound for the LQG diameters of certain small squares, which are infinite in the supercritical case. So, we need to work with somewhat different events in the supercritical case. Because of this, we will give most of the details for Step~\ref{item-step-shield} in this paper. This is done in Sections~\ref{sec-good-annuli} and~\ref{sec-geo-kill}. 

Step~\ref{item-step-iterate} requires only very minor modifications as compared to the subcritical case. In particular, in the subcritical case, the H\"older continuity of the LQG metric with respect to the Euclidean metric is used in one place. In our setting, we can replace this use of H\"older continuity by using Lemma~\ref{lem-bdy-dist0}, and then the argument goes through verbatim. As such, we will not give much detail about this step; see Section~\ref{sec-finite-geo-proof}. 

As in the case of step~\ref{item-step-shield}, step~\ref{item-step-single} requires non-trivial modifications in the supercritical case. Again, this is because the event used to ``kill off" all but one of the geodesics in the subcritical case involves bounds for LQG diameters. We will provide most of the details for the parts of step~\ref{item-step-single} which require modification, see Section~\ref{sec-one-geo}. 
 
\subsection{Preliminary results about LQG metric balls and geodesics}
\label{sec-unique}

We know that $D_h$-geodesics and outer boundaries of filled $D_h$-metric balls are simple, Euclidean-continuous curves (Lemma~\ref{lem-geodesic-cont} and Proposition~\ref{prop-jordan}). Furthermore, we know that $D_h(0,z) = s$ for each $s >0$ and each $z\in \bdy\mcl B_s^\bullet$ (Lemma~\ref{lem-outer-bdy-dist}). With these facts in hand, most of the results in~\cite[Section 2.1]{gm-confluence} and their proofs carry over verbatim to the supercritical case. 

We first state a result to the effect that ordinary and filled LQG metric balls are local sets for $h$ as defined in~\cite[Lemma 3.9]{ss-contour}. 
Let us recall the definition. Suppose $(h,A)$ is a coupling of $h$ with a random set $A$. 
We say that a closed set $A \subset\BB C$ is a \emph{local set} for $h$ if for any open set $U\subset \BB C$, the event $\{A\cap U \not=\emptyset\}$ is conditionally independent from $h|_{\BB C\setminus U}$ given $h|_U$. If $A$ is determined by $h$ (which will be the case for all of the local sets we consider), this is equivalent to the statement that $A$ is determined by $h|_U$ on the event $\{A\subset U\}$. 
For a local set $A$, we can condition on the pair $(A,h|_A)$: this is by definition the same as conditioning on the $\sigma$-algebra $\bigcap_{\ep > 0} \sigma(A , h|_{B_\ep(A)})$. 
The conditional law of $h|_{\BB C\setminus A}$ given $(A,h|_A)$ is that of a zero-boundary GFF on $\BB C\setminus A$ plus a harmonic function on $\BB C\setminus A$ which is determined by $(A,h|_A)$. 

\begin{lem} \label{lem-ball-local}
Let $x\in\BB C$ and $y\in\BB C\cup \{\infty\}$ be deterministic. 
If $\tau$ is a stopping time for the filtration generated by $(\mcl B_s(x) , h|_{\mcl B_s(x)})$, then $\mcl B_\tau(x)$ is a local set for $h$.
The same is true with $\mcl B_s^{y,\bullet}(x)$ in place of $ \mcl B_s(x)$.  
\end{lem}
\begin{proof}
Note that $\mcl B_s(x)$ and $\mcl B_s^{y,\bullet}(x)$ are Euclidean-closed (Lemma~\ref{lem-ball-closed}). In light of this, the lemma follows from exactly the same proof as~\cite[Lemma 2.1]{gm-confluence} (see also~\cite[Lemma 2.2]{local-metrics}). 
\end{proof}

Our next result gives the uniqueness of $D_h$-geodesics between typical points. 

\begin{lem} \label{lem-geo-unique}
For each fixed $z,w\in\BB C$, a.s.\ there is a unique $D_h$-geodesic from $z$ to $w$.
\end{lem}
\begin{proof}
We know that a.s.\ $D_h(z,w) < \infty$ and there is at least one $D_h$-geodesic from $z$ to $w$ (Lemma~\ref{lem-geodesic-cont}). 
The a.s.\ uniqueness of this geodesic follows from exactly the same argument as in the subcritical case, see~\cite[Theorem 1.2]{mq-geodesics}. 
\end{proof}

We emphasize that Lemma~\ref{lem-geo-unique} only holds a.s.\ for a \emph{fixed} choice of $z$ and $w$. We expect that there are exceptional pairs of points $z,w$ which are joined by multiple distinct $D_h$-geodesics (such points are known to exist in the subcritical case, see~\cite{akm-geodesics,gwynne-geodesic-network,mq-strong-confluence}). 
We also record the following analog of~\cite[Lemma 2.3]{gm-confluence}.

\begin{lem} \label{lem-non-cross}
For $q\in\BB Q^2$, let $P_q$ be the a.s.\ unique $D_h$-geodesic from 0 to $q$. The following holds a.s. If $q\in\BB Q^2$, $P'$ is a $D_h$-geodesic started from 0, and $u \in P_q\cap P'$, then there is a time $s \geq 0$ such that $P_q(s) = P'(s) = u$ and $P_q(t) = P'(t)$ for each $t \in [0,s]$.  
\end{lem}
\begin{proof}
Lemma~\ref{lem-geo-unique} implies that a.s.\ the $D_h$-geodesic from 0 to $q$ is unique for each $q\in\BB Q^2$. The lemma now follows from exactly the same argument as in~\cite[Lemma 2.3]{gm-confluence}. 
\end{proof}

The following result, which is the supercritical analog of~\cite[Lemma 2.4]{gm-confluence}, tells us that for $z\in\bdy\mcl B_s^\bullet$, there are two distinguished $D_h$-geodesics from 0 to $z$.

\begin{lem} \label{lem-leftmost-geodesic} 
Almost surely, for each $s>0$ and each $z\in \bdy \mcl B_s^\bullet$, there exists a (necessarily unique) \emph{leftmost (resp.\ rightmost) geodesic} $P_z^-$ (resp.\ $P_z^+$) from 0 to $z$ such that each $D_h$-geodesic from 0 to $z$ lies (weakly) to the right (resp.\ left) of $P_z^-$ (resp.\ $P_z^+$) if we stand at $z$ and look outward from $\mcl B_{s }^\bullet$. 
Moreover, there are sequences of points $q_n^- , q_n^+ \in \BB Q^2 \setminus\mcl B_s^\bullet$ such that the $D_h$-geodesics from 0 to $q_n^\pm$ satisfy $P_{q_n^\pm} \rta P_z^\pm$ uniformly w.r.t.\ the Euclidean topology. 
\end{lem}

See Figure~\ref{fig-leftmost} for an illustration of the statement and proof of Lemma~\ref{lem-leftmost-geodesic}. 
The proof of~\cite[Lemma 2.4]{gm-confluence} uses the Arz\'ela-Ascoli theorem and the continuity of the subcritical LQG metric with respect to the Euclidean metric to take limits of $D_h$-geodesics. In order to do this in the supercritical case, we need the following lemma.

\begin{figure}[t!]
 \begin{center}
\includegraphics[width=0.4\textwidth]{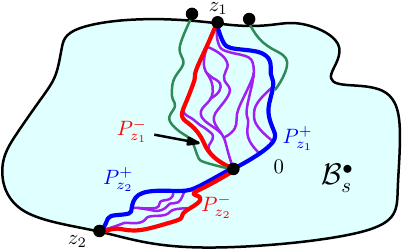}
\vspace{-0.01\textheight}
\caption{Two points $z_1,z_2 \in \bdy \mcl B_s^\bullet$ and their associated leftmost and rightmost $D_h$-geodesics (red and blue). Other $D_h$-geodesics from 0 to $z_1$ and $z_2$ are shown in purple. We have also shown two $D_h$-geodesics from 0 to points of $\BB Q^2\setminus \mcl B_s^\bullet$ (green) which approximate $P_{z_1}^-$ and $P_{z_1}^+$, respectively. 
Note that $P_{z_1}^-$ and $P_{z_1}^+$ intersect only at their endpoints, whereas $P_{z_2}^-$ and $P_{z_2}^+$ coincide for an initial time interval. Theorem~\ref{thm-clsce} implies that a.s.\ the latter situation holds simultaneously for every $z\in \bdy\mcl B_s^\bullet$, but this has not been established yet. A similar figure and caption appeared in~\cite{gm-confluence}.
}\label{fig-leftmost}
\end{center}
\vspace{-1em}
\end{figure}

\begin{lem} \label{lem-aa}
Almost surely, the following is true. Let $\{z_n\}_{n\in\BB N}$, $\{w_n\}_{n\in\BB N}$, $z$, and $w$ be non-singular points for $D_h$ such that $z_n \rta z$, $w_n \rta w$, and $\limsup_{n\rta\infty} D_h(z_n,w_n)  < \infty$. 
Let $\{P_n\}_{n\in\BB N}$ be a sequence of $D_h$-rectifiable paths from $z_n$ to $w_n$, each parametrized by $D_h$-length, such that $\op{len}(P_n ; D_h) - D_h(z_n,w_n) \rta 0$ as $n\rta\infty$, where $\op{len}(P_n ; D_h)$ denotes the $D_h$-length. There is a subsequence along which the paths $P_n$ converge uniformly w.r.t.\ the Euclidean metric to a $D_h$-rectifiable path $P$ from $z$ to $w$. If $\lim_{n\rta\infty} D_h(z_n,w_n) = D_h(z,w)$, then $P$ is a $D_h$-geodesic. 
\end{lem}

The statement of Lemma~\ref{lem-aa} allows for uniform convergence of paths which are defined on $[0,T_n]$ where $T_n$ possibly depends on $n$. To make sense of uniform convergence under these circumstances, we view all of our paths as being defined on $[0,\infty)$ by extending them to be constant after time $T_n$. 

\begin{proof}[Proof of Lemma~\ref{lem-aa}] 
Let $T_n := \op{len}(P_n;D_h)$, so that $P_n : [0,T_n] \rta \BB C$. 
Since $P_n$ is parametrized by $D_h$-length, for $0\leq t \leq s \leq T_n$, we have $D_h(P_n(s) , P_n(t)) \leq s-t$ and $D_h(z_n,w_n) \leq (T_n - s) + D_h(P_n(s) , P_n(t)) + t$. Therefore, 
\eqbn
T_n  - D_h(z_n,w_n) \geq s-t  -  D_h(P_n(s) , P_n(t))  \geq 0.
\eqen
Since $T_n - D_h(z_n,w_n) \rta 0$ as $n\rta\infty$ by hypothesis,  
\eqb \label{eqn-aa-dist-lim}
\lim_{n\rta\infty} \sup_{0\leq t \leq s \leq T_n} \left| s-t  -  D_h(P_n(s) , P_n(t)) \right| = 0 .
\eqe
In particular,~\eqref{eqn-aa-dist-lim} implies that the $P_n$'s are $D_h$-equicontinuous. 
 
Since $\limsup_{n\rta\infty} D_h(z_n,w_n) < \infty$ and $D_h$-metric balls are Euclidean-bounded, there is a bounded open subset of $\BB C$ which contains $P_n$ for each $n\in\BB N$. 
Since the identity mapping $(\BB C , D_h) \rta (\BB C, |\cdot|)$ is continuous and the $P_n$'s are $D_h$-equicontinuous, it follows that the $P_n$'s are Euclidean equicontinuous. Hence there is a sequence $\mcl N$ of positive integers tending to $\infty$ and a Euclidean-continuous path $P :[0, T] \rta \BB C$ from $z$ to $w$ such that $P_n \rta P$ uniformly w.r.t.\ the Euclidean topology along $\mcl N$. 

Since $D_h$ is lower semicontinuous w.r.t.\ the Euclidean metric,~\eqref{eqn-aa-dist-lim} implies that $D_h(P(s) , P(t)) \leq |s-t|$ for any two times $s,t \in [0,T]$. Consequently, $P$ is $D_h$-rectifiable and for any $0\leq t \leq s \leq T$, the $D_h$-length of $P([t,s])$ is at most $s-t$. If $\lim_{n\rta\infty} D_h(z_n,w_n) = D_h(z,w)$, then $T = D_h(z,w)$. Since the $D_h$-length of $P$ is at most $T$, it follows that the $D_h$-length of $P$ is exactly $T$ and $P$ is a $D_h$-geodesic. 
\end{proof}

\begin{proof}[Proof of Lemma~\ref{lem-leftmost-geodesic}]
The proof is essentially the same as~\cite[Lemma 2.4]{gm-confluence}, but there are a couple of minor differences so we will give the details.  
Fix a point $w \in \bdy\mcl B_s^\bullet\setminus \{z\}$. 
Let $A^-$ and $A^+$, respectively, be the clockwise and counterclockwise arcs of $\bdy\mcl B_s^\bullet$ from $w$ to $z$, not including $w$ and $z$ themselves. Note that these arcs are well-defined since $\bdy\mcl B_s^\bullet$ is a Jordan curve (Proposition~\ref{prop-jordan}). 
We can choose sequences $z_n^- \in A^-$  (resp.\ $z_n^+ \in A^+$) which converge to $z$ from the left (resp.\ right) w.r.t.\ the Euclidean topology (with ``left" and ``right" defined as in the lemma statement).
 
The set $\BB Q^2$ is a.s.\ $D_h$-dense in $\BB C\setminus \{\text{singular points}\}$~\cite[Proposition 1.13]{pfeffer-supercritical-lqg} and the set $\mcl B_\ep(z_n^\pm) \setminus \mcl B_s^\bullet$ contains a $D_h$-open set for each $\ep > 0$. Applying this with $\ep$ equal to the minimum of $1/n$ and $\frac14 D_h(z_n^\pm,A^\mp)$, we see that for $n \in \BB N$ we can find $q_n^\pm \in \BB Q^2\setminus\mcl B_s^\bullet$ such that  
\eqb \label{eqn-leftmost-rational}
D_h(q_n^\pm , z_n^\pm) \leq \min\left\{ \frac{1}{n} , \frac{1}{2} D_h(q_n^\pm , A^\mp) \right\} . 
\eqe 

Let $P_{q_n^\pm}$ be the (a.s.\ unique, by Lemma~\ref{lem-geo-unique}) $D_h$-geodesic from 0 to $q_n^\pm$. 
Then $P_{q_n^\pm}(s) \in \bdy \mcl B_s^\bullet$. 
Since $D_h(0,z_n^\pm) = s$ (Lemma~\ref{lem-outer-bdy-dist}), 
\eqbn
D_h(P_{q_n^\pm}(s) , q_n^\pm) = D_h(0, q_n^\pm) - s \leq D_h(q_n^\pm , z_n^\pm).
\eqen
From this and~\eqref{eqn-leftmost-rational}, 
\eqbn
D_h(P_{q_n^\pm}(s) , A^\mp) \geq  D_h(q_n^\pm , A^\mp) - D_h(P_{q_n^\pm}(s), q_n^\pm)  \geq \frac12 D_h(q_n^\pm , A^\mp) > 0 .
\eqen
Hence $P_{q_n^\pm}(s) \notin A^\mp$. From~\eqref{eqn-leftmost-rational} and since $z_n^- \to z$ from the left, we see that also $P_{q_n^-}(s) \to z$ from the left. The same is true for $z_n^+$, but with ``right" in place of ``left".

Since $D_h(0,z_n^\pm) = s$, we have $0\leq  D_h(0,q_n^\pm) - s \leq 1/n$. We may therefore apply Lemma~\ref{lem-aa} to get that after possibly passing to a subsequence, we can arrange that the paths $P_{q_n^\pm}$ converge uniformly w.r.t.\ the Euclidean metric to $D_h$-geodesics $P_z^\pm$ from 0 to $z$.
By Lemma~\ref{lem-non-cross}, no $D_h$-geodesic from 0 to $z$ can cross any of the geodesics $P_{q_n^\pm}$. 
If a geodesic from 0 to $z$ does not lie in the closure of the open subset of $\mcl B_s^\bullet$ lying to the right of $P_z^-$ and to the left of $P_z^+$, then it must cross $P_{q_n^-}$ or $P_{q_n^+}$ for some $n$. 
Hence each geodesic from 0 to $z$ lies to the right of $P_z^-$ and to the left of $P_z^+$.
\end{proof}

Our next lemma is used in the iterative argument used to prove confluence of geodesics (see step~\ref{item-step-iterate} of the outline at the beginning of this section).

\begin{lem} \label{lem-geo-arc}
Almost surely, the following is true for each $0 < s < s' < \infty$. 
Let $\mcl I$ be a finite collection of disjoint arcs of $\bdy\mcl B_s^\bullet$. 
For each $I\in\mcl I$, let $I'$ be the set of $z\in\bdy\mcl B_{s'}^\bullet$ such that the leftmost $D_h$-geodesic from 0 to $z$ passes through $I$.
Then each $I'$ is either empty or is a connected arc of $\bdy\mcl B_{s'}^\bullet$ and the arcs $I'$ for different choices of $I\in\mcl I$ are disjoint.
\end{lem}
\begin{proof}
Since we know that each $\bdy\mcl B_s^\bullet$ is a Jordan curve and $D_h(0,z) = s$ for each $z\in\bdy\mcl B_s^\bullet$, the proofs of~\cite[Lemmas 2.6 and 2.7]{gm-confluence} extend verbatim to the supercritical case (note that~\cite[Lemma 2.5]{gm-confluence} is a deterministic statement which can be re-used in the supercritical case). In particular,~\cite[Lemma 2.7]{gm-confluence} gives precisely the statement of the present lemma. 
\end{proof}

Finally, we record an FKG inequality for the LQG metric, which is proven in exactly the same way as~\cite[Proposition 2.8]{gm-confluence}. For the statement, we note that if $D$ is a weak LQG metric with parameter $\xi$ as in Definition~\ref{def-metric}, $U\subset \BB C$ is open, and $\rng h$ is a zero-boundary GFF on $U$, then we can define $D_{\rng h}$ as a random lower semicontinuous metric on $U$ follows. Let $h$ be a whole-plane GFF. We can write $h|_U = \rng h + \frk h$, where $\frk h$ is a random harmonic function on $U$ (see, e.g.,~\cite[Lemma 2.2]{gms-harmonic}). We then define $D_{\rng h} = e^{-\xi \frk h} \cdot D_{\rng h}$, using the notation~\eqref{eqn-metric-f}. As explained in~\cite[Remark 1.2]{gm-confluence}, it is easily seen that $D_{\rng h}$ is a measurable function $\rng h$. 

\begin{prop}[FKG for the LQG metric] \label{prop-fkg-metric}
Let $\xi > 0$, let $U\subset\BB C$ be an open domain, let $\rng h$ be a zero-boundary GFF on $U$, and let $D$ be a weak LQG metric with parameter $\xi$.  
Let $\Phi$ and $\Psi$ be bounded, real-valued measurable functions on the space of lower semicontinuous metrics on $U$ which are non-decreasing in the sense that for any two such metrics $D_1,D_2$ with $D_1(z,w) \leq D_2(z,w) $ for all $z,w\in U$, one has $\Phi(D_1) \leq \Phi(D_2)$ and $\Psi(D_1)\leq \Psi(D_2)$. 
Suppose further that $\Phi$ and $\Psi$ are a.s.\ continuous at $D_{\rng h}$ in the sense that for every (possibly random) sequence of continuous functions $\{f^n\}_{n\in\BB N}$ which converges to zero uniformly on $U$, one has $\Phi(e^{\xi f^n} \cdot D_{\rng h} ) \rta \Phi(D_{\rng h})$ and $\Psi(e^{\xi f^n} \cdot D_{\rng h}) \rta \Psi(D_{\rng h})$.  
Then $\op{Cov}(\Phi(D_{\rng h}), \Psi(D_{\rng h})) \geq 0$.
\end{prop}
\begin{proof}
This follows from Weyl scaling (Axiom~\ref{item-metric-f}) together with the FKG inequality for the GFF given in~\cite[Lemma 2.10]{gm-confluence}, via exactly the same argument as in the proof of~\cite[Proposition 2.8]{gm-confluence}. 
\end{proof}

\subsection{Finitely many leftmost geodesics across an LQG annulus}
\label{sec-finite-geo}

In this subsection we explain how to extend the core part of the argument in~\cite{gm-confluence}, corresponding to steps~\ref{item-step-shield} and~\ref{item-step-iterate} above, to the supercritical case. 
We start in Section~\ref{sec-good-annuli} by defining an event for a Euclidean annulus which will be used to build ``shields" which $D_h$-geodesics cannot cross.
Then, in Section~\ref{sec-geo-kill}, we explain how to use this event to ``kill off" all of the geodesics which pass through a given boundary arc of a filled $D_h$-metric ball. 
We will give most of the details of the arguments in these two subsections since non-trivial modifications are required as compared to the analogous arguments in~\cite{gm-confluence}.
In Section~\ref{sec-finite-geo-proof}, we state a more quantitative version of Theorem~\ref{thm-finite-geo0} (Theorem~\ref{thm-finite-geo-quant}) and explain why the theorem follows from the same proof as its subcritical analog from~\cite[Section 3.4]{gm-confluence}, except for one trivial modification. 

\subsubsection{Good annuli}
\label{sec-good-annuli}

We now define an event for a Euclidean annulus which will eventually be used to build ``shields" surrounding boundary arcs of a filled $D_h$-metric ball through which $D_h$-geodesics to 0 cannot pass. See Figure~\ref{fig-geo-event} for an illustration.

\begin{figure}[t!]
 \begin{center}
\includegraphics[scale=1]{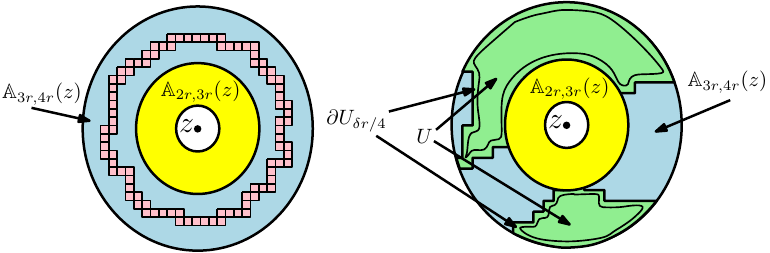}
\vspace{-0.01\textheight}
\caption{Illustration of the definitions in Section~\ref{sec-good-annuli}. The set $\mcl U_r(z) = \mcl U_r(z;\delta)$ consists of open subsets $U$ of $\BB A_{3r,4r}(z)$ such that  $\BB A_{3r,4r}(z) \setminus  U$ is a finite union of sets of the form $S \cap \BB A_{3r,4r}(z)$ for $\delta r\times\delta r$ squares $S\in   \mcl S_{\delta r}(\BB A_{3r,4r}(z))$ (i.e., with corners in $\delta \BB Z^2$). One such set is shown in light green in the right panel. For each $U\in\mcl U_r(z)$, $E_r^U(z)$ is the event that (1) the $D_h$-distance across the yellow annulus $\BB A_{2r,3r}(z)$ is bounded below, (2) there is a path of squares in $\BB A_{3r,4r}(z)$ which disconnects the inner and outer boundaries of this annulus, with the property that the $D_h$-distance around $B_{2\delta r}(S) \setminus B_{\delta r} (S )$ is small for each square $S$ in the path (the squares are shown in pink in the left panel), and (3) the harmonic part of $h|_U$ is bounded above on the set $U_{\delta r/4} \subset U$ (outlined in black in right panel). 
}\label{fig-geo-event}
\end{center}
\vspace{-1em}
\end{figure}

For $\ep > 0$, $z\in \BB C$, and a set $V\subset\BB C$, we define the collection of Euclidean squares
\eqb \label{eqn-square-def}
\mcl S_\ep^z(V) := \left\{ [x, x+\ep] \times [y,y+\ep] : (x,y)\in\ep\BB Z^2 +z , \, ([x, x+\ep] \times [y,y+\ep])\cap V \not=\emptyset\right\}. 
\eqe
Note that $\mcl S_\ep^z(V)$ depends only on the value of $z$ modulo $\ep\BB Z^2$ and that $\mcl S_\ep^z(V) - z  = \mcl S_\ep^0(V-z)$. 
 
For $z\in\BB C$, $r> 0$, and $\delta\in (0,1)$, we define $\mcl U_r(z) = \mcl U_r(z;\delta)$ to be the (finite) set of open subsets $U$ of the annulus $\BB A_{3r,4r}(z)$ such that $\BB A_{3r,4r}(z) \setminus  U$ is a finite union of sets of the form $S \cap \BB A_{3r,4r}(z)$ for squares $S \in \mcl S_{\delta r}^z(\BB A_{3r,4r}(z)) $.
For $U\in \mcl U_r(z;\delta)$ and $\ep > 0$, we define
\eqb \label{eqn-U-nbd-def}
U_\ep := \left\{ u \in U : \op{dist}(z , \bdy U) > \ep \right\}
\eqe 
where $\op{dist}$ denotes Euclidean distance.
 
For $z\in\BB C$, $r  > 0$, parameters $  c,\delta    \in (0,1)$ and $A > 0$, and $U\in \mcl U_r(z;\delta)$, we let $E_r^U(z) = E_r^U(z;  c , \delta,A) $ be the event that the following is true. 
\begin{enumerate}  
\item $D_h \left( \text{across $\BB A_{2r,3r}(z)$} \right) \geq c \frk c_r e^{\xi h_r(z)}$. \label{item-clsce-across}
\item There exists a collection of $\delta r \times \delta r$ squares $S_1,\dots, S_N \in \mcl S_{\delta r}(\BB A_{3.1 r, 3.9 r}(z))$ with the following properties. \label{item-clsce-ball}
\begin{enumerate}
\item The squares $S_{j-1}$ and $S_j$ share a side for each $j=1,\dots,N$, where here we set $S_0 = S_N$. \label{item-square-path-side0}
\item The union of the squares $S_1,\dots,S_N$ contains a path which disconnects the inner and outer boundaries of $\BB A_{3.1 r,3.9 r}(z)$. \label{item-square-path-dc0}
\item For each $j=1,\dots,N$, we have $D_h\left(\text{around $B_{2\delta r}(S_j) \setminus B_{\delta r}(S_j)$} \right) \leq \frac{c}{100} \frk c_r e^{\xi h_r(z)}$. \label{item-square-path-around0}
\end{enumerate}
\item Let $\frk h^U$ be the harmonic part of $h|_U$. Then, in the notation~\eqref{eqn-U-nbd-def}, \label{item-clsce-harmonic} 
\eqb 
\sup_{u   \in U_{\delta r / 4}} |\frk h^U(u) - h_r(z) |  \leq A  .
\eqe  
\end{enumerate}
We also define
\eqb \label{eqn-clsce-all-event}
E_r(z) = E_r(z;c,\delta,A) := \bigcap_{U\in \mcl U_r(z;\delta)} E_r^U(z) .
\eqe
The first two conditions in the definition of $E_r^U(z)$ do not depend on $U$, so the only difference between $E_r(z)$ and $E_r^U(z)$ is that for the former event, condition~\ref{item-clsce-harmonic} is required to hold for all choices of $U$ simultaneously.

The events $E_r(z)$ and $E_r^U(z)$ are defined in exactly the same manner as in~\cite[Section 3.2]{gm-confluence} except that in~\cite{gm-confluence}, condition~\ref{item-clsce-ball} is replaced by an upper bound for the $D_h$-diameters of the squares in $\mcl S_{\delta r}(\BB A_{3r,4r}(z))$. Of course, such a diameter upper bound does not hold in the supercritical case, which is the reason for the modification. 

The occurrence of $E_r^U(z)$ or $E_r(z)$ is unaffected by adding a constant to the field. By this and the locality of $D_h$ (Axiom~\ref{item-metric-local}), these events are determined by $h|_{\BB A_{2r,5r}(z)}$, viewed modulo additive constant.  

We think of annuli $\BB A_{2r,5r}(z)$ for which $E_r(z)$ occurs as ``good". 
We will show in Lemma~\ref{lem-clsce-event-pos} just below that $\BB P[E_r(z)]$ can be made close to 1 by choosing the parameters $\delta, c , A$ appropriately, in a manner which is uniform over the choices of $r$ and $z$, 
The reason for separating $E_r(z)$ and $E_r^U(z)$ is that conditioning on $E_r^U(z)$ is easier than conditioning on $E_r(z)$ (see Lemma~\ref{lem-cond-diam-small} just below). 

We will eventually apply condition~\ref{item-clsce-harmonic} with $U$ equal to $\BB A_{3r,4r}(z)$ minus the union of the set of squares in $\mcl S_\ep^z(\BB A_{3r,4r}(z))$ which intersect a filled $D_h$-metric ball $\mcl B_\tau^\bullet$, for an appropriate stopping time $\tau$. Condition~\ref{item-clsce-harmonic} together with the Markov property of $h$ allows us to show that with uniformly positive conditional probability given $h|_{\BB C\setminus U}$ and the event $E_r^U(z)$, the maximal $D_h$-distance between the centers of any two squares in $\mcl S_{\delta r}^z( U)$ which are contained in the same connected component of $U$ is small (see Lemma~\ref{lem-cond-diam-small}). 
This combined with condition~\ref{item-clsce-ball} will show that with uniformly positive conditional probability given $h|_{\BB C\setminus U}$ and $E_r^U(z)$, there is a collection of paths in $\BB A_{3r,4r}(z)$ which each have small $D_h$-length and whose union disconnects the inner and outer boundaries of $\BB A_{3r,4r}(z)$ in $\BB C\setminus \mcl B_\tau^\bullet$ (see Lemma~\ref{lem-geo-kill}). 
Due to condition~\ref{item-clsce-across}, we can arrange that $D_h$-distance from each point of each of these paths to $\bdy\mcl B_\tau^\bullet$ will be smaller than $D_h(\text{across $\BB A_{2r,3r}(z)$})$. This will show that no $D_h$-geodesic from a point outside of $B_{4r}(z) \cup \mcl B_\tau^\bullet$ can cross $\BB A_{2r,3r}(z)$ before entering $\mcl B_\tau^\bullet$.
See Figure~\ref{fig-geo-kill} for an illustration of how the events $E_r^U(z)$ will eventually be used.

\begin{lem} \label{lem-clsce-event-pos}
For each $p \in (0,1)$, we can find parameters $ c , \delta  \in (0,1)$ and $A>0$ such that, in the notation~\eqref{eqn-clsce-all-event}, we have $\BB P[E_r(z)] \geq p$ for each $z\in\BB C$ and $r > 0$.
\end{lem}

In order to show that condition~\ref{item-clsce-ball} in the definition of $E_r^U(z)$ occurs with high probability, we will use the following lemma.

\begin{lem} \label{lem-square-path}
Fix $\zeta > 0$ and $0 < a < b < \infty$. For each $z \in \BB C$ and $r  >0$, it holds with superpolynomially high probability as $\delta \rta 0$, uniformly over the choice of $z$ and $r$, that there exists a collection of $\delta r \times \delta r$ squares $S_1,\dots, S_N \in \mcl S_{\delta r}(\BB A_{a r, b r}(z))$ with the following properties.
\begin{enumerate}
\item The squares $S_{j-1}$ and $S_j$ share a side for each $j=1,\dots,N$, where here we set $S_0 = S_N$. \label{item-square-path-side}
\item The union of the squares $S_1,\dots,S_N$ contains a path which disconnects the inner and outer boundaries of $\BB A_{a r, b r}(z)$. \label{item-square-path-dc}
\item For each $j=1,\dots,N$, we have $D_h\left(\text{around $B_{2\delta r}(S_j) \setminus B_{\delta r}( S_j )$} \right) \leq \delta^{\xi Q - \zeta} \frk c_r e^{\xi h_r(z)}$. \label{item-square-path-around}
\end{enumerate}
\end{lem}
\begin{proof}
This can be proven using level sets of the GFF (see, e.g., the arguments in~\cite[Section 2]{lfpp-pos} or~\cite[Section 5.1]{ghpr-central-charge}), but we will give a different argument based on estimates for weak LQG metrics with parameter $\wt\xi$, where $\wt\xi$ is large.   

Let $\wt\xi > \xi$ to be chosen later, in a manner depending on $\zeta$. Let $\wt D_h$ be a weak $\wt\xi$-LQG metric with respect to $h$ (e.g., a subsequential limit of LFPP with parameter $\wt\xi$). We denote objects associated with $\wt\xi$ and $\wt D_h$ with a tilde.

By Proposition~\ref{prop-two-set-dist} and a union bound, it holds with superpolynomially high probability as $\delta \rta 0$, uniformly over the choices of $z$ and $r$, that for each $S\in \mcl S_{\delta r}(\BB A_{a r, b r}(z))$, 
\eqbn
D_h\left(\text{around $B_{2\delta r}(S) \setminus B_{\delta r}(S)$} \right) \leq \delta^{- \xi \zeta } \frk c_{\delta r} e^{\xi h_{\delta r}(v_{S })} \quad \text{and} \quad
\wt D_h\left(\text{across $B_{2\delta r}(S) \setminus B_{\delta r}(S)$} \right) \geq \delta^{\wt\xi \zeta} \wt{\frk c}_{\delta r} e^{\wt \xi h_{\delta r}(v_{S})} .
\eqen
where $v_S$ is the center of $S$. 
Since $\frk c_{\delta r} = \delta^{\xi Q + o_\delta(1)} \frk c_r$ and similarly for $\wt{\frk c}_{\delta r}$, we can re-write this as
\allb \label{eqn-large-xi-across}
&D_h\left(\text{around $B_{2\delta r}(S) \setminus B_{\delta r}(S)$} \right) \leq \delta^{\xi (Q-\zeta) + o_\delta(1)} \frk c_r e^{\xi h_{\delta r}(v_S)} \quad \text{and} \notag\\
&\wt D_h\left(\text{across $B_{2\delta r}(S) \setminus B_{\delta r}(S)$} \right) \geq \delta^{\wt\xi (\wt Q +  \zeta)  + o_\delta(1)} \wt{\frk c}_{ r} e^{\wt \xi h_{\delta r}(v_S)} .
\alle

By another application of Proposition~\ref{prop-two-set-dist}, it holds with superpolynomially high probability as $\delta \rta 0$ that there is a path $\wt\pi$ in $\BB A_{a r, b r}(z)$ which disconnects the inner and outer boundaries of $\BB A_{a r, b r}(z)$ and has $\wt D_h$-length at most $\delta^{-\wt\xi \zeta } \wt{\frk c}_r e^{\wt\xi h_r(z)}$. Let $S_1,\dots,S_N$ be the squares in $\mcl S_{\delta r}(\BB A_{a r, b r}(z))$ which are hit by $\wt\pi$, listed in numerical order. Then $S_1,\dots,S_N$ satisfy properties~\ref{item-square-path-side} and~\ref{item-square-path-dc} in the lemma statement. 

For each $j$, the path $\wt\pi$ crosses between the inner and outer boundaries of $B_{2\delta r}(S_j) \setminus B_{\delta r}(S_j)$, so by~\eqref{eqn-large-xi-across},
\eqb
\delta^{\wt\xi (\wt Q +  \zeta) + o_\delta(1)} \wt{\frk c}_{ r} e^{\wt \xi h_{\delta r}(v_{S_j})}
\leq \wt D_h\left(\text{across $B_{2\delta r}(S_j) \setminus B_{\delta r}(S_j)$} \right) 
\leq \left(\text{$\wt D_h$-length of $\wt\pi$}\right)
\leq \delta^{-\wt\xi\zeta} \wt{\frk c}_r e^{\wt\xi h_r(z)} .
\eqe
Re-arranging this inequality, then taking the $1/\wt\xi$ power of both sides, gives 
\eqb
e^{ h_{\delta r}(v_{S_j}) - h_r(z) } \leq \delta^{ -   (\wt Q + 2 \zeta) }  .
\eqe
As $\wt\xi \rta \infty$, we have $\wt Q \rta 0$~\cite[Proposition 1.1]{dg-supercritical-lfpp}.
Hence, if $\wt\xi$ is chosen to be sufficiently large (depending on $\zeta$) then we can arrange that $\wt Q < \zeta$, so $e^{ h_{\delta r}(v_{S_j}) - h_r(z) } \leq \delta^{-3\zeta}$. Plugging this into the first inequality in~\eqref{eqn-large-xi-across} shows that for each $j$,  
\eqbn
D_h\left(\text{around $B_{2\delta r}(S_j) \setminus B_{\delta r}(S_j)$} \right) \leq \delta^{\xi (Q-4\zeta) + o_\delta(1)} \frk c_r e^{\xi h_{  r}(z)} .
\eqen
Since $\zeta$ is arbitrary, this implies that can arrange for the $S_j$'s to satisfy condition~\ref{item-square-path-around}.
\end{proof}

\begin{proof}[Proof of Lemma~\ref{lem-clsce-event-pos}]
By translation invariance and tightness across scales (Axioms~\ref{item-metric-translate} and~\ref{item-metric-coord}), the laws of the reciprocals of the scaled distances $\frk c_r^{-1} e^{-\xi h_r(z)} D_h( \text{across $\BB A_{2r,3r}(z)$} )$ for $z\in\BB C$ and $r>0$ are tight. Therefore, we can find $c = c(p) > 0$ such that for each $z\in\BB C$ and $r > 0$, condition~\ref{item-clsce-across} in the definition of $E_r^U(z)$ occurs with probability at least $1 - (1-p)/3$.
By Lemma~\ref{lem-square-path}, we can find $\delta = \delta(p , c)  \in (0,1)$ such that condition~\ref{item-clsce-ball} in the definition of $E_r(z)$ occurs with probability at least $1-(1-p)/3$. 
For a given choice of $\delta$, the collection of open sets $\mcl U_r(z;\delta)$ is finite, and is equal to $r \mcl U_1(0;\delta) + z$ (here we use the translation by $z$ in~\eqref{eqn-square-def}). 
Since $\frk h^U$ is continuous away from $\bdy U$, for any fixed choice of $U \in \mcl U_1(0;\delta)$, a.s.\ $\sup_{u  \in U_{\delta/4 }} |\frk h^U(u)  |  < \infty$. 
By combining this with the translation and scale invariance of the law of $h$, modulo additive constant, we find that there exists $A> 0$ (depending on $\delta$) such that with probability at least $1-(1-p)/3 $, condition~\ref{item-clsce-harmonic} in the definition of $E_r^U(z)$ holds simultaneously for every $U\in \mcl U_r(z;\delta)$.
\end{proof}

We now want to show that if we condition on $E_r^U(z)$, then with positive conditional probability the $D_h$-distances between certain points in $U$ are very small. 
For $r >0$, $z\in\BB C$, and $U\in\mcl U_r(z)$, let $\mcl V(U)$ be the set of connected components of $U$. Also let
\eqb
\mcl Z(U) :=  \left\{\text{center points of squares $S\in \mcl S_{\delta r}^z(U)$ with $S\subset \ol U$}\right\}
\eqe
be the set of centers of squares which are entirely contained in $\ol U$. 
We define the event  
\eqb \label{eqn-diam-small-event}
H_r^U(z) := \left\{\max_{V\in \mcl V(U)} \sup_{u,v \in V \cap \mcl Z(U)} D_h\left( u , v ; V \right) \leq \frac{c}{2} \frk c_r e^{\xi h_r(z)}\right\} ,
\eqe
i.e., $H_r^U(z)$ is the event that for any $V\in\mcl V(U)$, the $D_h$-internal distance in $V$ between any two of the centers of the squares which are entirely contained in $V$ is bounded above by $ \frac{c}{2} \frk c_r e^{\xi h_r(z)}$ (this quantity is relevant due to condition~\ref{item-clsce-across} in the definition of $E_r^U(z)$). We think of annuli $\BB A_{2r,5r}(z)$ for which $E_r^U(z) \cap H_r^U(z)$ occurs (for a suitable choice of $U$) as ``very good". 

We note that $H_r^U(z)$ does \emph{not} include an upper bound for the $D_h$-distance between two arbitrary points of $V$. This is because there are a.s.\ singular points contained in $V$, but a.s.\ none of the (finitely many) points in $\mcl Z(U)$ are singular points, so a.s.\ any two points in $\mcl Z(U)$ lie at finite $D_h$-distance from each other. 

The following is the analog of~\cite[Lemma 3.3]{gm-confluence} in our setting. It says that an annulus has positive conditional probability to be ``very good" given that it is ``good". 

\begin{lem} \label{lem-cond-diam-small}
For any choice of parameters $c,\delta, A$, there is a constant $\frk p = \frk p( c,\delta ,   A) > 0$ such that for each $r  > 0$, each $z\in\BB C$, and each $U \in \mcl U_r(z)$,  
\eqb  \label{eqn-cond-diam-small} 
\BB P\left[  H_r^U(z)   \,\big|\, h|_{\BB C\setminus U}  , E_r^U(z) \right]  \geq \frk p .
\eqe  
\end{lem}
\begin{proof}
This is proven via essentially the same argument as~\cite[Lemma 3.3]{gm-confluence}: we subtract a large bump function from $  h|_U$ to get a lower bound for $\BB P[H_r^U(z) \,| \, h|_{\BB C\setminus U} ]$, then we use the FKG inequality (Proposition~\ref{prop-fkg-metric}) to add in the conditioning on $E_r^U(z)$ (we only need to use the FKG inequality for the second condition in the definition of $E_r^U(z)$ since the other two conditions are determined by $h|_{\BB C\setminus U}$). 
The proof is actually slightly simpler than that of~\cite[Lemma 3.3]{gm-confluence} since we are not trying to bound distances between points which are arbitrarily close to $\bdy V$, so unlike in~\cite{gm-confluence} we do not need to worry about the diameters of the squares in $\mcl S_{\delta r}^z(V)$.  
\end{proof}

\subsubsection{Cutting off geodesics from a boundary arc}
\label{sec-geo-kill}

For $c,\delta\in (0,1)$ and $A>0$, define $E_r(z) = E_r(z; c,\delta, A)$ as in~\eqref{eqn-clsce-all-event}. 
We will use the events $E_r(z)$ to build ``shields" which prevent $D_h$-geodesics from hitting a given arc of a filled metric ball.  
For $z\in\BB C$ and $\BB r >0$, let $\rho_{\BB r}^0(z) := \BB r$ and for $n\in\BB N$, inductively define
\eqb \label{eqn-good-radius-def}
\rho_{\BB r}^n(z) := \inf\left\{ r \geq 6\rho_{\BB r}^{n-1}(z) : \text{$r = 2^k \BB r$ for some $k\in\BB Z$}, \: E_r(z) \: \text{occurs} \right\} .
\eqe
Since $E_r(z)$ is determined by $h|_{\BB A_{2r,5r}(z)}$, it follows that $\rho_{\BB r}^n(z)$ is a stopping time for the filtration generated by $h|_{B_{5r}(z)}$ for $r \geq \BB r$. The following lemma allows us to produce lots of annuli for which $E_r(z)$ occurs.

\begin{lem} \label{lem-clsce-all}
There exists a choice of parameters $c,\delta\in(0,1)$ and $A >0$ and another parameter $\eta > 0$, depending only on the choice of metric $D$, such that the following is true. 
For each compact set $K\subset\BB C$, it holds with probability $1-O_\ep(\ep^2)$ (at a rate depending on $K$) that
\eqb
\rho_{\ep \BB r}^{\lfloor \eta\log\ep^{-1} \rfloor}(z) \leq   \ep^{1/2} \BB r , \quad\forall z\in \left(\frac{\ep \BB r}{4} \BB Z^2 \right) \cap B_{\ep \BB r}(\BB r K) .
\eqe
\end{lem}
\begin{proof}
This follows from the the variant of Lemma~\ref{lem-annulus-iterate} where our radii are increasing rather than decreasing~\cite[Lemma 2.12]{gm-confluence} together with a union bound, exactly as in the proof of~\cite[Lemma 3.4]{gm-confluence}. 
\end{proof}

We henceforth let $c,\delta,A$, and $\eta$ be as in Lemma~\ref{lem-clsce-all}. 
For $\ep > 0$, $\BB r > 0$, and a compact set $K\subset\BB C$, let 
\eqb \label{eqn-extra-radius-eucl}
R_{\BB r}^\ep(K) :=  6 \sup\left\{ \rho_{\ep \BB r}^{\lfloor \eta \log \ep^{-1} \rfloor}(z) : z\in \left( \frac{\ep \BB r}{4} \BB Z^2 \right) \cap B_{\ep \BB r}\left( K \right) \right\} +\ep\BB r ,
\eqe
so that each of the radii $\rho_{\ep \BB r}^n(z)$ for $z\in \left( \frac{\ep \BB r}{4} \BB Z^2 \right) \cap B_{\ep \BB r}\left( K \right)$ and $n\in[1,\eta\log\ep^{-1}]_{\BB Z}$ is determined by $R_{\BB r}^\ep(K)$ and $h|_{B_{R_{\BB r}^\ep(K)}(K)}$. 
Lemma~\ref{lem-clsce-all} shows that for each fixed choice of $K$, $\BB P[ R_{\BB r}^\ep(\BB r K) \leq (6\ep^{1/2} + \ep) \BB r ]$ tends to 1 as $\ep\rta 0$, at a rate which is uniform in $\BB r$.

For $s >0$, define
\eqb \label{eqn-extra-radius}
\sigma_{s,\BB r}^\ep := \inf\left\{ s' > s :   B_{R_{\BB r}^\ep(\mcl B_s^\bullet)}(\mcl B_s^\bullet) \subset \mcl B_{s'}^\bullet  \right\}   ,
\eqe  
so that $\mcl B_{\sigma_{s,\BB r}^\ep}^\bullet$ contains $B_{6\rho_{\ep \BB r}^{\lfloor\eta\log\ep^{-1}\rfloor}(z)}(z)$ for each $z\in B_{\ep \BB r}(\mcl B_s^\bullet)$. 
Since each $\rho_{\ep \BB r}^{\lfloor \eta \log \ep^{-1} \rfloor}(z)$ is a stopping time for the filtration generated by $h|_{B_{5r}(z)}$ for $r\geq \ep\BB r$, it follows that if $\tau$ is a stopping time for $\left\{ \left( \mcl B_t^\bullet , h|_{\mcl B_t^\bullet} \right) \right\}_{t\geq 0}$, then so is $\sigma_{\tau , \BB r}^\ep$ (this would still be true if we replaced 6 by 5 in~\eqref{eqn-extra-radius}). 
The following lemma, which is analogous to~\cite[Lemma 3.6]{gm-confluence}, will be used to ``kill off" the $D_h$-geodesics from 0 which hit a given boundary arc of a filled $D_h$-metric ball.

\begin{lem}
\label{lem-geo-kill-pt}
There exists $\alpha  >0$, depending only on the choice of metric, such that the following is true. 
Let $\BB r > 0$, let $\tau$ be a stopping time for the filtration generated by $\left\{ \left( \mcl B_s^\bullet , h|_{\mcl B_s^\bullet} \right) \right\}_{s \geq 0}$,
and let $x\in\bdy\mcl B_\tau^\bullet$ and $\ep\in (0,1)$ be chosen in a manner depending only on $( \mcl B_\tau^\bullet  , h|_{\mcl B_\tau^\bullet} )$. 
There is an event $ G_x^\ep \in \sigma\left(\mcl B_{  \sigma_{\tau,\BB r}^\ep}^\bullet  , h|_{B_{  \sigma_{\tau,\BB r}^\ep}^\bullet} \right)$ with the following properties.
\begin{enumerate}[A.]
\item If, in the notation~\eqref{eqn-extra-radius-eucl}, we have $ R_{\BB r}^\ep(\mcl B_\tau^\bullet) \leq   \op{diam} \mcl B_\tau^\bullet$ (where $\op{diam}$ denotes Euclidean diameter) and $G_x^\ep $ occurs, then no $D_h$-geodesic from 0 to a point in $\BB C\setminus B_{R_{\BB r}^\ep(\mcl B_\tau^\bullet)}(\mcl B_\tau^\bullet)$ can enter $B_{\ep \BB r}(x) \setminus \mcl B_\tau^\bullet$. \label{item-geo-event-kill-pt}
\item There is a deterministic constant $C_0 >1$ depending only on the choice of metric such that a.s. $\BB P\left[     G_x^\ep \,\big|\, \mcl B_\tau^\bullet  , h|_{\mcl B_\tau^\bullet} \right] \geq 1 -  C_0 \ep^\alpha$. \label{item-geo-event-prob-pt}
\end{enumerate}
\end{lem} 
\begin{proof}
The proof is similar to that of the subcritical version~\cite[Lemma 3.6]{gm-confluence}, but the geometric part of the argument (i.e., the verification of Property~\ref{item-geo-event-kill-pt}) is slightly different due to the different way in which the events $E_r (z)$ are defined in the supercritical case. 
We will therefore repeat part of the proof in order to explain the details of this geometric argument. 
See Figure~\ref{fig-geo-kill} for an illustration.    
\medskip

\begin{figure}[t!]
 \begin{center}
\includegraphics[scale=.85]{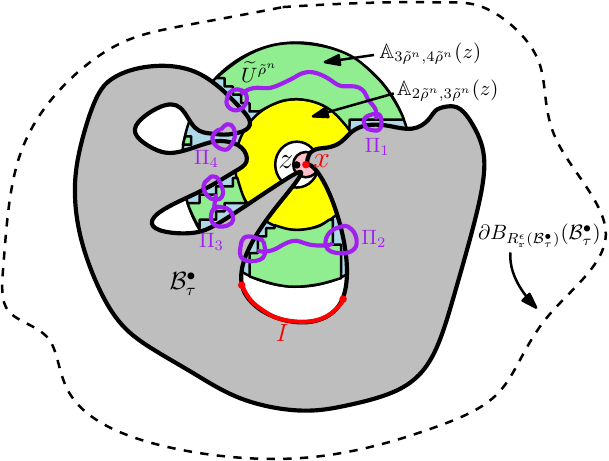}
\vspace{-0.01\textheight}
\caption{Illustration of the proof of Lemma~\ref{lem-geo-kill-pt}. The point $z \in \frac{\ep \BB r}{4} \BB Z^2$ is chosen so that $B_{\ep \BB r}(x) \subset B_{2\ep \BB r}(z)$. On the event $G_x^\ep$ defined in~\eqref{eqn-geo-event-def}, there is some $n \in [1,\eta\log\ep^{-1}]_{\BB Z}$ for which the event $H_{\wt\rho^n}^{\wt U^{\wt\rho^n}}(z)$ as defined in~\eqref{eqn-diam-small-event} occurs. 
For this choice of $n$, we can use the definition of $H_{\wt\rho^n}^{\wt U^{\wt\rho^n}}(z)$ together with condition~\ref{item-clsce-ball} in the definition of $E_{\wt\rho^n}^{\wt U^{\wt\rho^n}}(z)$ to build paths $\Pi_k$ (purple) in the connected components of $\BB A_{3\wt\rho^n,4\wt\rho^n}(z) \setminus \mcl B_\tau^\bullet$ which disconnect $\BB A_{2\wt\rho^n,3\wt\rho^n}(z)$ from $\infty$ in $\BB C\setminus\mcl B_\tau^\bullet$ and whose $D_h$-lengths are each less than $c \frk c_{\wt\rho^n} e^{\xi h_{\wt\rho^n}(z)}$. 
That is, each of the purple paths is $D_h$-shorter than the $D_h$-distance across $\BB A_{2\wt\rho^n,3\wt\rho^n}(z)$. In order for a path $P$ from a point outside of $B_{R_{\BB r}^\ep(\mcl B_s^\bullet)}(\mcl B_s^\bullet)$ to 0 to enter $B_{\ep \BB r}(x)\setminus \mcl B_\tau^\bullet$, it would first have to hit one of these purple paths, which would give us a path to 0 which is shorter than $P$. Hence such a path $P$ cannot be a $D_h$-geodesic.  
The condition that $ R_{\BB r}^\ep(\mcl B_\tau^\bullet) \leq  \op{diam} \mcl B_\tau^\bullet$ ensures that $\BB A_{3\wt\rho^n,4\wt\rho^n}(z)$ intersects $\mcl B_\tau^\bullet$. 
We can also prevent $D_h$-geodesics from hitting an arc $I$ of $\bdy\mcl B_\tau^\bullet$ by choosing $x$ so that $B_{\ep \BB r}(x)$ disconnects $I$ from $\infty$ in $\BB C\setminus\mcl B_\tau^\bullet$; see Lemma~\ref{lem-geo-kill}. 
}\label{fig-geo-kill}
\end{center}
\vspace{-1em}
\end{figure}

\noindent\textit{Step 1: setup.}
We can choose $z\in \left( \frac{\ep \BB r}{4} \BB Z^2 \right) \cap B_{\ep \BB r}\left( \mcl B_\tau^\bullet  \right)$ such that $B_{\ep \BB r}(x) \subset B_{2\ep \BB r}(z)$, in a manner depending only on $( \mcl B_\tau^\bullet  , h|_{\mcl B_\tau^\bullet} )$.  
Recalling the set of squares $\mcl S_{\delta r}^z(\cdot)$ from~\eqref{eqn-square-def}, for $r > 0$ we define
\eqb \label{eqn-ball-comp-set}
\wt U^r := \wt U^r(z) := \BB A_{3r,4r}(z) \setminus  \bigcup \left\{ S \in \mcl S_{\delta r}^z(\BB A_{3r,4r}(z)) : S \cap \mcl B_\tau^\bullet \not=\emptyset \right\} .
\eqe
Note that $\wt U^r$ belongs to the set $\mcl U_r(z)$ of Section~\ref{sec-good-annuli} and $\wt U^r$ is determined by $(\mcl B_\tau^\bullet , h|_{\mcl B_\tau^\bullet})$. 

Let $\wt\rho^0 := \ep \BB r$ and for $n\in\BB N$, inductively define
\eqb \label{eqn-good-radius-def'}
\wt\rho^n = \wt\rho_{\ep \BB r}^n(z) := \inf\left\{ r \geq 6\wt\rho^{n-1} : \text{$r = 2^k \BB r$ for some $k\in\BB Z$}, \: E_r^{\wt U^r}(z) \: \text{occurs} \right\} .
\eqe
In other words, $\wt\rho^n$ is defined in the same manner as $\rho_{\ep \BB r}^n(z)$ from~\eqref{eqn-good-radius-def} (with $\ep \BB r$ in place of $\BB r$) but with $E_r^{\wt U^r}(z)$ instead of $E_r(z)$.
This means that $E_{\wt\rho^n}^U(z)$ is only required to occur for $U = \wt U^{\wt\rho^n}$ instead of for every $U\in\mcl U_{\wt\rho^n}(z)$. 
By this and the definition~\eqref{eqn-extra-radius-eucl} of $R_{\BB r}^\ep(\mcl B_\tau^\bullet)$, 
\eqb \label{eqn-good-radii-compare}
\wt\rho^n \leq \rho_{\ep \BB r}^n(z) ,\: \forall n \in \BB N_0 \quad \text{and hence} \quad \wt\rho^{\lfloor \eta\log\ep^{-1}\rfloor} \leq \frac16 R_{\BB r}^\ep(\mcl B_\tau^\bullet).
\eqe
The reason for considering $\wt\rho^n$ instead of $\rho_{\ep \BB r}^n(z)$ is because we can only condition on $E_r^U(z)$, not on $E_r(z)$, in Lemma~\ref{lem-cond-diam-small}.
 
Recalling that $\mcl V(\wt U^{\wt\rho^n})$ denotes the set of connected components of $\wt U^{\wt\rho^n}$, we define 
\eqb \label{eqn-geo-event-def}
G_x^\ep := \left\{ \exists   n\in [1,\eta\log\ep^{-1}]_{\BB Z}  \: \text{such that $H_{\wt\rho^n}^{\wt U^{\wt\rho^n}}(z)$ occurs} \right\}  ,
\eqe
where $H_{\wt\rho^n}^{\wt U^{\wt\rho^n}}(z)$ is the event of~\eqref{eqn-diam-small-event} with $U = \wt U^{\wt\rho^n}$.    

Since $z$ and $\wt U^r$ for $r > 0$ are each determined by $(\mcl B_\tau^\bullet , h|_{\mcl B_\tau^\bullet})$, it follows that each $E^{\wt U^r}(z)$ is determined by $(\mcl B_\tau^\bullet , h|_{\mcl B_\tau^\bullet})$ and $h|_{\BB A_{2r, 5r}(z)}$. 
Hence $\wt\rho^n$ is a stopping time for the filtration generated by $h|_{B_{5r}(z)}$ for $r\geq \ep\BB r$ and $(\mcl B_\tau^\bullet , h|_{\mcl B_\tau^\bullet})$.
By~\eqref{eqn-good-radii-compare} and the definition~\eqref{eqn-extra-radius} of $\sigma_{\tau,\BB r}^\ep$, we have $B_{5\wt\rho^n}(z)  \subset \mcl B_{ \sigma_{\tau,\BB r}^\ep}^\bullet$.
By combining these statements with~\eqref{eqn-geo-event-def} and the locality of the metric (Axiom~\ref{item-metric-local}), we get that $G_x^\ep \in \sigma\left(\mcl B_{ \sigma_{\tau,\BB r}^\ep}^\bullet  , h|_{B_{  \sigma_{\tau,\BB r}^\ep}^\bullet} \right)$. 

We need to check properties~\ref{item-geo-event-kill-pt} and~\ref{item-geo-event-prob-pt} for the event $G_x^\ep$. 
\medskip

\noindent\textit{Step 2: proof that $G_x^\ep$ satisfies property~\ref{item-geo-event-kill-pt}.} 
Assume that $ R_{\BB r}^\ep(\mcl B_\tau^\bullet) \leq   \op{diam} \mcl B_\tau^\bullet$ and $G_x^\ep$ occurs. 
Choose $n\in [1,\eta\log\ep^{-1}]_{\BB Z}$ as in the definition~\eqref{eqn-geo-event-def} of $G_x^\ep$. 
Then 
\eqbn
\ep \BB r \leq \wt\rho^n \leq  \frac16  R_{\BB r}^\ep(\mcl B_\tau^\bullet)  \leq \frac16  \op{diam}\mcl B_\tau^\bullet  .
\eqen
By our choice of $z$, this means that both the inner and outer boundaries of $\BB A_{3\wt\rho^n , 4\wt\rho^n}(z)$ intersect $\mcl B_\tau^\bullet$ and $\BB A_{2\wt\rho^n , 3\wt\rho^n}(z)$ disconnects $B_{\ep \BB r}(x)$ from $\infty$. We will argue that no $D_h$-geodesic from a point outside of $\BB C\setminus B_{R_{\BB r}^\ep(\mcl B_\tau^\bullet)}(\mcl B_\tau^\bullet)$ to 0 can cross between the inner and outer boundaries of $\BB A_{2\wt\rho^n,3\wt\rho^n}(z)$ before hitting $\mcl B_\tau^\bullet$, which implies that no such $D_h$-geodesic can hit $B_{\ep\BB r}(x)$ before entering $\mcl B_\tau^\bullet$. The idea of the proof is that the definition~\eqref{eqn-diam-small-event} of $H_{\wt\rho^n}^{\wt U^{\wt\rho^n}}(z)$ together with condition~\ref{item-clsce-ball} in the definition of $E_{\wt\rho^n}^{\wt U^{\wt\rho^n}}(z)$ allow us to build a collection of paths in $\BB A_{3\wt\rho^n,4\wt\rho^n}(z)$ which act as ``shortcuts". Let us now explain the construction of these paths. 
\medskip

\noindent\textit{Step 2(a): constructing paths in $\BB A_{3\wt\rho^n,4\wt\rho^n}(z)$.}
Let $S_1,\dots,S_N$ be the path of squares in $\mcl S_{\delta\wt\rho^n}(\BB A_{3.1\wt\rho^n,3.9\wt\rho^n}(z))$ as in condition~\ref{item-clsce-ball} in the definition of $E_{\wt\rho^n}^{\wt U^{\wt\rho^n}}(z)$. Let $K$ be the number of squares in $\{S_1,\dots,S_N\}$ which intersect $\mcl B_\tau^\bullet$ (equivalently, the number of such squares which are not contained in $\wt U^{\wt\rho^n}$). For $k\in [1,K]_{\BB Z}$, let $j_k$ be the $k$th smallest value of $j\in[1,N]_{\BB Z}$ for which $S_j$ intersects $\mcl B_\tau^\bullet$. Also set $S_{j_0} =S_{j_K}$. 

For each $k\in [1,K]_{\BB Z}$ such that $S_{j_k}$ intersects $\bdy \mcl B_\tau^\bullet$, we will define a path $\Pi_k$ associated with $S_{j_k}$ in such a way that the following properties are satisfied. 
\begin{enumerate}[(i)]
\item Each $\Pi_k$ has $D_h$-length strictly less than $c \frk c_{\wt\rho^n} e^{\xi h_{\wt\rho^n}(z)}$ and intersects $\bdy \mcl B_\tau^\bullet$. \label{item-square-path-length}
\item The union of the paths $\Pi_k$ over all $k$ such that $S_k\cap \bdy\mcl B_\tau^\bullet\not=\emptyset$ disconnects $\BB A_{2\wt\rho^n,3\wt\rho^n}(z) \setminus \mcl B_\tau^\bullet$ from $\bdy B_{4\wt\rho^n}(z) \setminus \mcl B_\tau^\bullet$ in $\BB C\setminus \mcl B_\tau^\bullet$. \label{item-square-path-union}
\end{enumerate}
The paths $\Pi_k$ are shown in purple in Figure~\ref{fig-geo-kill}. 

To define these paths, let $k\in [1,K]_{\BB Z}$ such that $S_{j_k} \cap \bdy\mcl B_\tau^\bullet\not=\emptyset$. We consider two cases. 
If $j_{k-1} + 1 = j_k $, we let $\Pi_k$ be a path around $B_{2\delta \wt\rho^n}(S_{j_k}) \setminus B_{\delta \wt\rho^n} ( S_{j_k} ) $ whose $D_h$-length is at most $\frac{c}{100} \frk c_{\wt\rho^n} e^{\xi h_{\wt\rho^n}(z)}$, as afforded by condition~\ref{item-clsce-ball} in the definition of $E_{\wt\rho^n}^{\wt U^{\wt\rho^n}}(z)$.  

If $j_{k-1} +1 <  j_k$, then there is a connected component $V$ of $\wt U^{\wt\rho^n}$ whose boundary intersects the boundaries of each of $S_{j_{k-1}}$ and $S_{j_k}$ such that $S_{j_{k-1}+1} , \dots, S_{j_k-1} \subset \ol V$. Since the event $H_{\wt\rho^n}^{\wt U^{\wt\rho^n}}(z)$ of~\eqref{eqn-diam-small-event} occurs, there is a path $\pi_V$ in $V$ from the center point of $S_{j_{k-1}+1}$ to the center point of $S_{j_k-1}$ whose $D_h$-length is at most $\frac{c}{2} \frk c_{\wt\rho^n} e^{\xi h_{\wt\rho^n}(z)}$. By the last part of condition~\ref{item-clsce-ball} in the definition of $E_{\wt\rho^n}^{\wt U^{\wt\rho^n}}(z)$, there are paths $\pi_0$ and $\pi_1$ in the annular regions $B_{2\delta \wt\rho^n}(S_{j_{k-1}+1}) \setminus B_{\delta\wt\rho^n} (S_{j_{k-1}+1} ) $ and $B_{2\delta \wt\rho^n}(S_{j_k-1}) \setminus B_{\delta\wt\rho^n}( S_{j_k -1})$, respectively, which disconnect the inner and outer boundaries of these annular regions and whose $D_h$-lengths are each at most $\frac{c}{100} \frk c_{\wt\rho^n} e^{\xi h_{\wt\rho^n}(z)}$. 
Note that the paths $\pi_0$ and $\pi_1$ necessarily intersect both $\pi_V$ and $\bdy\mcl B_\tau^\bullet$. 
Let $\Pi_k$ be a concatenation of $  \pi_0, \pi_V , \pi_1$.  

It is clear from the above definitions that our desired property~\eqref{item-square-path-length} is satisfied. 
 To check property~\eqref{item-square-path-union}, consider a path $\frk P$ from a point of $\bdy B_{4\wt\rho^n}(z) \setminus \mcl B_\tau^\bullet$ to a point of $\BB A_{2\wt\rho^n,3\wt\rho^n}(z) \setminus \mcl B_\tau^\bullet$ in $\BB C\setminus \mcl B_\tau^\bullet$. There is a sub-path $\frk P'$ of $\frk P$ which is contained in $\ol{\BB A_{3\wt\rho^n,4\wt\rho^n}(z)}$ and whose endpoints lie on the inner and outer boundaries of $\BB A_{3\wt\rho^n,4\wt\rho^n}(z)$, respectively. 
Since the union of the squares $S_1,\dots,S_N$ contains a path which disconnects the inner and outer boundaries of $\BB A_{3\wt\rho^n,4\wt\rho^n}(z)$ (condition~\ref{item-clsce-ball} in the definition of $E_{\wt\rho^n}^{\wt U^{\wt\rho^n}}(z)$), there must be some $j$ such that $S_j \not\subset\mcl B_\tau^\bullet$ and $\frk P'$ intersects $S_j$. If $S_j \cap \bdy\mcl B_\tau^\bullet\not=\emptyset$, then $j = j_k$ for some $k$ and $\frk P'$ intersects the path $\Pi_k$, so we are done. Otherwise, there exists $k\in [1,K]_{\BB Z}$ for which $j \in [j_{k-1} + 1 , j_k - 1]_{\BB Z}$. Let $O$ be the connected component of $\BB A_{3\wt\rho^n , 4\wt\rho^n}(z) \setminus \mcl B_\tau^\bullet$ which contains $S_{j_{k-1}+1},\dots,S_{j_k-1}$. Then $\frk P' \subset O$. Furthermore, by construction, the path $\Pi_k$ disconnects $\bdy O \cap \bdy B_{3\wt\rho^n}(z)$ and $\bdy O\cap \bdy B_{4\wt\rho^n}(z)$ in $O$. Therefore, $\frk P'$ must intersect $\Pi_k$, as required. 
\medskip

\noindent\textit{Step 2(b): preventing a $D_h$-geodesic from crossing $\BB A_{2\wt\rho^n,3\wt\rho^n}(z)$.}
Due to Lemma~\ref{lem-outer-bdy-dist}, a $D_h$-geodesic from a point outside of $\mcl B_\tau^\bullet$ to $0$ hits $\bdy\mcl B_\tau^\bullet$ exactly once. So, if such a geodesic hits $B_{\ep \BB r}(x) \setminus \mcl B_\tau^\bullet$, then it hits $B_{\ep \BB r}(x)$ \emph{before} entering $\mcl B_\tau^\bullet$. 
Therefore, to prove property~\ref{item-geo-event-kill-pt}, it suffices to consider a path $P$ from a point outside of $\BB C\setminus B_{R_{\BB r}^\ep(\mcl B_\tau^\bullet)}(\mcl B_\tau^\bullet)$ to 0 which enters $B_{\ep \BB r}(x)$ before entering $\mcl B_\tau^\bullet$ and show that $P$ cannot be a $D_h$-geodesic.
 
 Since $\BB A_{2\wt\rho^n  , 3\wt\rho^n }(z)$ disconnects $B_{\ep \BB r}(x)$ from $\infty$, the path $P$ must cross from the outer boundary of $\BB A_{2\wt\rho^n,3\wt\rho^n}(z)$ to the inner boundary of $\BB A_{2\wt\rho^n,3\wt\rho^n}(z)$ before hitting $B_{\ep \BB r}(x)$, and hence also before hitting $\mcl B_\tau^\bullet$.  
By condition~\ref{item-clsce-across} in the definition of $E_{\wt\rho^n}^{\wt U^{\wt\rho^n}}(z)$, each path between the inner and outer boundaries of $\BB A_{2\wt\rho^n,3\wt\rho^n}(z)$ has $D_h$-length at least $c\frk c_{\wt\rho^n} e^{\xi h_{\wt\rho^n}(z)}$. Hence, the $D_h$-length of the segment of $P$ after the first time it enters $\BB A_{2\wt\rho^n , 3\wt\rho^n}(z)$ must be at least $c \frk c_{\wt\rho^n} e^{\xi h_{\wt\rho^n}(z)}   + \tau$.  

But, $P$ must cross between the inner and outer boundaries of $\BB A_{3\wt\rho^n , 4\wt\rho^n}(z)$ before entering $\BB A_{2\wt\rho^n , 3\wt\rho^n}(z)$, so $P$ must hit one of the paths $\Pi_k$ above before entering $\BB A_{2\wt\rho^n , 3\wt\rho^n}(z)$. Since $\Pi_k$ intersects $\bdy\mcl B_\tau^\bullet$ and has $D_h$-length strictly less than $c \frk c_{\wt\rho^n} e^{\xi h_{\wt\rho^n}(z)}$, it follows that each point of $\Pi_k$ lies at $D_h$-distance strictly less than $c\frk c_{\wt\rho^n} e^{\xi h_{\wt\rho^n}(z)} + \tau$ from 0. 
Combining this with the conclusion of the preceding paragraph shows that $P$ cannot be a $D_h$-geodesic to 0.  
\medskip

\noindent\textit{Step 3: proof that $G_x^\ep$ satisfies property~\ref{item-geo-event-prob-pt}.} 
Recall the definition of $G_x^\ep$ from~\eqref{eqn-geo-event-def}. 
From Lemma~\ref{lem-cond-diam-small} and an elementary conditioning argument, exactly as in the proof of~\cite[Lemma 3.6, Step 3]{gm-confluence}, we obtain that for every $n\in [1,\eta\log\ep^{-1}]_{\BB Z}$, a.s.\ 
\eqb \label{eqn-geo-event-prob-cond}
\BB P\left[ H_{\wt\rho^n}^{\wt U^{\wt\rho^n}}(z) \,|\, \wt U^{\wt\rho^n} , h|_{\BB C\setminus \wt U^{\wt\rho^n}} \right]
\geq \frk p,
\eqe 
where $\frk p > 0$ is as in Lemma~\ref{lem-cond-diam-small}. 
Note that by the definition~\eqref{eqn-good-radius-def'} of $\wt\rho^n$, it is automatically the case that the event $E_r^{\wt U^{\wt\rho^n}}(z)$ occurs. By the definition~\eqref{eqn-diam-small-event} and the locality property of $D_h$, the event $H_{\wt\rho^n}^{\wt U^{\wt \rho^n}}(z)$ is a.s.\ determined by $\wt U^{\wt\rho^n}$ and the restriction of $h$ to $\wt U^{\wt\rho^n}$. Since the open sets $\wt U^{\wt\rho^n}$ for different values of $n$ are disjoint from each other and from $\mcl B_\tau^\bullet$, we can apply~\eqref{eqn-geo-event-prob-cond} iteratively to get
\eqbn
\BB P[G_x^\ep \,|\, \mcl B_\tau^\bullet, h|_{\mcl B_\tau^\bullet}] \geq 1 - (1-\frk p)^{\lfloor \eta\log\ep^{-1}\rfloor} .
\eqen
See~\cite[Lemma 3.6, Step 3]{gm-confluence} for details. This last estimate gives Property~\ref{item-geo-event-prob-pt} for an appropriate choice of $C_0$ and $\alpha$.  
\end{proof}

Analogously to~\cite[Lemma 3.7]{gm-confluence}, we also have the following variant of Lemma~\ref{lem-geo-kill-pt} where we prevent $D_h$-geodesics from hitting a boundary arc rather than a neighborhood of a point.
 
\begin{lem} \label{lem-geo-kill}
Let $\alpha$ be as in Lemma~\ref{lem-geo-kill-pt}. 
Let $\BB r > 0$, let $\tau$ be a stopping time for the filtration generated by $\left\{ \left( \mcl B_s^\bullet , h|_{\mcl B_s^\bullet} \right) \right\}_{s \geq 0}$.
Also let $\ep \in (0,1)$ and $I\subset \bdy \mcl B_\tau^\bullet$ be an arc, each chosen in a manner depending only on $( \mcl B_\tau^\bullet  , h|_{\mcl B_\tau^\bullet} )$, such that $I$ can be disconnected from $\infty$ in $\BB C\setminus \mcl B_\tau^\bullet$ by a set of Euclidean diameter at most $\ep \BB r$.   
There is an event $G_I \in \sigma\left(\mcl B_{  \sigma_{\tau,\BB r}^\ep}^\bullet  , h|_{B_{  \sigma_{\tau,\BB r}^\ep}^\bullet} \right)$ with the following properties.
\begin{enumerate}[A.]
\item If $ R_{\BB r}^\ep(\mcl B_\tau^\bullet) \leq  \op{diam} \mcl B_\tau^\bullet$ and $G_I  $ occurs, then no $D_h$-geodesic from 0 to a point in $\BB C\setminus  \mcl B_{ \sigma_{\tau,\BB r}^\ep}^\bullet $ can pass through $I$. \label{item-geo-event-kill}
\item There is a deterministic constant $C_0 >1$ depending only on the choice of metric such that a.s. $\BB P\left[     G_I \,\big|\, \mcl B_\tau^\bullet  , h|_{\mcl B_\tau^\bullet} \right] \geq 1 -  C_0 \ep^\alpha$. \label{item-geo-event-prob}
\end{enumerate}
\end{lem} 
\begin{proof}
This follows from Lemma~\ref{lem-geo-kill-pt} via exactly the same argument as in the proof of~\cite[Lemma 3.7]{gm-confluence}. 
\end{proof}

\subsubsection{Proof of Theorem~\ref{thm-finite-geo0}}
\label{sec-finite-geo-proof}

To prove Theorem~\ref{thm-finite-geo0}, it remains to carry out Step~\ref{item-step-iterate} in the outline at the beginning of this subsection. For this step, the argument from~\cite{gm-confluence} carries over almost verbatim so we will not give details.

We first define the regularity event that we will work on. 
Fix $\BB r > 0$ and define $\tau_{\BB r}$ as in~\eqref{eqn-tau_r-def}. 
Also let $\beta >0$ be the parameter from Lemma~\ref{lem-bdy-dist0}.
For $a \in (0,1)$, we define $\mcl E_{\BB r}(a)$ to be the event that the following is true.
\begin{enumerate}
\item $B_{a \BB r}(0) \subset \mcl B_{\tau_{\BB r}}^\bullet$.  \label{item-quantum-ball-contained}
\item $ \tau_{3\BB r} - \tau_{2\BB r} \geq a \frk c_{\BB r} e^{\xi h_{\BB r}(0)}$. \label{item-quantum-ball-compare}
\item For each $s,t\in [\tau_{2\BB r} , \tau_{3\BB r} ]$ with $|s-t| \leq a \frk c_{\BB r} e^{\xi h_{\BB r}(0)}$, we have
\eqbn
\frac{1}{\BB r} \op{dist}\left( \bdy \mcl B_s^\bullet, \bdy \mcl B_t^\bullet \right) \geq \left(\frac{|t-s|}{\frk c_{\BB r} e^{\xi h_{\BB r}(0)}} \right)^{1/\beta } ,
\eqen 
where $\op{dist}$ denotes Euclidean distance.   \label{item-holder-cont}
\item In the notation~\eqref{eqn-good-radius-def}, we have $\rho_{\ep \BB r}^{\lfloor \eta \log \ep^{-1} \rfloor}(z) \leq \ep^{1/2} \BB r $ for each $ z\in \left( \frac{\ep \BB r}{4} \BB Z^2 \right) \cap B_{4 \BB r}(0) $ and each $\ep \in (0,a] \cap \{2^{-k}\}_{k\in\BB N}$ (here $\eta$ is as in Lemma~\ref{lem-clsce-all}).    \label{item-good-radii-small}
\end{enumerate}
Our above definition of $\mcl E_{\BB r}(a)$ is identical to the analogous definition in~\cite[Section 3.4]{gm-confluence} except that in~\cite{gm-confluence}, condition~\ref{item-holder-cont} is replaced by a H\"older continuity condition for $D_h$ w.r.t.\ the Euclidean metric. This condition is of course not true in the supercritical case. 

\begin{lem} \label{lem-finite-geo-reg}
For each $p\in (0,1)$, there exists $a = a(p) > 0$ such that $\BB P[\mcl E_{\BB r}(a)] \geq p$ for every $\BB r >0$.
\end{lem}
\begin{proof}
By Lemma~\ref{lem-ball-contain}, if $a$ is chosen to be sufficiently small then the probability of condition~\ref{item-quantum-ball-contained} is at least $1-(1-p)/4$.  
By tightness across scales (Axiom~\ref{item-metric-coord}), after possibly decreasing $a$ we can arrange that the probability of condition~\ref{item-quantum-ball-compare} is also at least $1-(1-p)/4$.
By Lemma~\ref{lem-bdy-dist0}, after possibly shrinking $a$ we can arrange that the probability that condition~\ref{item-quantum-ball-contained} holds is at least $1- (1-p)/4$.   
By Lemma~\ref{lem-clsce-all} and a union bound over dyadic values of $\ep$ with $\ep \in (0, a]$, the probability of condition~\ref{item-good-radii-small} is at least $1-(1-p)/4$.
Combining these estimates shows that $\BB P[\mcl E_{\BB r}(a) ] \geq p$. 
\end{proof}

The following is a more quantitative version of Theorem~\ref{thm-finite-geo0}, analogous to~\cite[Theorem 3.9]{gm-confluence}. 

\begin{thm} \label{thm-finite-geo-quant}
For each $a \in (0,1)$, there is a constant $b_0 > 0$ depending only on $a$ and constants $b_1 , \alpha > 0$ depending only on the choice of metric $D$ such that the following is true. 
For each $\BB r>0$, each $N\in\BB N$, and each stopping time $\tau$ for $\{(\mcl B_s^\bullet , h|_{\mcl B_s^\bullet})\}_{s\geq 0}$ with $\tau \in [\tau_{\BB r} ,\tau_{2\BB r}]$ a.s., the probability that $\mcl E_{\BB r}(a)$ occurs and there are more than $N$ points of $\bdy\mcl B_{\tau}^\bullet$ which are hit by leftmost $D_h$-geodesics from 0 to $\bdy\mcl B_{\tau +  N^{-\alpha} \frk c_{\BB r} e^{\xi h_{\BB r}(0)} }^\bullet$ is at most $b_0 e^{-b_1 N^\alpha}$. 
\end{thm}

It is easy to see that Theorem~\ref{thm-finite-geo-quant} implies Theorem~\ref{thm-finite-geo0}; see the beginning of~\cite[Section 3]{gm-confluence} for a proof of this in the subcritical case. The supercritical case is identical, with the caveat that we use Lemma~\ref{lem-bdy-dist0} to show that $r\mapsto \tau_r$ is continuous and surjective.

The proof of Theorem~\ref{thm-finite-geo-quant} is identical to the proof of its subcritical analog, which is given in~\cite[Section 3.4]{gm-confluence}, with one minor exception, which we discuss just below.

For the sake of completeness, we provide a short outline of the argument; see~\cite[Section 3.4]{gm-confluence} for details. We work on the event $\mcl E_{\BB r}(a)$ defined at the beginning of this subsection. Start with an arbitrary initial collection $\mcl I_0$ of arcs of $\bdy\mcl B_\tau^\bullet$ which cover $\bdy\mcl B_\tau^\bullet$ and intersect only at their endpoints. By a deterministic geometric lemma~\cite[Lemma 2.14]{gm-confluence}, if $\#\mcl I_0$ is large then at least half of the arcs in $\mcl I_0$ can be disconnected from $\infty$ in $\BB C\setminus \mcl B_\tau^\bullet$ by a set of Euclidean diameter at most a constant times $(\#\mcl I_0)^{-1/2}$. We apply Lemma~\ref{lem-geo-kill-pt} (with $\ep \asymp (\#\mcl I_0)^{-1/2}$) to each of these arcs to get that with high probability, the following is true. For at least $1/4 $ of the arcs $I\in\mcl I_0$, there is no $D_h$-geodesic from 0 to a point outside of the Euclidean $r \BB r$-neighborhood of $\mcl B_\tau^\bullet$ which passes through $I$; here $r > 0$ is related to the number $R_{\BB r}^\ep(\mcl B_\tau^\bullet)$ from Lemma~\ref{lem-geo-kill-pt} and can be bounded above by a negative power of $\#\mcl I_0$. 

We then choose a new radius $s_1 > \tau$ so that $B_{r\BB r}(\mcl B_\tau^\bullet) \subset \mcl B_{s_1}^\bullet$. By Condition~\ref{item-holder-cont} in the definition of $\mcl E_{\BB r}(a)$, we have that $s_1 - \tau$ is small if $r$ is small, and hence if $\#\mcl I_0$ is large. 
We apply the same argument with $\tau$ replaced by $s_1$ and with $\mcl I_0$ replaced by the set $\mcl I_1$ of arcs of $\bdy\mcl B_{s_1}^\bullet$ defined so that for each $I\in\mcl I_1$, all of the leftmost geodesics from 0 to points of $I$ pass through the same arc in $\mcl I_0$ (see Lemma~\ref{lem-geo-arc}). With high probability we have $\#\mcl I_1 \leq \frac34\#\mcl I_0$. We then iterate this procedure, defining radii $\tau < s_1 < s_2 \dots$. At each step we typically reduce the number of surviving arcs by a constant factor. Moreover, since the increase in the radius at each step is bounded above by a negative power of the number of surviving arcs, the total increase in the radius of the metric ball needed to get down to $N$ surviving arcs can be bounded above independently of the choice of $\mcl I_0$. To conclude, we apply this to a sequence of initial arc collections $\{\mcl I_0^k\}_{k\in\BB N}$ such that $\#\mcl I_0^k \to \infty$ and the maximal Euclidean diameter of the arcs in $\mcl I_0^k$ tends to zero.

Roughly speaking, the reason why we get the quantitative estimate $b_0 e^{-b_1 N^\alpha}$ is that for each step of the iteration, if we condition on the previous steps, there is a positive conditional probability to kill off a positive fraction of the remaining arcs; and we only need to ``succeed" for a positive fraction of the steps. See~\cite[Lemma 3.10]{gm-confluence}.

The one minor point where the argument in the supercritical case differs from the argument of~\cite[Section 3.4]{gm-confluence} is as follows. In~\cite{gm-confluence} condition~\ref{item-holder-cont} in the definition of $\mcl E_{\BB r}(a)$ is replaced by a H\"older continuity condition. However, this condition is only used once in~\cite{gm-confluence}, in the proof of~\cite[Lemma 3.11]{gm-confluence}, in order to prove that a filled $D_h$-metric ball contains a small Euclidean neighborhood of a smaller filled $D_h$-metric ball. Condition~\ref{item-holder-cont} can be used in place of the H\"older continuity condition from~\cite{gm-confluence} for this purpose. 

We note that the proof of Theorem~\ref{thm-finite-geo-quant} uses Lemmas~\ref{lem-geo-arc} and~\ref{lem-geo-kill} and also re-uses the deterministic estimate from~\cite[Lemma 2.15]{gm-confluence}.

\subsection{Reducing to a single geodesic}
\label{sec-one-geo}

In this subsection we will explain how to deduce Theorem~\ref{thm-clsce} from Theorem~\ref{thm-finite-geo0}. That is, we will explain how to go from finitely many points on the boundary of a filled metric ball which are hit by $D_h$-geodesics, to just one such point. 
The main tool which allows us to do this is Lemma~\ref{lem-pos-kill} just below, which says that for certain appropriately chosen arcs $I$ of the boundary of a filled metric ball centered at 0, there is a positive chance that \emph{every} $D_h$-geodesic from 0 to a point sufficiently far away from the filled metric ball passes through $I$. 

To state this result precisely, we need to introduce a particular way of measuring (Euclidean) distances in a planar domain. 
Let $O \subset\BB C$ be a domain bounded by a Jordan curve.  
Following~\cite[Equation (2.18)]{gm-confluence}, for $z,w\in \ol O$, we define 
\eqb \label{eqn-d^U-def} 
d^O(z,w) = \inf\left\{\op{diam}(X) : \text{$X$ is a connected subset of $O$ with $z,w\in \ol X $}\right\} ,
\eqe
where here $\op{diam}$ denotes the Euclidean diameter. Then $d^O$ is a metric on $\ol O$ which is bounded below by the Euclidean metric on $\BB C$ restricted to $\ol O$ and bounded above by the internal Euclidean metric on $\ol O$. Note that $d^O$ is not a length metric.

\begin{lem} \label{lem-pos-kill}
For each $A > 1$, $\ep \in (0,(A-1)/100)$, and $p \in (0,1)$, there exists $\frk p = \frk p(A , \ep , p) > 0$ such that the following is true. 
Let $\BB r > 0$ and let $\tau_{\BB r} = D_h(0,\bdy B_{\BB r}(0))$ be as in~\eqref{eqn-tau_r-def}. 
Let $I\subset \bdy \mcl B_{\tau_{\BB r}}^\bullet$ be a closed boundary arc, chosen in a manner depending only on $( \mcl B_{\tau_{\BB r}}^\bullet  , h|_{\mcl B_{\tau_{\BB r}}^\bullet} )$, with the property that the $d^{\BB C\setminus \mcl B_{\tau_{\BB r}}^\bullet}$-neighborhood of radius $\ep\BB r$ of $\bdy\mcl B_{\tau_{\BB r}}^\bullet \setminus I $ (with $d^{\BB C\setminus \mcl B_{\tau_{\BB r}}^\bullet}$ defined as in~\eqref{eqn-d^U-def}) does not disconnect $I$ from $\infty$ in $\BB C\setminus\mcl B_{\tau_{\BB r}}^\bullet$. 
With probability at least $p$, it holds with conditional probability at least $\frk p$ given $( \mcl B_{\tau_{\BB r}}^\bullet  , h|_{\mcl B_{\tau_{\BB r}}^\bullet} )$ that every $D_h$-geodesic from 0 to a point of $\BB C\setminus B_{A  \BB r }(0)$ passes through $I$.
\end{lem}

The proof of Lemma~\ref{lem-pos-kill} is very similar to the proof of its subcritical analog,~\cite[Lemma 4.1]{gm-confluence}. 
However, just like in the setting of Sections~\ref{sec-good-annuli} and~\ref{sec-geo-kill}, we need to make some non-trivial changes to the definitions of the events involved so we will explain most of the details of the proof. 

The proof of Lemma~\ref{lem-pos-kill} is similar to the proof of Lemma~\ref{lem-geo-kill}, but simpler since we only need something to happen with positive probability, not probability close to 1, and this probability is allowed to depend on the parameter $\ep$. We will define ``good" events $E_{\BB r}^U$ for certain domains $U$, which occur with high probability (see~\eqref{eqn-pos-kill-prob}). We will then argue that if $E_{\BB r}^U$ occurs, then there is a positive chance that the distances between certain points in $U$ are very small (Lemma~\ref{lem-cond-diam-pos}). We will then choose $U$ in a manner which depends on $\mcl B_{\tau_{\BB r}}^\bullet $ and $I$. We will use Lemma~\ref{lem-cond-diam-pos} to argue that with positive conditional probability given  $( \mcl B_{\tau_{\BB r}}^\bullet  , h|_{\mcl B_{\tau_{\BB r}}^\bullet} )$, there is a ``shortcut" in $U$ which prevents $D_h$-geodesics from 0 to points of $\BB C\setminus B_{A\BB r}(0)$ from hitting $\bdy\mcl B_{\tau_{\BB r}}^\bullet \setminus I$. 
 
To lighten notation, let
\eqb \label{eqn-diam-ball-def}
\mcl B^* := \left\{ z \in \ol{\BB C\setminus\mcl B_{\tau_{\BB r}}^\bullet} :  d^{\BB C\setminus \mcl B_{\tau_{\BB r}}^\bullet}\left(z ,  \bdy\mcl B_{\tau_{\BB r}}^\bullet \setminus I \right) <  \ep \BB r / 4   \right\}
\eqe 
be a slightly smaller $ d^{\BB C\setminus \mcl B_{\tau_{\BB r}}^\bullet}$-neighborhood of $\bdy\mcl B_{\tau_{\BB r}}^\bullet\setminus I$ than the one appearing in Lemma~\ref{lem-pos-kill}. 

By the H\"older continuity of the Euclidean metric w.r.t.\ $D_h$ (Proposition~\ref{prop-holder}), we can find $c = c(A,\ep , p)  >0$ such that with probability at least $1 - (1-p)/3$, each subset of $B_{A \BB r}(0)$ with Euclidean diameter at least $\ep \BB r /4$ has $D_h$-diameter at least $c \frk c_{\BB r} e^{\xi h_{\BB r}(0)}$. By the definition~\eqref{eqn-diam-ball-def} of $\mcl B^*$, each path in $\BB C\setminus \mcl B_{\tau_{\BB r}}^\bullet$ from $\bdy\mcl B_{\tau_{\BB r}}^\bullet \setminus I$ to a point of $\BB C\setminus (\mcl B_{\tau_{\BB r}}^\bullet \cup \mcl B^*)$ has Euclidean diameter at least $\ep \BB r /4$. 
Hence, with probability at least $1-(1-p)/3$, 
\eqb \label{eqn-pos-kill-across}
D_h\left(   \bdy\mcl B_{\tau_{\BB r}}^\bullet \setminus I , \BB C\setminus (\mcl B_{\tau_{\BB r}}^\bullet \cup \mcl B^*)  ; \BB C\setminus\mcl B_{\tau_{\BB r}}^\bullet \right) \geq c \frk c_{\BB r} e^{\xi h_{\BB r}(0)}    .
\eqe

Define the collection of $\delta \BB r\times \delta \BB r$ squares $\mcl S_{\delta \BB r}(B_{A\BB r}(0)) = \mcl S_{\delta \BB r}^0(B_{A\BB r}(0))$ with corners in $\delta \BB r\BB Z^2$ as in~\eqref{eqn-square-def} with $z=0$. 
By Lemma~\ref{lem-set-tightness}, the random variables $\frk c_{\BB r}^{-1} e^{-\xi h_{\BB r}(0)} \tau_{\BB r}$ and their reciprocals are tight. By combining this with Corollary~\ref{cor-hit-ball'}, we can find $\delta =\delta(c,A,\ep) \in (0,\ep^2/100) $ such that with probability at least $1 - (1-p)/3$, 
\eqb \label{eqn-pos-kill-ball}
D_h\left(\text{around $B_{\delta^{1/2}\BB r}(S) \setminus B_{\delta\BB r}(S)$}\right) \leq \frac{c}{100} \frk c_{\BB r} e^{\xi h_{\BB r}(0)} ,\quad\forall S \in \mcl S_{\delta \BB r}\left(\bdy\mcl B_{\tau_{\BB r}}^\bullet \right) .
\eqe

Let $\mcl U_{\BB r} $ be the (finite) set of sub-domains $U$ of $B_{ A \BB r}(0)$ such that $B_{ A \BB r}(0)  \setminus U$ is a finite union of sets of the form $S \cap B_{ A \BB r}(0)$ for $S \in \mcl S_{\delta \BB r}(B_{ A \BB r}(0))$. 
For $U\in\mcl U_r $, let $\frk h^U$ be the harmonic part of $h|_U$. Also let $U_{\delta \BB r/4}$ be the set of points in $U$ which lie at Euclidean distance at least $\delta \BB r/4$ from $\bdy U$. 
Since there are only finitely many sets in $\mcl U_{\BB r} $ and by the translation and scale invariance of the law of $h$, modulo additive constant, we can find $C = C(\delta , A , \ep ) > 0$ such that with probability at least $1 - (1-p)/3$, it holds simultaneously for each $U\in\mcl U_{\BB r}$ that
\eqb \label{eqn-pos-kill-harmonic}
\sup_{u   \in U_{\delta  \BB r  / 4}} |\frk h^U(u) - h_{\BB r}(0) |  \leq C    .
\eqe

For a given choice of $U\in \mcl U_{\BB r}$, let $E_{\BB r}^U$ be the event that~\eqref{eqn-pos-kill-across}, \eqref{eqn-pos-kill-ball}, and~\eqref{eqn-pos-kill-harmonic} all hold, so that 
\eqb \label{eqn-pos-kill-prob}
\BB P\left[ \bigcap_{U\in\mcl U_{\BB r}} E_{\BB r}^U \right] \geq p . 
\eqe
The reason for considering $E_{\BB r}^U$ instead of $\bigcap_{U\in\mcl U_{\BB r}} E_{\BB r}^U$ is the same as in Section~\ref{sec-good-annuli}: it is easier to condition on $E_{\BB r}^U$ than on $\bigcap_{U\in\mcl U_{\BB r}} E_{\BB r}^U$ (see Lemma~\ref{lem-cond-diam-pos} just below).

We note that the definition of $E_{\BB r}^U$ given just above is identical to the definition of the analogous event in~\cite[Section 4.1]{gm-confluence}, except that in~\cite{gm-confluence} the condition~\eqref{eqn-pos-kill-ball} is replaced by an upper bound for the $D_h$-diameters of the squares $S\in \mcl S_{\delta\BB r}(B_{A\BB r}(0))$. Such an upper bound does not hold in the supercritical case. 

For $U\in \mcl U_{\BB r}$, let
\eqb \label{eqn-cond-diam-set}
\mcl Z(U) := \left\{\text{centers of squares $S \in \mcl S_{\delta\BB r}(U)$ with $S\subset \ol U$}\right\}
\eqe
and let $H_{\BB r}^U$ to be the event that for each $z\in \mcl Z(U)$, there is a path $\Pi = \Pi_z$ in $U$ which disconnects 0 from $\infty$ and hits $z$ and which has $D_h$-length at most $\frac{c}{2} \frk c_{\BB r} e^{\xi h_{\BB r}(0)}$.  
Note that for some choices of $U \in \mcl U_{\BB r}$ and $z\in \mcl Z(U)$, there is \emph{no} path in $U$ which disconnects 0 from $\infty$ and hits $z$. For such a choice of $U$ we have $\BB P[H_{\BB r}^U] = 0$. The following lemma will play an analogous role to Lemma~\ref{lem-cond-diam-small} from Section~\ref{sec-good-annuli}. 

\begin{lem} \label{lem-cond-diam-pos}
There is a constant $\frk p = \frk p(A,\ep , p) > 0$ such that the following is true.   
Suppose $U \in \mcl U_{\BB r}$ is connected and contains a path which disconnects 0 from $\infty$.  
On the event that $U\cap \left( \mcl B_{\tau_{\BB r}}^\bullet \cup \mcl B^* \right) = \emptyset$, a.s.\  
\eqb  \label{eqn-cond-diam-pos'} 
\BB P\left[  H_{\BB r}^U   \,\big|\,h|_{\BB C\setminus U} , E_{\BB r}^U  \right] \geq \frk p .
\eqe  
\end{lem} 
\begin{proof}
This follows from the Markov property of the GFF and the FKG inequality (Proposition~\ref{prop-fkg-metric}), via exactly the same argument as in the proof of~\cite[Lemma 3.3 or Lemma 4.2]{gm-confluence}. 
\end{proof}

\begin{proof}[Proof of Lemma~\ref{lem-pos-kill}]
Most of the proof is exactly the same as the the proof of~\cite[Lemma 4.1]{gm-confluence}, but the geometric part of the argument is slightly different so we will repeat part of the argument to explain the differences. See Figure~\ref{fig-one-geo} for an illustration of the proof. 
\medskip

\begin{figure}[t!]
 \begin{center}
\includegraphics[scale=.8]{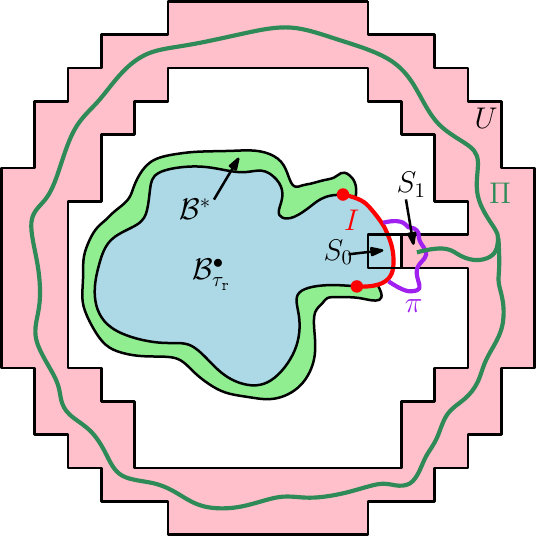}
\vspace{-0.01\textheight}
\caption{Illustration of the proof of Lemma~\ref{lem-pos-kill}. If $E_{\BB r}^U \cap H_{\BB r}^U$ occurs for the domain $U\in\mcl U_{\BB r}$ defined in the proof, then by the definition of $H_{\BB r}^U$ we can find a path $\Pi$ in $U$ which contains the center of the square $S_1$, disconnects $\mcl B_{\tau_r}^\bullet \cup \mcl B^*$ from $\infty$, and whose $D_h$-length is small. Furthermore, we can find a path $\pi \subset (\BB C\setminus \mcl B_{\tau_{\BB r}}^\bullet \cup \mcl B_*)$ --- a segment of the path around $B_{2\delta\BB r}(S_0) \setminus B_{\delta\BB r}(S_0)$ given by~\eqref{eqn-pos-kill-ball} --- which intersects both $I$ and $\Pi$. By~\eqref{eqn-pos-kill-across}, the sum of the $D_h$-lengths of $\pi$ and $\Pi$ is smaller than the $D_h$-distance from $\BB C\setminus (\mcl B_{\tau_{\BB r}}^\bullet \cup \mcl B^*)$ to $\bdy\mcl B_{\tau_{\BB r}}^\bullet \setminus I$ restricted to paths which do not enter $\mcl B_{\tau_{\BB r}}^\bullet$. This prevents a $D_h$-geodesic from 0 to a point outside of $B_{A\BB r}(0)$ from hitting $\mcl B_{\tau_{\BB r}}^\bullet \setminus I$. 
}\label{fig-one-geo}
\end{center}
\vspace{-1em}
\end{figure} 

\noindent\textit{Step 1: choosing a random domain $U$.}
We first choose the domain $U$ to which we will apply Lemma~\ref{lem-cond-diam-pos}. 
The choice will depend on $\mcl B_{\tau_{\BB r}}^\bullet$ and $I$, which is why we need a lower bound for the probability of the intersection of all of the $E_{\BB r}^U$'s in~\eqref{eqn-pos-kill-prob}.

Since $\mcl B_{\tau_{\BB r}}^\bullet\subset \ol{B_{\BB r}(0)}$ and $\ep < (A-1)/100$, we have $\mcl B_{\tau_{\BB r}}^\bullet \cup \mcl B^* \subset B_{(A+1)\BB r/2}(0)$. 
By hypothesis, the $d^{\BB C\setminus \mcl B_{\tau_{\BB r}^\bullet}}$-neighborhood of $\bdy \mcl B_{\tau_{\BB r}}^\bullet \setminus I$ of radius $\ep\BB r$ does not disconnect $I$ from $\infty$ in $\BB C\setminus \mcl B_{\tau_{\BB r}}^\bullet$.
Hence we can choose, in a manner depending only on $\mcl B_{\tau_{\BB r}}^\bullet$ and $I$, a path $\frk P$ in $B_{A\BB r}(0) \setminus  \mcl B_{\tau_{\BB r}}^\bullet  $ from a point of $I$ to a point of $B_{(A+1)\BB r/2}(0)$ such that each point of $\frk P$ lies at $d^{\BB C \setminus \mcl B_{\tau_{\BB R}}^\bullet}$-distance at least $\ep\BB r  $ from $\bdy \mcl B_{\tau_{\BB r}}^\bullet \setminus I$. By slightly perturbing $\frk P$ if necessary, we can assume that $\frk P$ does not hit any of the corners of any of the squares in $\mcl S_{\delta\BB r}(B_{A\BB r}(0))$. 

Let $\wt U$ be the interior of the union of all of the $\delta \BB r\times \delta \BB r$ squares $S\in \mcl S_{\delta \BB r}(B_{A \BB r}(0))$ which intersect $\frk P \cup \BB A_{(A+1)\BB r/2 , A\BB r}(0)$ but do not intersect $\mcl B_{\tau_{\BB r}}^\bullet$ or $\bdy B_{A\BB r}(0)$. 
Let $U$ be the connected component of $\wt U$ which intersects $ \BB A_{(A+1)\BB r/2 , A\BB r}(0)$. 
Then $U\in\mcl U_{\BB r}$, as defined just above~\eqref{eqn-pos-kill-harmonic}. 
  
By definition, $U\cap \mcl B_{\tau_{\BB r}}^\bullet=\emptyset$. We claim that also $U\cap \mcl B^* =\emptyset$. Indeed, each of the $\delta \BB r\times\delta \BB r$ squares $S$ in the union defining $U$ is contained in $\BB C\setminus \mcl B_{\tau_{\BB r}}^\bullet$ and has Euclidean diameter at most $\sqrt 2 \delta \BB r < \ep \BB r/4 $. If one of these squares intersected $\mcl B^*$, then by the triangle inequality and the definition~\ref{eqn-d^U-def} of $d^{\BB C\setminus\mcl B_{\tau_{\BB r}}^\bullet}$, the $d^{\BB C\setminus\mcl B_{\tau_{\BB r}}^\bullet}$-distance from $\frk P$ to $\bdy\mcl B_{\tau_{\BB r}}^\bullet\setminus I$ would be at most $\ep \BB r/2$, contrary to the definition of $\frk P$. 

Hence $U\cap \left( \mcl B_{\tau_{\BB r}}^\bullet \cup \mcl B^* \right) = \emptyset$. 
Since $\mcl B_{\tau_{\BB r}}^\bullet$ is a local set for $h$ (Lemma~\ref{lem-ball-local}) and $\frk P$ is determined by $(\mcl B_{\tau_{\BB r}}^\bullet, h|_{\mcl B_{\tau_{\BB r}}^\bullet})$, for each deterministic $\frk U \in \mcl U_{\BB r}$, the event $\{U = \frk U\}$ is determined by $h|_{\BB C\setminus \frk U}$. 
Furthermore, by definition the set $U$ is connected and contains a path which disconnects 0 from $\infty$. 
Therefore, the bound~\eqref{eqn-cond-diam-pos'} of Lemma~\ref{lem-cond-diam-pos} holds a.s.\ for our (random) choice of $U$. 
\medskip

\noindent\textit{Step 2: bounding conditional probabilities.} By~\eqref{eqn-pos-kill-prob} and Markov's inequality
\eqb \label{eqn-use-pos-kill-prob}
\BB P\left[ \BB P\left[ E_{\BB r}^U \,\big|\, \mcl B_{\tau_{\BB r}}^\bullet, h|_{\mcl B_{\tau_{\BB r}}^\bullet} \right] \geq 1-(1-p)^{1/2} \right] \geq 1- (1-p)^{1/2}  .
\eqe  
By this together with the bound~\eqref{eqn-cond-diam-pos'}, and since $p$ can be made arbitrarily close to 1, to conclude the proof of the lemma we only need to show that if $E_{\BB r}^U \cap H_{\BB r}^U$ occurs, then every $D_h$-geodesic from 0 to a point of $\BB C\setminus B_{A \BB r}(0)$ passes through $I$.
This will be accomplished via a similar argument to the proof of Lemma~\ref{lem-geo-kill-pt}, as we now explain.
\medskip

\noindent\textit{Step 3: preventing $D_h$-geodesics from hitting $\bdy\mcl B_{\tau_{\BB r}}^\bullet\setminus I$.}
Let $S_1$ be the first square in $ \mcl S_{\delta \BB r}(B_{A \BB r}(0))$ hit by $\frk P$ whose interior is contained in $U$. 
Since $\frk P$ starts from a point of $I \subset\bdy\mcl B_{\tau_{\BB r}}^\bullet$ and $U\cap \bdy\mcl B_{\tau_{\BB r}}^\bullet = \emptyset$, $S_1$ is not the first square of $\mcl S_{\delta \BB r}(B_{A \BB r}(0))$ hit by $\frk P$. Hence, there is a square $S_0$ which is hit by $\frk P$ prior to the first time $\frk P$ hits $S_1$ such that $S_0$ and $S_1$ share a side (here we use that $\frk P$ does not hit any of the four corners of $S_1$). By the definition of $S_1$, we have $S_0\cap \bdy\mcl B_{\tau_{\BB r}}^\bullet \not=\emptyset$. 

By~\eqref{eqn-pos-kill-ball}, there is a path $\wt\pi$ in the annular region $B_{\delta^{1/2}\BB r}(S_0) \setminus B_{\delta\BB r}(S_0)$ which disconnects the inner and outer boundaries of this annular region and has $D_h$-length at most $\frac{c}{100} \frk c_{\BB r} e^{\xi h_{\BB r}(0)}$. Since $S_1\subset B_{\delta^{1/2}\BB r}(S_0) \setminus B_{\delta\BB r}(S_0)$ and $S_1$ is disjoint from $\mcl B_{\tau_{\BB r}}^\bullet$, there is a sub-path $\pi$ of $\wt\pi$ which is contained in $\ol{\BB C\setminus \mcl B_{\tau_{\BB r}}^\bullet}$ and which disconnects $S_1$ from $\infty$ in $\ol{\BB C\setminus \mcl B_{\tau_{\BB r}}^\bullet}$. Since $S_0$ intersects $\pi \cap I$, which is disjoint from $\mcl B_*$, and $\delta^{1/2} \leq \ep/10$, the definition~\eqref{eqn-d^U-def} of $d^{\BB C\setminus \mcl B_{\tau_{\BB r}}^\bullet}$ shows that the path $\pi$ cannot intersect $\bdy\mcl B_{\tau_{\BB r}}^\bullet \setminus I$. Hence $\pi$ must intersect $I$. 

By the definition of $H_{\BB r}^U$ (just below~\eqref{eqn-cond-diam-set}), there is a path $\Pi$ in $U$ which contains a point of $S_1$ and which disconnects 0 from $\infty$ (and hence also $\mcl B_{\tau_{\BB r}}^\bullet$ from $\BB C\setminus B_{A\BB r}(0)$. The union of $\pi$ and $\Pi$ is connected, has $D_h$-length strictly less than $c\frk c_{\BB r} e^{\xi h_{\BB r}(0)}$, intersects $I$, and disconnects $\mcl B_{\tau_{\BB r}}^\bullet$ from $\BB C\setminus B_{A\BB r}(0)$. 

Any path $P$ from a point of $B_{A\BB r}(0)$ to 0 which first hits $\bdy\mcl B_{\tau_{\BB r}}^\bullet$ at a point not in $I$ must hit $ \Pi$ and then must subsequently cross from a point of $\BB C\setminus (\mcl B_*\cup \mcl B_{\tau_{\BB r}}^\bullet)$ to $\bdy\mcl B_{\tau_{\BB r}}^\bullet \setminus I$. 
By the preceding paragraph, the $D_h$-distance from the first point of $\Pi$ hit by $P$ to 0 is strictly smaller than $\tau_{\BB r} + c \frk c_{\BB r} e^{\xi h_{\BB r}(0)} $. On the other hand,~\eqref{eqn-pos-kill-across} shows that the $D_h$-length of the segment of $P$ which crosses from $\BB C\setminus (\mcl B_*\cup \mcl B_{\tau_{\BB r}}^\bullet)$ to $\bdy\mcl B_{\tau_{\BB r}}^\bullet \setminus I$ is at least $ c \frk c_{\BB r} e^{\xi h_{\BB r}(0)} $, so the $D_h$-length of the segment of $P$ after it first hits $U$ is at least $\tau_{\BB r} +  c \frk c_{\BB r} e^{\xi h_{\BB r}(0)} $. 
Therefore, $P$ cannot be a $D_h$-geodesic. 
\end{proof}

Now that Lemma~\ref{lem-pos-kill} is established, we can conclude the proof of Theorem~\ref{thm-clsce} in exactly the same way as in the subcritical case (see~\cite[Section 4.2]{gm-confluence}). The proof of~\cite[Lemma 4.3]{gm-confluence} requires some containment relations between filled $D_h$-metric balls and Euclidean balls, but these are easily supplied by Lemma~\ref{lem-ball-contain} and Proposition~\ref{prop-holder}.

\bibliography{cibib}
\bibliographystyle{hmralphaabbrv}

\end{document}